\documentclass{amsart}

\usepackage{lipsum,subeqnarray}
\usepackage{amsfonts}
\usepackage{graphicx}
\usepackage{amsmath}
\usepackage{mathtools}
\usepackage{amssymb}
\usepackage{enumerate}
\usepackage{slashed}
\usepackage{graphicx}
\usepackage{newlfont}
\usepackage{amsrefs}
\usepackage{comment}
\usepackage[normalem]{ulem}
\usepackage{mathtools}
\usepackage{accents}
\usepackage{leftidx}
\usepackage{bbm}

\numberwithin{equation}{section}

\theoremstyle{plain}
\newtheorem{theorem}{Theorem}[section]

\newtheorem{cor}{Corollary}

\newtheorem{proposition}[theorem]{Proposition}
\usepackage{xcolor}

\theoremstyle{definition}

\theoremstyle{remark}
\newtheorem{remark}{Remark}
\usepackage{stackengine}
\usepackage{subeqnarray} 

\stackMath
\newcommand\tenq[2][1]{%
 \def\useanchorwidth{T}%
  \ifnum#1>1%
    \stackunder[0pt]{\tenq[\numexpr#1-1\relax]{#2}}{\scriptscriptstyle\sim}%
  \else%
    \stackunder[1pt]{#2}{\scriptscriptstyle\sim}%
  \fi%
}

\newcommand{\overbar}[1]{\mkern 1.6mu\overline{\mkern-1.6mu#1\mkern-1.6mu}\mkern 1.6mu}

\DeclareMathOperator*{\esssup}{ess\,sup}

\begin{document}
\title[Existence and Uniqueness of Solutions to a Regularized Oldroyd-B Model]{Existence and Stability of Global Solutions to a regularized Oldroyd-B Model in its Vorticity Formulation}
 
\author{Jaroslaw S. Jaracz and Young Ju Lee} 
\begin{abstract}
We present a new regularized Oldroyd-B model in three dimensions which satisfies an energy estimate analogous to that of the standard model, and maintains the positive semi-definiteness of the conformation tensor. This results in the unique existence and stability of global solutions in a periodic domain. To be precise, given an initial velocity $u_0$ and initial conformation tensor $\sigma_0$, both with components in $H^2$, we obtain a velocity $u$ and conformation tensor $\sigma$ both with components in $C([0, T]; H^2)$ for all $T>0$. Assuming better regularity for the initial data allows us to obtain better regularity for the solutions. We treat both the diffusive and non-diffusive cases of the model. Notably, the regularization in the equation for the conformation tensor in our new model has been applied only to the velocity, rather than to the conformation tensor, unlike other available regularization techniques \cite{barrett;suli2009}. This is desired since the stress, and thus the conformation tensor, is typically less regular than the velocity for the creeping flow of non-Newtonian fluids. In \cite{constantinnote} the existence and regularity of solutions to the non-regularized two dimensional diffusive Oldroyd-B model was established. However, the proof cannot be generalized to three dimensions nor to the non-diffusive case. The proposed regularization overcomes these obstacles. Moreover, we show that the solutions in the diffusive case are stable in the $H^2$ norm. In the non-diffusive case, we are able to establish that the solutions are stable in the $L^2$ norm. Furthermore, we show that as the diffusivity parameter goes to zero, our solutions converge in the $L^2$ norm to the non-diffusive solution.  
\end{abstract}
\maketitle

\section{Introduction}\label{SEC:Intro}

The classical incompressible Oldroyd-B model is given by the following system of equations:  
\begin{equation}\label{OldroydB} 
\left \{ \begin{array}{l}
\partial_t u + \left(u \cdot \nabla \right) u - \frac{\nu}{R_e} \Delta u = \frac{1}{R_e}\text{div} \sigma -\frac{1}{R_e}\nabla p+ \frac{1}{R_e} \mathcal{F}, \\
\partial_t \sigma - \epsilon \Delta \sigma + \left(u \cdot \nabla \right) \sigma - (\nabla u)\sigma - \sigma (\nabla u)^T + \frac{1}{W_i}\sigma  = \frac{1 - \nu}{W_i^2} I,  \\ 
\nabla \cdot u = 0, 
\end{array}\right. 
\end{equation}
where $u = (u^i)_{1\leq i \leq d}$ is the velocity, $\sigma = (\sigma_{ij})_{1 \leq i,j \leq d}$ is the conformation tensor, $p$ is the pressure, and $\mathcal{F}$ is an applied external force. Here $d$ is the dimension. The stress $\tau$ is related to the conformation tensor by the equation $\tau = \sigma - \frac{1-\nu}{W_i} I$ where $I$ is the identity matrix. We will use raised indices for velocities, and lowered indices for conformation tensors. The parameters $\nu, R_e, W_i,$ and $\epsilon$ are physical constants satisfying $\nu, R_e, W_i>0$ and $\epsilon \geq 0$. The constant $W_i$ is called the Weissenberg number, and its significance in the context of numerical simulation will be discussed later, while the constant $\epsilon$ is the diffusivity parameter. Usually, the case where $\epsilon = 0$ is simply referred to as the Oldroyd-B model, as historically that was the model which was first introduced. However, for clarity, we will refer to it as the non-diffusive Oldroyd-B model to make it absolutely clear we are discussing the case where $\epsilon = 0$. The case $\epsilon > 0$ is referred to as the diffusive Oldroyd-B model. The Oldroyd-B model is sometimes presented in terms of the stress, rather than the conformation tensor. However, it is preferable to work with the conformation tensor because it has certain positivity properties which are lacked by the stress, as will be discussed later.

One way of studying the system of equations \eqref{OldroydB} is to look at what is known as the vorcitiy formulation. Namely, we take the curl for the momentum equation in \eqref{OldroydB} and obtain the vorticity formulation (see Appendix \ref{VorticityDerivation} for more details):  
\begin{equation}\label{OldroydBVorticityForm} 
\left \{ \begin{array}{l}
\partial_t \omega - \frac{\nu}{R_e} \Delta \omega + (u \cdot \nabla ) \omega - (\omega \cdot \nabla) u = \frac{1}{R_e} \nabla \times \text{div} \sigma + g \\
\partial_t \sigma- \epsilon \Delta \sigma + \left( u \cdot \nabla \right) \sigma - (\nabla u) \sigma - \sigma (\nabla u)^T + \frac{1}{W_i}\sigma = \frac{1 - \nu}{W_i^2} I  \\ 
\nabla \cdot u = 0, \quad \mbox{ and } \quad \nabla \times u = \omega, 
\end{array} \right. 
\end{equation}
where $g = \nabla \times \left( \frac{1}{R_e} \mathcal{F} \right )$.

There are several reasons why this is an attractive approach. Notably, it eliminates the pressure from consideration, and since the velocity can be recovered from the vorticity, nothing is lost. In fact, eliminating the pressure is useful when performing numerical analysis on such a system as it removes one more object to worry about, and recovering the velocity is accomplished by basically solving Laplace's equation, for which efficient numerical schemes are available. On a more fundamental level, oftentimes it is the vorticity itself that one is interested in studying, as it is closely related to turbulence, which remains a poorly understood phenomenon \cite{thomases2007emergence}. Thus, numerically one often prefers working in the vorticity formulation.

The Oldroyd-B model is a generic viscoelastic model that has its origin in continuum mechanics \cite{Oldroyd1950,Oldroyd1958,Oliveira2009}, but it can be derived from a microscopic bead and spring model as well \cite{RGOWENS_TNPHILLIPS_2002}. As clearly discussed in \cite{JWBarrett;ESuli2007}, a careful look at the microscopic derivations \cite{RGOWENS_TNPHILLIPS_2002,Li;Zhang2007} of the Oldroyd-B model provides room where the model can be improved, especially in terms of stability. There, a modification using the Brownian force for the beads to derive the motion of the center of mass of the beads can introduce two regularizing terms. In this particular regularized model, there are two stabilizing terms, (1) the diffusion term and (2) the regularized velocity. This regularized model, based on a
careful microscopic derivation, is proven to be well-posed
\cite{Barrett;Schwab;Suli2004,JWBarrett;ESuli2007,JWBarrett;ESuli2008a},
even without the diffusion when $\alpha$ is non-zero. We will introduce a different regularization.

We remark that the non-diffusive Oldroyd-B model at relatively large Weissenberg number has for decades been elusive for numerical rheologists \cite{Hulsen2005,Fattal2004}. While progress has been made towards resolving the issues \cite{Hulsen2005, Wapperom.Renardy2005}, beyond a certain critical Weissenberg number all numerical methods fail to obtain reliable solutions. A steady state local analysis of the behavior of the stress in the case of fluid passing by a cylinder reveals that there is a singularity located in the resulting wake \cite{bajaj2008coil}. Such singularities appear to be the source of the difficulty in simulating highly elastic flows, and this suggests that the Oldroyd-B model may have a problem at the continuous level at high Weissenberg number \cite{thomases2007emergence}. The primary way of dealing with this issue was to introduce the diffusive Oldroyd-B model, as opposed to tackling the issue of the high Weissenberg number directly. The diffusive term $\epsilon \Delta \sigma$ of course gives better regularity to the conformation tensor, and hence to the stress. Thus one can think of the $\epsilon > 0$ case as somewhat artificial, and it is really the $\epsilon=0$ case that is of major interest. 

In this paper, we present a new regularized Oldroyd-B model which we treat in the periodic domain. Our new model satisfies an energy estimate very much like the one satisfied by the Oldroyd-B model. In addition, we show that if we begin with a positive semi-definite conformation tensor, it remains positive semi-definite. Combined, these results provide apriori estimates for the solutions of our regularized system, which we can then use to establish the existence of unique global solutions for any $\epsilon \geq 0$. Readers can refer to \cite{liu2008global} for the global existence of classical solution to the original Oldroyd-B model in a periodic domain, limited to the case of small initial data and $\mathcal{F} = 0$, where $\mathcal{F}$ is a forcing function. Our result has no such limitations. Other relevant work can be found in \cite{barrett2011existence}, where a different regularized Oldroyd-B model is introduced as well as the global existence of its weak solution discussed. Technically, the regularizer is applied to the stress and the existence proof does not use the positivity of the conformation tensor, but instead uses a so-called the free-energy estimate.

We briefly comment on why we believe our model is better than other competing ones. The key is that our model has global solutions, in both the diffusive and non-diffusive cases, can handle dimension $d=3$, and is formulated in the periodic setting. This is in contrast to the result in \cite{constantinnote}, where, even though the result is given for the non-regularized model, the result only holds for dimension $d=2$ and the diffusive case $\epsilon>0$. Furthermore, the result presented there is in an unbounded domain, which is not particularly useful for numerical simulation. The proof relies on Ladyzhenskaya's identity in deriving certain energy estimates, and the exponents present in that identity do not work out properly in $d=3$, meaning the proof cannot be generalized. See equation (45) in \cite{constantinnote}. Also the non-diffusive case $\epsilon=0$ is not treated. On the other hand, unlike the regularized model presented in \cite{barrett2011existence, Barrett;Schwab;Suli2004,JWBarrett;ESuli2007,JWBarrett;ESuli2008a}, our model does not regularize the conformation tensor, and thus the stress, in the constitutive equation for the conformation tensor. It is only the velocity which is regularized in that equation. This is preferable as experimentally the stress is typically less regular than the velocity for the creeping flow of non-Newtonian fluids, and thus while regularizing the velocity is acceptable from an empirical point of view, the regularization of the stress or conformation tensor should be kept to a minimum. However, we still have to regularize the conformation tensor in the velocity equation, for technical reasons. As mentioned, seeing that the velocity is experimentally well behaved, this should not have any negative effect on the practical applicability of our model.

In our proofs, the diffusive and non-diffusive cases have been treated separately. This is because in the diffusive case, the presence of the $\epsilon \Delta \sigma$ term for $\epsilon >0 $ allows us to apply parabolic techniques to solve for the conformation tensor, which are not available for the case $\epsilon = 0$. In both cases, we obtain the same levels of regularity. In particular, starting with initial data $(u_0, \sigma_0)$ with components in the Sobolev space $H^2$, we obtain that our solution $(u, \sigma)$ has components in $C([0, T]; H^2)$ (see \S \ref{notation}). One may think that the result for the non-diffusive case can be obtained from that of diffusive case by taking the limit. However, this is difficult to show. One would expect that starting from the same initial data, in the limit as $\epsilon \rightarrow 0$ the solutions to the diffusive system, should converge to the non-diffusive solution, component wise in $H^2$. While this proved difficult to show, we were able show this convergence in the $L^2$ norm. See also \cite{Linliuzhang2005} for a similar result. Also see \cite{Berti2008,CYDLu_PDOlmsted_RCBall_2000}. 

Our solutions also exhibit stability, meaning that given two sets of initial data and the corresponding solutions, the norm of the difference of the solutions can be estimated in terms of the norm of the difference of the initial data. In the diffusive case, this stability is in terms of the $H^2$ norm of the components. In the non-diffusive case establishing the $H^2$ stability is more difficult, and so we only establish the stability in terms of the $L^2$ norm.

\subsection{Notation}\label{notation} 

We will be dealing with functions of both space and time. We will primarily focus on the case where the dimension of our system is $d=3$. Thus we denote the spatial variable by $x=(x^1, x^2, x^3)$ using raised indices. We denote our velocity by $u=(u^1, u^2, u^3)$ again with raised indices. The conformation tensor is given by $\sigma=(\sigma_{ij})_{1\leq i, j \leq 3}$ with lowered indices. Sometimes when convenient, we will suppress one (usually the space) or both variables. Thus for example we would write $f(x, t)=f(t)=f$. Derivatives with respect to time will be denoted by $\partial_t, \partial_t^2$, etc., while derivatives with respect to space will be denoted by $\partial_{x^i}=\partial_i$ or $D^\lambda$ where $\lambda$ is a multi-index. As usual,  we use $\nabla=(\partial_{1}, \partial_{2}, \partial_{3})$ for the gradient operator, and $\Delta$ for the Laplacian.

Since the domain is given as periodic, we shall assume without loss of generality that all of our functions are $(1, 1, 1)$ periodic, though a different periodicity could be taken. Thus we can identify the cube $[1, 1]^3 \subset \mathbb{R}^3$ with the three dimensional torus, which we denote by $\mathbb{T}$. Thus we will be able to think of all of our functions as functions on $\mathbb{T}$ which is a compact three dimensional Riemannian manifold without boundary with the flat metric induced from $\mathbb{R}^3$. We will switch between these two different points of view whenever convenient. Most of the time working on the trous provides a more convenient viewpoint. Thus, we shall use the torus, except where explicitly stated. The symbol $|\mathbb{T}|$ is designated for the volume of $\mathbb{T}$.   

As usual, the symbol $W^{m, p}(\mathbb{T})$ denotes 
the Sobolev space consisting of (real valued) functions on the flat three dimensional torus $\mathbb{T}$ having weak derivatives of order $m$ lying in $L^p(\mathbb{T})$. These can always be lifted to periodic functions in $\mathbb{R}^3$, if desired, in the obvious way. We shall work with the case for $p=2$, and as such will denote the resulting (real) Hilbert space by $H^m(\mathbb{T}) := W^{m, 2}(\mathbb{T}).$ Furthermore, we will be dealing with velocities and stress tensors, whose components will each lie in some appropriate Sobolev space. As such, we define
\begin{equation*}
\mathbb{H}^m_\ell = \underbrace{H^m(\mathbb{T})\times \dots \times H^m(\mathbb{T})}_\ell, 
\end{equation*} 
where the product is taken $\ell$-times, with the Hilbert norm given by 
\begin{equation*}
\|\left(f_1, \dots , f_\ell  \right) \|_{\mathbb{H}^m_\ell}^2 = \underbrace{\|f_1\|_{H^m}^2 + \dots + \|f_\ell\|_{H^m}^2}_\ell. 
\end{equation*}
In the case when $m=0$, we write $\mathbb{L}^2_\ell$. We will be particularly interested in the case $m=2$, and so we will have our velocity $u \in \mathbb{H}^2_3$ and the conformation tensor $\sigma \in \mathbb{H}^2_9$, since it has nine components.

Recall that for $U$, which can either be an open domain in the plane or a Riemannian manifold with or without boundary, the space of functions $H^1_0(U)$ consists of the functions in $H^1(U)$ which are zero on the boundary in the trace sense. We also define $H^{-1}(U)=(H^1_0(U))^*$ to be the dual space for $H^1_0(U)$. Since the torus has an empty boundary, we have $H^1_0(\mathbb{T})=H^1(\mathbb{T})$ and $H^{-1}(\mathbb{T})=(H^1(\mathbb{T}))^*$ 

We are interested in the case of incompressible flows for which the velocity $u$ satisfies
\begin{equation*}
\nabla \cdot u = 0.
\end{equation*}
Our velocities will lie in $\mathbb{H}^2_3$, where we mean that the above condition holds in the weak sense. Therefore, we define
\begin{equation*}
\mathcal{V} = \lbrace v \in C^\infty (\mathbb{T}) \times C^\infty (\mathbb{T}) \times C^\infty (\mathbb{T}) \; | \; \nabla \cdot v=0 \rbrace \subset \mathbb{H}^2_3, 
\end{equation*}
and then take the closure of this subspace with respect to the $\mathbb{H}^2_3$ norm. We denote the resulting space by
\begin{equation*}
\overbar{\mathbb{H}}^2_3 \coloneqq \overbar{\mathcal{V}} \subset \mathbb{H}^2_3, 
\end{equation*}
and so 
\begin{equation*}
\overbar{\mathbb{H}}^2_3 = \lbrace u \in \mathbb{H}^2_3 \; | \; \nabla \cdot u =0 \; \text{almost everywhere} \rbrace.
\end{equation*}
We shall also use the standard notation for spaces involving time. Let $Z$ denote a real Banach space. The space $L^p(0, T; Z)$ consists of all measurable functions $\textbf{u}: \left[ 0, T \right] \rightarrow Z$ with
\begin{equation*}
\Vert \textbf{u} \Vert_{L^p(0, T; Z)} \coloneqq \left( \int_{0}^{T} \Vert \textbf{u}(t) \Vert^p \, dt \right)^{1/p} < \infty
\end{equation*} 
for $1\leq p<\infty$, and $\|\textbf{u}\|_{L^\infty (0, T; Z)} \coloneqq \esssup_{0\leq t\leq T} \|\textbf{u}(t)\| < \infty$ for $p=\infty$. The space $C(\left[ 0, T \right]; Z)$
comprises all continuous functions $\textbf{u}(t):[0, T]\rightarrow Z$ with
\begin{equation*}
\Vert \textbf{u} \Vert_{C([0, T]; Z)} \coloneqq \max_{0\leq t \leq T} \Vert \textbf{u}(t) \Vert < \infty. 
\end{equation*}
Now let $\textbf{u} \in L^1(0, T; Z)$. We say $\textbf{v} \in L^1(0, T; Z)$ is the weak time derivative of $\textbf{u}$ provided that
\begin{equation*}
\int_{0}^{T} w'(t)\textbf{u}(t) \, dt=-\int_{0}^{T} w(t)\textbf{v}(t) \, dt, 
\end{equation*}
for all scalar test functions $w \in C^\infty_c(0, T)$. The Sobolev space $
W^{1, p}\left( 0, T; Z \right)$ consists of all functions $\textbf{u}\in L^p(0, T; Z)$ such that $\textbf{u}'$ exists in the weak sense and belongs to $L^p(0, T; Z)$. Furthermore,
\begin{equation*}
\Vert \textbf{u} \Vert_{W^{1, p}(0, T; Z)} \coloneqq 
\begin{cases} 
\left( \int_{0}^{T} \Vert \textbf{u}(t) \Vert^p + \Vert \textbf{u}'(t)\Vert^p \, dt \right)^{1/p} & \left( 1\leq p< \infty\right) \\
\esssup_{0\leq t\leq T} \left( \Vert \textbf{u}(t) \Vert + \Vert \textbf{u}'(t) \Vert \right)& \left(p=\infty \right). 
\end{cases}
\end{equation*}
In addition, we write $H^1(0, T; Z)=W^{1, 2}(0, T; Z)$. 
In addition, given two tensors we define
\begin{equation*}
A : B = \sum_{k, l} A_{kl}B_{kl}
\end{equation*}
to be the Frobenius inner product, with its associated norm. Finally, we remark that if $v$ is divergence free then for any scalar function $h$ we have 
\begin{equation*}
    \int_{\mathbb{T}} h\left[(v\cdot \nabla)h\right]  \, dx = -\int_\mathbb{T} \left[\nabla \cdot (hv) \right] h \, dx = -\int_\mathbb{T} \left[ (\nabla \cdot v) h + (v\cdot \nabla)h \right] h \, dx = -\int_{\mathbb{T}} h\left[(v\cdot \nabla)h\right]  \, dx = 0
\end{equation*}
where we integrated by parts. This identity is used often without comment in various calculations.

\subsection{Organization of the Paper}
The rest of the paper is organized as follows. In \S \ref{sec:model}, we present our regularized Oldroyd-B model and its vorticity formulation. We state the main theorems of the paper regarding the existence and uniqueness of global solutions for both the diffusive and non-diffusive models. The two dimensional case is then given as a corollary. We also clarify the relationship between the regularized Oldroyd-B model in the velocity form and the vorticity formulation. The main theorems for the diffusive and non-diffusive models are then proven in \S \ref{SectionDissipativeProof} and \S \ref{SectionNonDissipative}, respectively. In these sections we also give the relevant stability results. We then investigate the $L^2$ convergence of the solutions in the limit as $\epsilon \rightarrow 0$ in \S \ref{SEC:L2ConvergenceSection}. In \S \ref{con}, we offer some concluding remarks. 
\section{The Regularized Oldroyd-B Model and the Main Theorem}\label{sec:model} 

We define a regularizing operator $J$ as follows. Let $\eta$ be a smooth bump function satisfying 
\begin{equation}
	\int \eta(x) \, dx=1.
\end{equation} 
For any function $f(x, t)$ we define its regularization in space by
\begin{equation}
	Jf(x, t)=\int \eta(x-y) f(y, t) \, dy
\end{equation}  
which is simply the convolution. If the object we are regularizing has multiple components, then the regularization is done component-wise. For example, for the velocity $u$, 
\begin{equation}
    Ju(x, t)=(Ju^i(x, t))_{1\leq i\leq 3}
\end{equation}
and for the conformation tensor 
\begin{equation}
    J\sigma(x, t)=(J\sigma_{ij}(x, t))_{1\leq i, j \leq 3}.
\end{equation}
We also point out that if $u$ is incompressible, so is $Ju$ since
\begin{subeqnarray*}
\nabla \cdot Ju(x, t) = \nabla \cdot \left( \int \eta(x-y) u(y, t) \, dy \right) = \nabla \cdot \left(\int \eta(y) u(x-y, t) \, dy\right) = \int \eta(y) \left(\nabla \cdot u(x-y, t) \right)\, dy = 0. 
\end{subeqnarray*}
The operator $J$ depends on the choice of $\eta$. For the discussion that follows, we choose some $\eta$ and fix it once and for all.

Next, we take \eqref{OldroydB} as our starting point and introduce the regularized set of equations
\begin{equation}\label{RegularizedOldroydB}
\left \{ \begin{array}{l}
\partial_t u - \frac{\nu}{R_e} \Delta u + \left(Ju \cdot \nabla\right) u = \frac{1}{R_e}\text{div} J\sigma -\frac{1}{R_e}\nabla p+ \frac{1}{R_e}\mathcal{F} \\
\partial_t \sigma- \epsilon \Delta \sigma + \left(Ju \cdot \nabla\right) \sigma - (\nabla Ju)\sigma - \sigma (\nabla Ju)^T + \frac{1}{W_i}\sigma = \frac{1 - \nu}{W_i^2} I  \\ 
\nabla \cdot u = 0, \quad 
 \int_{\mathbb{T}} u(x, t) \; dx = 0.
\end{array}\right. 
\end{equation}
We will explain this last equation later on.
We impose the initial conditions
\begin{equation} \label{InitialConditionsH2}
   \quad u(0)=u_0 \in \overbar{\mathbb{H}}^2_3, \quad \mbox{ with } \int_{\mathbb{T}} u_0 \; dx = 0, \quad \mbox{ and } \quad  \sigma(0) =\sigma_0 \in \mathbb{H}^2_{9}.
\end{equation}
We will now put this system of equations into the vorticity form. To simplify the notation we will define
\begin{equation}
\alpha \coloneqq \frac{\nu}{R_e}, \quad \beta \coloneqq \frac{1}{R_e}, \quad \gamma \coloneqq \frac{1}{W_i}, \quad \mbox{ and } \quad \delta \coloneqq \frac{1-\nu}{W_i^2}. 
\end{equation}
While taking the curl of the velocity equation, the only term which is difficult to deal with is
\begin{equation}
    \nabla \times \left( \left( Ju \cdot \nabla\right) u \right)
\end{equation}
which can be expressed in two different ways, as detailed in Appendix \ref{VorticityDerivation}. Writing $\nabla \times u= \omega$ and $\nabla \times Ju = \omega_J$ we can write 
\begin{equation}
    \nabla \times \left( \left( Ju \cdot \nabla\right) u \right) = \frac{1}{2} \left ( - \nabla \times \nabla \times (Ju \times u) + (u \cdot \nabla ) \omega_J - (\omega_J \cdot \nabla) u + (Ju \cdot \nabla) \omega - (\omega \cdot \nabla)Ju \right )
\end{equation}
and while it is nice to know that it is possible to express this term in a closed form, this expression is too unwieldy to deal with. Instead, we can express the term as
\begin{equation}
    \nabla \times \left( \left( Ju \cdot \nabla\right) u \right)=(Ju\cdot \nabla) \omega + \Omega(Ju, u)
\end{equation}
where $\Omega(Ju, u)$ is given by the expression \eqref{A10ExpressionForOmega}.
Therefore taking the curl of the velocity equation gives us
\begin{equation}
    \partial_t \omega - \alpha \Delta \omega + (Ju \cdot \nabla)\omega + \Omega(Ju, u) = \beta \left( \nabla \times \text{div} J\sigma   \right) + g
\end{equation}

As mentioned earlier, we will be looking for a velocity and conformation tensor for which the components are in $H^2=W^{2, 2}$ on the torus. With the operator $J$ defined, we will look for solutions of the system: 
\begin{equation}\label{RegularizedVorticityForm}
\left \{ \begin{array}{l} 
\partial_t \omega - \alpha \Delta \omega + (Ju \cdot \nabla)\omega + \Omega(Ju, u) = \beta \left( \nabla \times \text{div} J\sigma   \right) + g	\\ 
\partial_t \sigma  -\epsilon \Delta \sigma + \left(Ju \cdot \nabla \right) \sigma  - \left(\nabla Ju\right) \sigma -\sigma \left(\nabla Ju\right)^T + \gamma \sigma = \delta I \\
\nabla \cdot u = 0, \quad \nabla \times u =\omega, \quad  \int_{\mathbb{T}} u(x, t) \; dx = 0
\end{array} \right. 
\end{equation}  
subject to the initial conditions \eqref{InitialConditionsH2},
thought of as either a periodic system or a system on the torus. We need to assume some kind of regularity for the forcing function $g = g(x,t) = (g^i)_{1\leq i \leq 3}$. 
First, we assume that 
\begin{equation}\label{CurlForce}
    g= \beta \nabla \times \mathcal{F}
\end{equation}
for some vector valued function $\mathcal{F}$. For simplicity, we assume that the components of $\mathcal{F}$ are $C^1$ in space and continuous in time, resulting in the components of $g$ being in  $C^0$ in space and continuous in time. In particular, we assume
\begin{equation*}
    \sup_{0\leq t < \infty} \Vert \mathcal{F}(x, t) \Vert_{C^1} < \infty
\end{equation*}
where $\Vert \cdot \Vert_{C^k}$ indicates the uniform norm in the spatial variable. In particular we assume that the following estimate holds
\begin{equation} \label{gBound}
	\sup_{0\leq t < \infty} \sum_i \left(   \Vert \mathcal{F}^i(x, t) \Vert_{C^0} + \Vert g^i(x, t) \Vert_{C^0} \right)  \leq K_g,
\end{equation}
for some constant $K_g$ and that
\begin{equation} \label{gCondition}
    \int_{\mathbb{T}} \mathcal{F}(x, t) \; dx =  \int_{\mathbb{T}} g(x, t) \; dx =0. 
\end{equation}
Note that the constant $K_g$ can take on any finite value. These are reasonable assumption to put on the external force on the system.

We now comment on the condition
\begin{equation}\label{AverageCondition}
    \int_{\mathbb{T}} u(x, t) \; dx = 0, 
\end{equation}
as this condition might initially seem strange. We see that integrating the velocity equation we have
\begin{equation*}
\int_{\mathbb{T}} \partial_t u  \; dx= \int_{\mathbb{T}} \left( \alpha \Delta u - (Ju \cdot \nabla)u + \beta \text{div} J\sigma - \beta \nabla p + \beta \mathcal{F}    \right) \; dx =0 
\end{equation*}
as long as 
\begin{equation*}
\int_{\mathbb{T}} \mathcal{F} \; dx =0, 
\end{equation*}
which as mentioned is a mild assumption. Then for a sufficiently regular solution we'd obtain
\begin{equation*}
\int_{\mathbb{T}} \partial_t u \; dx =\partial_t \left(  \int_{\mathbb{T}} u \; dx  \right) =0, 
\end{equation*}
which says that the average velocity remains constant. So if 
\begin{equation} \label{AverageInitialCondition}
\int_{\mathbb{T}} u_0 (x) \; dx = 0
\end{equation}
then
\begin{equation*}
\int_{\mathbb{T}} u(x, t) \; dx = 0, \quad \forall \quad t \ge 0. 
\end{equation*}
Using a coordinate transformation we can always ensure that the average initial velocity is $0$. Thus, it is a reasonable requirement to add to our system of equations. In fact, when obtaining solutions numerically, the condition \eqref{AverageCondition} is imposed.

We are now in a position to state our main theorems. We begin with the ones for the diffusive case:
\begin{theorem}\label{GlobalExistenceTheorem1}
	Suppose the forcing function $g= \beta \nabla \times \mathcal{F}$ satisfies the conditions \eqref{gBound} and \eqref{gCondition}, and that $\sigma_0$ is symmetric and positive semi-definite. Then the system of equations \eqref{RegularizedVorticityForm} subject to the initial conditions \eqref{InitialConditionsH2}  on the torus (or thought of as a system with periodic boundary conditions) has a unique solution $(u, \sigma)$ with $u\in C([0, T]; \overbar{\mathbb{H}}^2_3)$,  $u'\in L^2(0, T; \overbar{\mathbb{H}}^1_3)$, $\sigma \in C([0, T]; \mathbb{H}^2_{9})$, and $\sigma'\in L^2(0, T; \mathbb{H}^1_9)$ for any $\epsilon > 0$,  for every $T>0$, meaning the solution is global. 
\end{theorem}
The stability of these solutions is described in Theorem \ref{H2Stability}. We can obtain better regularity for our solutions. Of course, for this to be the case, the initial data and the forcing function must be more regular. So let us suppose we have the initial conditions
\begin{equation} \label{InitialConditionsHk}
    u(0)=u_0 \in \overbar{\mathbb{H}}^k_3, \quad \mbox{ with } \int_{\mathbb{T}} u_0 \; dx = 0, \quad \mbox{ and } \quad  \sigma(0) =\sigma_0 \in \mathbb{H}^k_{9}
\end{equation}
and that the forcing function satisfies
\begin{equation} \label{ForcingFunctionBoundK}
	   \sup_{0\leq t < \infty} \sum_i \Vert g^i(x, t) \Vert_{C^{k-2}}\leq K_{g, k}
	\end{equation}
for some constant $K_{g, k}$. We obtain the following generalization:

\begin{theorem}\label{GlobalExistenceTheorem1HigherRegularity}
	Let $k\geq 2$. Suppose the forcing function $g= \beta \nabla \times \mathcal{F}$ satisfies the conditions \eqref{gBound}, \eqref{gCondition} and \eqref{ForcingFunctionBoundK}  
	for some constant $K_{g, k}$. Suppose that $\sigma_0$ is symmetric and positive semi-definite. Then the system of equations \eqref{RegularizedVorticityForm} subject to the initial conditions \eqref{InitialConditionsHk} on the torus (or thought of as a system with periodic boundary conditions) has a unique solution $(u, \sigma)$ with $u\in C([0, T]; \overbar{\mathbb{H}}^k_3)$,  $u'\in L^2(0, T; \overbar{\mathbb{H}}^{k-1}_3)$, $\sigma \in C([0, T]; \mathbb{H}^{k}_{9})$, and $\sigma'\in L^2(0, T; \mathbb{H}^{k-1}_9)$  for any $\epsilon > 0$,  for every $T>0$, meaning the solution is global. In particular, if $k\geq 4$, then the solutions are classical.
\end{theorem}
We have analogous theorems for the non-diffusive case:

\begin{theorem}\label{GlobalExistenceTheorem2}
		Suppose the forcing function $g= \beta \nabla \times \mathcal{F}$ satisfies the conditions \eqref{gBound} and \eqref{gCondition}, and that $\sigma_0$ is symmetric and positive semi-definite. Then the system of equations \eqref{RegularizedVorticityForm} subject to the initial conditions \eqref{InitialConditionsH2} on the torus (or thought of as a system with periodic boundary conditions) has a unique solution  $(u, \sigma)$ with $u\in C([0, T]; \overbar{\mathbb{H}}^2_3)$,  $u'\in L^2(0, T; \overbar{\mathbb{H}}^1_3)$, $\sigma \in C([0, T]; \mathbb{H}^2_{9})$, and $\sigma'\in L^2(0, T; \mathbb{H}^1_9)$ for $\epsilon = 0$,  for every $T>0$, meaning the solution is global. 
\end{theorem}

\begin{theorem}\label{GlobalExistenceTheorem2HigherRegularity2}
	Let $k\geq 2$. Suppose the forcing function $g= \beta \nabla \times \mathcal{F}$ satisfies the conditions \eqref{gBound}, \eqref{gCondition} and \eqref{ForcingFunctionBoundK}  
	for some constant $K_{g, k}$. Suppose that $\sigma_0$ is symmetric and positive semi-definite. Then the system of equations \eqref{RegularizedVorticityForm} subject to the initial conditions \eqref{InitialConditionsHk} on the torus (or thought of as a system with periodic boundary conditions) has a unique solution $(u, \sigma)$ with $u\in C([0, T]; \overbar{\mathbb{H}}^k_3)$,  $u'\in L^2(0, T; \overbar{\mathbb{H}}^{k-1}_3)$, $\sigma \in C([0, T]; \mathbb{H}^{k}_{9})$, and $\sigma'\in L^2(0, T; \mathbb{H}^{k-1}_9)$  for $\epsilon = 0$,  for every $T>0$, meaning the solution is global. In particular, if $k\geq 4$, then the solutions are classical.
\end{theorem}

The stability of the non-diffusive solutions is described in Theorem \ref{L2StabilityNonDiffusive}.
There is also a corresponding two dimensional analogue of \eqref{RegularizedVorticityForm}. Normally, the $d = 2$ and $d = 3$ cases would require different proofs, with the $d = 2$ case being much easier. As mentioned earlier, this is due to Ladyzhenskaya's inequality, which has different exponents in different dimensions. See equation (45) in \cite{constantinnote}. However, the presence of the regularizer allows us to prove the $d = 3$ case directly, and then the existence of the solution in two dimensions can be easily obtained as a corollary to the three dimensional case.  

\begin{cor}
Let $k\geq 2$. Suppose that the (one component) forcing function $g=\beta \nabla \times \mathcal{F}$ satisfies \eqref{gBound}, \eqref{gCondition} and \eqref{ForcingFunctionBoundK}. Suppose that $\sigma_0$ is symmetric and positive semi-definite. Then the two dimensional system corresponding to \eqref{RegularizedVorticityForm} subject to the two-dimensional version of the initial conditions \eqref{InitialConditionsHk} on the torus (or thought of as a system with periodic boundary conditions) has a unique solution $(u, \sigma)$ with $u\in C([0, T]; \overbar{\mathbb{H}}^k_2)$,  $u'\in L^2(0, T; \overbar{\mathbb{H}}^{k-1}_2)$, $\sigma \in C([0, T]; \mathbb{H}^{k}_{4})$, and $\sigma'\in L^2(0, T; \mathbb{H}^{k-1}_4)$  for any $\epsilon \geq 0$,  for every $T>0$, meaning the solution is global. 
\end{cor}
\begin{proof}
 There are two possible ways of handling the two dimensional case. One is to define the ``curl" of a two dimensional quantity $f=(f^1, f^2)$ as 
\begin{equation*}
    \nabla \times f \coloneqq  \partial_1 f^2 - \partial_2 f^1
\end{equation*}
and applying this to the corresponding two dimensional velocity equation. This would then yield an equation for the vorticity (which in this case would be a single scalar function), and the analysis we use to obtain a solution to the 3 dimensional case would carry over almost verbatim to yield a solution to the 2 dimensional case. The only modification that would have to be made is how we recover the velocity from the vorticity.

The other possibility is to embed the two dimensional problem into the three dimensional one. Suppose we start with a two dimensional initial velocity and symmetric conformation tensor, only depending on $x^1$  and $x^2$. We can extend these to obtain 3 dimensional initial quantities
\begin{align*}\label{InitialData}
    u_0=\begin{pmatrix}
    u_0^1(x^1, x^2) \\
    u_0^2(x^1, x^2) \\
    0
    \end{pmatrix} \quad
    \sigma_0=\begin{pmatrix}
    \sigma_{0, 11}(x^1, x^2) & \sigma_{0, 12}(x^1, x^2) & 0 \\
    \sigma_{0, 21}(x^1, x^2) & \sigma_{0, 22}(x^1, x^2) & 0 \\
    0 & 0 & 0
    \end{pmatrix}
\end{align*}
and furthermore let $g$ be a function satisfying \eqref{gBound} and \eqref{gCondition}. Then we can let
\begin{equation*}
    \tilde{g}=\begin{pmatrix}
    0 \\
    0 \\
    g
    \end{pmatrix}
\end{equation*}
be our forcing function.

We can then use these to obtain a solution
\begin{align*}
    u(x, t)=\begin{pmatrix}
    u^1(x, t) \\
    u^2(x, t) \\
    u^3(x, t)
    \end{pmatrix} \quad \mbox{ and } \quad 
    \sigma=\begin{pmatrix}
    \sigma_{11}(x, t) & \sigma_{12}(x, t) & \sigma_{13}(x, t) \\
    \sigma_{21}(x, t) & \sigma_{22}(x, t) & \sigma_{2, 3}(x, t) \\
    \sigma_{31}(x, t) & \sigma_{32}(x, t) & \sigma_{33}(x, t)
    \end{pmatrix}. 
\end{align*}
It is then not too difficult to check that the solution is of the form
\begin{eqnarray}\label{SolutionForm}
u(x, t)=\begin{pmatrix}
u^1(x^1, x^2, t) \\
u^2(x^1, x^2, t) \\
0
\end{pmatrix} \quad \mbox{ and } \quad 
\sigma=\begin{pmatrix}
\sigma_{11}(x^1, x^2, t) & \sigma_{12}(x^1, x^2, t) & 0 \\
\sigma_{21}(x^1, x^2, t) & \sigma_{22}(x^1, x^2, t) & 0 \\
0 & 0 & 0
\end{pmatrix}, 
\end{eqnarray}
and so we can extract out the solution to the two dimensional problem. This completes the proof. 
\end{proof}

Next we comment on the relationship between the solutions of \eqref{RegularizedOldroydB} and \eqref{RegularizedVorticityForm}. It is not surprising that they are equivalent, and the proof is basically trivial. However, we include the proof for completeness.  

\begin{proposition} \label{equivalence}
The functions $u$ and $\sigma$, of sufficiently high regularity, are a solution of the system of equations \eqref{RegularizedOldroydB} with the auxiliary conditions
\begin{equation}\label{AuxiliaryConditions}
u_0 \in \overbar{\mathbb{H}}^2_3, \quad \sigma_0 \in \mathbb{H}^2_9, \quad \mbox{ and } \quad \int_{\mathbb{T}} u_0 \; dx = \int_{\mathbb{T}} u \; dx = 0
\end{equation} 
if and only if they are a solution of the system \eqref{RegularizedVorticityForm}. That is, the two systems are equivalent, and obtaining a sufficiently regular, unique, global solution to one system, yields a unique global solution of the other. 
\end{proposition} 
\begin{proof} One direction is obvious. If we have a sufficiently regular solution of \eqref{RegularizedOldroydB}, then upon taking the curl of the velocity equation, we see that $u$ and $\sigma$ also satisfy \eqref{RegularizedVorticityForm}. On the other hand, suppose that $u$ and $\sigma$ satisfy \eqref{RegularizedVorticityForm}, but not \eqref{RegularizedOldroydB}. Let us consider these as periodic systems in $\mathbb{R}^3$. Therefore, we must have 
\begin{equation*}
    \partial_t u - \frac{\nu}{R_e} \Delta u + \left(Ju \cdot \nabla\right) u \neq \frac{1}{R_e}\text{div} J\sigma -\frac{1}{R_e}\nabla p+ \frac{1}{R_e}f
\end{equation*}
or put another way, there must be some extra term on the right hand side, which we call $H$. So we would have
\begin{equation*}
    \partial_t u - \frac{\nu}{R_e} \Delta u + \left(Ju \cdot \nabla\right) u = \frac{1}{R_e}\text{div} J\sigma -\frac{1}{R_e}\nabla p+ \frac{1}{R_e}f + H
\end{equation*}
and notice we must have $H \neq \nabla \Psi$, since otherwise we could absorb this into the pressure term. However, upon taking the curl of the equation we find that necessarily
\begin{equation*}
    \nabla \times H = 0
\end{equation*}
which means there exists some $\Psi$ such that $H = \nabla \Psi$, yielding a contradiction. Hence, $u$ and $\sigma$ must be a solution of \eqref{RegularizedOldroydB}. Therefore, a unique solution to one system also yields a unique solution to the other. In the case of solutions with components in $H^2$ we simply think of everything in the weak sense and the argument still holds. This completes the proof. 
\end{proof} 
Next, we obtain an energy estimate satisfied by the solutions of $\eqref{RegularizedOldroydB}$, which by Proposition \ref{equivalence}, is also satisfied by the solutions of \eqref{RegularizedVorticityForm}. This energy estimate will be instrumental in obtaining global existence for our solutions. We remark that we need to regularize the conformation tensor in the velocity equation for the following proof to work. This is the technical reason we mentioned earlier.

\begin{proposition}[Energy Estimate]\label{EnergyEstimate}
There exist generic constants $C > 0$ and $0 < \overline{\alpha} \leq \alpha$ such that the following energy estimate holds true for the solutions of \eqref{RegularizedOldroydB}:
\begin{align}\label{EnergyEstimateEquation} \begin{split}
\int_{\mathbb{T}} |u(x, t)|^2 dx + 2 \overline{\alpha} \int_0^t \int_{\mathbb{T}}|\nabla {u}(x, s)|^2 dx \, ds+ \beta\int_\mathbb{T} {tr} \sigma(x, t) \,  dx + \beta \gamma \int_0^t \int_\mathbb{T} { tr}  \sigma(x, s) \,dx \, ds \\
\leq \int_\mathbb{T} |{u}(x,0) |^2 dx + \beta \int_\mathbb{T} { tr} 
\sigma(x,0) \,dx + \beta \delta  \int_0^t d |\mathbb{T}| ds + C \int_0^t \|\mathcal{F}\|^2_{\mathbb{L}^2_3} \, ds.  \end{split}
\end{align}
where $d$ is the dimension, which is $3$ in our case and $|\mathbb{T}|$ is the measure of $\mathbb{T}$, which is $1$ in our case. In particular, if $\sigma(t)$ is positive semi-definite (meaning that $tr \sigma \geq 0$) and $\mathcal{F}$ satisfies \eqref{gBound} , then the estimate
\begin{align}\label{k1k2} \begin{split}
\Vert u(t) \Vert^2_{\mathbb{L}^2_3} = \int_\mathbb{T} |u(x, t)|^2 dx &\leq \int_\mathbb{T} |{u}(x,0) |^2 dx + \beta \int_\mathbb{T} { tr} 
\sigma(x,0) \,dx + \beta \delta  \int_0^t d |\mathbb{T}| dt + C \int_0^t \|\mathcal{F}(s)\|^2_{\mathbb{L}^2_3} \, ds \\ &\coloneqq E_1t + E_2 \end{split}
\end{align}
holds for some constants $E_1$ and $E_2$ which depend only on the initial data, $\beta$, $\delta$, $d$, $|\mathbb{T}|$, the generic constant $C$, and the constant $K_g$ in \eqref{gBound}.
\end{proposition}
\begin{proof} 
Employing the summation convention, we recall that 
\begin{eqnarray*} 
\int_{\mathbb{T}} Ju^i\partial_i u^j u^j \, dx = - \int_{\mathbb{T}} u^j \partial_i (Ju^i u^j) \,dx = - \int_{\mathbb{T}} u^j (\partial_i Ju^i u^j + Ju^i \partial^i u^j) \,dx = - \int_{\mathbb{T}} u^j Ju^i \partial_i u^j \,dx.     
\end{eqnarray*}
This means that
\begin{eqnarray*}
\int_{\mathbb{T}} Ju^i\partial_i u^j u^j \, dx = 0.
\end{eqnarray*}
Next we note the following two identities: 
\begin{eqnarray*} 
\int_\mathbb{T} \nabla_x \cdot \left ( \int_\mathbb{T} J({x} - {y}) {\sigma}({y}) \, d{y} \right ) {u}({x}) \, d {x}
&=&  \int_\mathbb{T} \left ( \int_\mathbb{T} \nabla_x  \cdot J({x} - {y}) {\sigma}({y}) \, d{y} \right ) {u}({x}) \, d {x}  \\
&=& - \int_\mathbb{T} {\sigma}({y}) \left ( \int_\mathbb{T} \nabla_{{y}} \cdot J({y} - {x}) {u}({x}) \, d{x} \right ) \, d{y}, 
\end{eqnarray*}
and
\begin{eqnarray*}
(\sigma:S(Ju)) = \frac{1}{2} {\mathrm tr}( (\nabla J u) \sigma + \sigma (\nabla J u)^T)\, \quad \mbox{ and } \quad \int_\mathbb{T} \Delta {\mathrm tr}(\sigma) \, dx = 0,  \end{eqnarray*}
where $S({u})=\frac{1}{2}(\nabla {u}+\nabla^T {u})$. From the equation and the incompressibility constraint, we can can take the innder product of the velocity equation with ${u}$ and integrate in space to obtain the equation 
\begin{eqnarray}\label{EnergyEstimateEq1}
\frac{1}{2}\frac{d}{dt}\int_\mathbb{T} |{u}|^2 dx+ \alpha \int_\mathbb{T} |\nabla {u}|^2 dx + \beta \int_{\mathbb{T}} (\sigma : S(J {u})) dx  = \beta \int_\mathbb{T} \mathcal{F} u\, dx. 
\end{eqnarray}
On the other hand, taking the constitutive equation and taking its Frobenius product with the identity matrix $I$ we obtain
\begin{eqnarray}
\int_{\mathbb{T}}\left( \partial_t \sigma : I \right ) dx - \int_{\mathbb{T}}((\nabla Ju) \sigma + \sigma (\nabla Ju) : I) dx + \gamma \int_{\mathbb{T}}({\sigma}:I ) dx = \delta \int_{\mathbb{T}} (I : I) dx =  \delta d |\mathbb{T}|. \end{eqnarray}
Multiplying by $\beta$, rearranging, and using the aforementioned identities we can rewrite this as 
\begin{equation} \label{EnergyEstimateEq2}
\beta \int_{\mathbb{T}} ( \sigma : S (Ju)) dx = \frac{\beta}{2} \frac{d}{dt} \int_{\mathbb{T}} {\mathrm tr}(\sigma) dx + \frac{\beta \gamma}{2} \int_\mathbb{T} {\mathrm tr} \sigma \, dx - \frac{\beta \delta}{2} d |\mathbb{T}|.
\end{equation}
Therefore substituting \eqref{EnergyEstimateEq2} into \eqref{EnergyEstimateEq1} we arrive at the identity: 
\begin{eqnarray*}
\frac{1}{2}\frac{d}{d t}\int_{\mathbb{T}}|{u}|^2 dx + \alpha \int_{\mathbb{T}}|\nabla {u}|^2 dx + \frac{\beta}{2} \frac{d}{dt} \int_\mathbb{T} 
{\mathrm tr} {\sigma}  \,  dx + \frac{ \beta \gamma}{2} 
\int_{\mathbb{T}}{\mathrm tr}\sigma \,dx 
= \frac{\beta \delta}{2} d |\mathbb{T}| + \beta \int_\mathbb{T} \mathcal{F}  u \, dx. 
\end{eqnarray*}
Applying the Cauchy-Schwarz ineqquality, we obtain that
\begin{eqnarray*}
\int_\mathbb{T} \mathcal{F}  u \, dx \leq \frac{1}{ 2\varepsilon} \left (\int_\mathbb{T} |\mathcal{F} |^2 \, dx \right) + \frac{\varepsilon}{2} \left ( \int_\mathbb{T} |u|^2\, dx \right ).
\end{eqnarray*}
By choosing $\varepsilon$ sufficiently small and applying Poincare's equality, there exists $0 < \overline{\alpha} \leq \alpha$ such that the following inequality holds: 
\begin{eqnarray*}
\frac{1}{2}\frac{d}{d t}\int_{\mathbb{T}}|{u}|^2 dx+ \overline{\alpha} \int_{\mathbb{T}}|\nabla {u}|^2 dx + \frac{\beta}{2} \frac{d}{dt} \int_\mathbb{T} {\mathrm tr}{\sigma}  \,  dx +\frac{\beta \gamma}{2} 
\int_{\mathbb{T}}{\mathrm tr}{\sigma} \,dx \leq \frac{\beta \delta }{2} d |\mathbb{T}| + \frac{1}{2\varepsilon} \left (\int_\mathbb{T} |\mathcal{F} |^2 \, dx \right). 
 \end{eqnarray*}
Setting $C = \frac{1}{2\varepsilon}$ and  integrating with respect to time, we obtain the first inequality. For the second inequality, we note that if $\sigma$ is positive semi-definite for all $t$, then $tr \sigma(x, t) \geq 0$ and so all of the terms on the left hand side of \eqref{EnergyEstimateEquation} are non-negative and so \eqref{k1k2} follows. 
\end{proof}

For the remainder of this paper, we will reserve the constants $E_1$ and $E_2$ for the constants appearing in this theorem. 

\section{Global existence of solutions to the diffusive regularized Oldroyd-B model}\label{SectionDissipativeProof}
In this section we prove Theorems \ref{GlobalExistenceTheorem1} and \ref{GlobalExistenceTheorem1HigherRegularity}. We assume that $\epsilon > 0$ is fixed and consider the space 
\begin{equation}
X = X(T)\coloneqq C([0, T]; \overbar{\mathbb{H}}^2_3)\times C([0, T]; \mathbb{H}^2_9) 
\end{equation}
with the norm
\begin{equation*}
\Vert (v, w) \Vert_X^2 = \max_{0\leq t \leq T} \left( \Vert v(t) \Vert_{\mathbb{H}^2_3}^2 + \Vert w(t) \Vert_{\mathbb{H}^2_9}^2 \right).
\end{equation*}
We also consider the space
\begin{equation*}
	Y\coloneqq \overbar{\mathbb{H}}^2_3 \times \mathbb{H}^2_9
\end{equation*} 
with the norm
\begin{equation*}
\Vert (a, b) \Vert_Y^2 =  \Vert a \Vert_{\mathbb{H}^2_3}^2 + \Vert b \Vert_{\mathbb{H}^2_9}^2.
\end{equation*}
We will take $\left( \tilde{u}, \tilde{\sigma} \right)\in X$ and then try solving 
\begin{equation}\label{MMollifiedFormulation}
\left \{ \begin{array}{l} 
\partial_t \omega - \alpha \Delta \omega  = \beta \left( \nabla \times \text{div} J\tilde{\sigma}   \right)- (J\tilde{u} \cdot \nabla)\tilde{\omega} - \Omega(J\tilde{u}, \tilde{u}) + g	\\ 
\partial_t \sigma - \epsilon \Delta \sigma = \delta I - \left(J\tilde{u} \cdot \nabla \right) \tilde{\sigma}  + \left(\nabla J\tilde{u}\right) \tilde{\sigma} +\tilde{\sigma} \left(\nabla J\tilde{u}\right)^T  - \gamma \tilde{\sigma} \\
\nabla \cdot u =0\, \quad \nabla \times u =\omega \, \quad \mbox{ and } \quad \int_{\mathbb{T}} u(x, t) \; dx = 0, 
\end{array}\right. 
\end{equation}
subject to the initial conditions: 
\begin{equation}
u(0) = u_0 \in \overbar{\mathbb{H}}^2_3  \quad \mbox{ with } \quad \int_{\mathbb{T}} u_0 \, dx =0 \quad \mbox{ and } \quad 
\sigma(0) =\sigma_0 \in \mathbb{H}^2_{9}.  
\end{equation}
Solving this system of equations will yield a map $\Phi\left( \tilde{u}, \tilde{\sigma} \right) = \left( u, \sigma  \right)$. We will show that this is a contraction mapping in $X$ for some $T$, giving us a solution to the system.

\begin{remark}
	We could solve equations of the form $\partial_t \omega - \alpha \Delta \omega + J\tilde{u} \cdot \nabla \omega = \beta \left(\nabla \times \text{div} J\tilde{\sigma} \right) - \Omega(J\tilde{u}, \tilde{u})+ g$ and similarly for $\sigma$ but then the estimates we obtain would depend on the norm of $\tilde{u}$ which needlessly complicates things.
\end{remark}

\begin{remark}
At this stage, we are not imposing any condtions on $\sigma_0$ other than that $\sigma_0 \in \mathbb{H}^2_{9}$. The short time existence of solutions can be established for an arbitrary $\sigma_0$. However, the estimate \eqref{k1k2} will be key to establishing global existence of solutions, and it requires $\sigma(x, t)$ to be positive semi-definite for all $x$ and $t$. Later on we will show that if $\sigma_0$ is symmetric and positive semi-definite, then $\sigma(x, t)$ remains symmetric and positive semi-definite for all $t\geq 0$ for which it exists.  
\end{remark}

We take the initial data $(u_0, \sigma_0) \in Y$ and take $\left( \tilde{u}, \tilde{\sigma}   \right) \in X$. 
We first focus on solving the equation for the vorticity vector 
\begin{equation}\label{Vorticity}
\left \{ \begin{array}{rl} 
\partial_t \omega - \alpha \Delta \omega &= f \\ \omega(0) &= \nabla \times u_0
\end{array} \right. 
\end{equation}
 where 
\begin{equation*}
f\coloneqq \beta \left( \nabla \times \text{div} J\tilde{\sigma}   \right)- (J\tilde{u} \cdot \nabla)\tilde{\omega} - \Omega(J\tilde{u}, \tilde{u}) + g
\end{equation*}   
and
\begin{equation*}
    \tilde{\omega}=\nabla \times \tilde{u}
\end{equation*}
Due to the regularization, for any multi-index $\lambda$ with $|\lambda|\leq 4$ there exists a constant $K$, depending only on our fixed choice of $\eta$ in the definition of $J$, such that  
\begin{subeqnarray}\label{VBound}
\sum_{|\lambda|\leq 4} \sum_i |D^{\lambda }J\tilde{u}^i(x)| &\leq& K \|\tilde{u}\|_{\mathbb{L}^2_3}, \\
\sum_{|\lambda| \leq 4} \sum_{i,j} |D^{\lambda} J\tilde{\sigma}_{ij} | &\leq& K \| \tilde{\sigma}\|_{\mathbb{L}^2_9}, \\
|\Omega(J\tilde{u}, \tilde{u})| &\leq& \sum_{i, j} \sqrt{K} \|\tilde{u}\|_{\mathbb{L}^2_3} |\partial_j \tilde{u}_i|. 
\end{subeqnarray}
Therefore, we can easily obtain the estimates
\begin{subeqnarray}\label{CapitalOmegaEstimate}
\|\nabla \times \text{div} J \tilde{\sigma} \|_{\mathbb{L}^2_3}^2 &\leq& 9K^2 |\mathbb{T}| \| \tilde{\sigma} \|_{\mathbb{L}^2_9}^2 \\
\|\Omega (J\tilde{u}, \tilde{u} ) \|_{\mathbb{L}^2_3}^2 &\leq& K \|\tilde{u}\|^2_{\mathbb{L}^2_3} \left(81 \|\tilde{u} \|_{\mathbb{H}^1_3}^2 \right) \\
\|(J\tilde{u} \cdot \nabla )\tilde{\omega} \|_{\mathbb{L}^2_3}^2 &\leq& K^2 \|\tilde{u} \|_{\mathbb{L}^2_3}^2 \|\tilde{\omega} \|_{\mathbb{H}^1_3}^2 \leq 4 K^2 \|\tilde{u} \|_{\mathbb{L}^2_3}^2 \| \tilde{u} \|_{\mathbb{H}^2_3}^2.
\end{subeqnarray}
With these, we derive the following estimates: 
\begin{subeqnarray*}
\|f\|_{\mathbb{L}^2_3}^2 &\leq& 4\beta^2 \|\nabla \times \text{div} J \tilde{\sigma} \|_{\mathbb{L}^2_3}^2 + 4 \|\Omega (J\tilde{u}, \tilde{u} ) \|_{\mathbb{L}^2_3}^2 + 4 \|(J\tilde{u} \cdot \nabla )\tilde{\omega} \|_{\mathbb{L}^2_3}^2 + 4K_g^2 |\mathbb{T}| \\ 
&\leq& 36 K^2 |\mathbb{T}| \|\tilde{\sigma} \|_{\mathbb{L}^2_9}^2 + 324 K \|\tilde{u} \|^2_{\mathbb{L}^2_3} \|\tilde{u}\|_{\mathbb{H}^1_3}^2 + 16 K^2 \|\tilde{u} \|_{\mathbb{L}^2_3}^2 \|\tilde{u}\|_{\mathbb{H}^2_3}^2 + 4K_g^2 |\mathbb{T}|.
\end{subeqnarray*}
Thus if $(\tilde{u}, \tilde{\sigma})\in X$ we easily see
that $f\in C([0, T]; \mathbb{L}^2_3)$ with the estimate
\begin{equation}\label{EstimateForf}
\|f\|_{C([0, T]; \mathbb{L}^2_3)} \leq 36 K^2 |\mathbb{T}| \| (\tilde{u}, \tilde{\sigma} ) \|_X^2 + 324 K \|(\tilde{u}, \tilde{\sigma}) \|_X^4  + 16 K^2 \|(\tilde{u}, \tilde{\sigma} ) \|_X^4 + 4K_g^2 |\mathbb{T}|, 
\end{equation} 
and we can also estimate
\begin{equation} \label{EstimateForOmega}
\Vert \omega_0 \Vert_{ \mathbb{H}^1_3  } \leq 2\Vert u_0 \Vert_{\overbar{\mathbb{H}}^2_3} \leq 2\Vert (u_0, \sigma_0)\Vert_Y .
\end{equation}
We can now solve \eqref{Vorticity} in the weak sense.

\subsection{Existence of a Weak Solution to the Vorticity Equation}\label{SolvingForVorticity}

We solve \eqref{Vorticity} by using standard PDE methods. Since our equation is linear, we can treat it as three equations for the components $\omega_i$ of the vorticity vector. The equation for each component on $\mathbb{T}$ is of the form
\begin{subeqnarray}\label{ParabolicEquation} 
\partial_t W - \alpha \Delta W &=& h  \\
W(0) &=& W_0
\end{subeqnarray}
where $W$ and $h$ are scalar functions.

\begin{proposition}[Existence of a Weak Solution] \label{ExistenceOfWeakSolution}
	The initial value problem \eqref{ParabolicEquation} has a unique weak solution with $W\in L^2(0, T; H^1(\mathbb{T}))$ and $W'\in L^2(0, T; H^{-1}(\mathbb{T}))$.
\end{proposition}
\begin{proof} 
This follows in the usual way by using Galerkin approximations and energy estimates. We again remark that $H^1_0(\mathbb{T})=H^1(\mathbb{T})$. See Theorems 3 and 4 in 7.1.2 of \cite{evans1997partial}. 
\end{proof} 

\begin{proposition}[Improved Regularity] \label{ImprovedRegularity0}
There exists a constant $C_1$ depending only on $T$, $\alpha$, and the geometry of the torus such that our weak solution satisfies the estimate
\begin{equation}\label{ImprovedRegularity1}
\esssup_{0\leq t\leq T} \Vert W(t) \Vert_{H^1(\mathbb{T})} + \Vert W \Vert_{L^2(0, T; H^{2}(\mathbb{T}))} + \Vert W' \Vert_{L^2(0, T; L^2(\mathbb{T}))} \leq C_1 \left(\Vert W_0 \Vert_{H^1(\mathbb{T})} + \Vert h \Vert_{L^2(0, T; L^2(\mathbb{T}))}    \right).
\end{equation}
Furthermore the following more precise estimates hold: 
\begin{subeqnarray*} 
\esssup_{0\leq t \leq T} \|W(t)\|_{H^1(\mathbb{T})}^2 &\leq& \left( 1+ e^{(1+2\alpha)T}  \right) \|W_0\|_{H^1(\mathbb{T})}^2 + \left( \frac{1}{\alpha} + e^{(1+2\alpha)T}  \right)\|h\|_{L^2(0, T; L^2(\mathbb{T})) }^2 \\
\|W'\|_{L^2(0, T; L^2(\mathbb{T})) }^2 &\leq& (2+\alpha ) \left[ \|W_0\|_{H^1(\mathbb{T})}^2 + \|h\|_{L^2(0, T; L^2(\mathbb{T})) }^2  \right] \\
\|W\|_{L^2(0, T; H^2(\mathbb{T})) }^2 &\leq& C \left( 1+ Te^{(1+2\alpha)T}\right) \|W_0\|_{H^1(\mathbb{T})}^2 + C\left( 1 + e^{(1+2\alpha)T} \right)\|h\|_{L^2(0, T; L^2(\mathbb{T}))}^2, 
\end{subeqnarray*}
where $C$ is an elliptic constant depending only on $\alpha$ and the geometry of the torus and is thus independent of $T$.
\end{proposition}

\begin{proof}
The first estimate follows from Theorem 5 in 7.1.2 of \cite{evans1997partial}. The more precise estimates follow by keeping track of all the constants and steps in the proof of that theorem. 
\end{proof}

As a result, we obtain the following theorem. 

\begin{proposition}[Existence and Regularity for $\omega$] \label{ImprovedRegularity} The problem \eqref{Vorticity} possesses a unique weak solution with $\omega \in L^2(0, T; \mathbb{H}^1_3)$ and $\omega'\in L^2(0, T; \mathbb{L}_3^2)$. Moreover, 
There exists a constant $C_1$ depending only on $T$, $\alpha$, and the geometry of the torus such that our weak solution satisfies the estimate
\begin{equation}\label{ImprovedRegularity1}
\esssup_{0\leq t\leq T} \|\omega(t) \|_{\mathbb{H}^1_3} + \|\omega\|_{L^2(0, T; \mathbb{H}^2_3)} + \|\omega' \|_{L^2(0, T; \mathbb{L}^2_3)} \leq C_1 \left(\|\omega_0 \|_{\mathbb{H}^1_3} + \|f\|_{L^2(0, T; \mathbb{L}^2_3)}  \right).
\end{equation}
Furthermore the following more precise estimates hold: 
\begin{subeqnarray}\label{IED1}
\esssup_{0\leq t \leq T} \|\omega(t) \|_{\mathbb{H}^1_3}^2 &\leq& \left( 1+ e^{(1+2\alpha)T}   \right) \|\omega_0\|_{\mathbb{H}^1_3}^2 + \left( \frac{1}{\alpha} + e^{(1+2\alpha)T}  \right)\|f \|_{L^2(0, T; \mathbb{L}^2_3 ) }^2 \slabel{IED1} \\  
\|\omega'\|_{L^2(0, T; \mathbb{L}^2_3  ) }^2 &\leq& (2+\alpha ) \left[ \|\omega_0 \|_{\mathbb{H}^1_3}^2 + \|f \|_{L^2(0, T; \mathbb{L}^2_3  ) }^2 \right] \slabel{IED2} \\ 
\|\omega\|_{L^2(0, T; \mathbb{H}^2_3  )) }^2 &\leq& C \left( 1+ Te^{(1+2\alpha)T}\right)\|\omega_0\|_{ \mathbb{H}^1_3}^2 + C\left( 1 + e^{(1+2\alpha)T}   \right)\|f\|_{L^2(0, T; \mathbb{L}^2_3)}^2,  \slabel{IED3} 
\end{subeqnarray}
where $C$ is an elliptic constant depending only on $\alpha$ and the geometry of the torus and is thus independent of $T$.
\end{proposition}

\begin{proof}
We apply Propositions \ref{ExistenceOfWeakSolution} and \ref{ImprovedRegularity0} to the components of $\omega$ and add up the various inequalities.
\end{proof}
We need $\omega$ to lie in a better space. This is accomplished with the final proposition of this subsection.
\begin{proposition}\label{KeyTheorem}
The weak solution satisfies $\omega \in C([0, T]; \mathbb{H}^1_3)$ with the estimate
\begin{equation*}
\max_{0\leq t \leq T} \Vert \omega(t) \Vert_{\mathbb{H}^1_3  }^2 \leq \left(1+  e^{(1+2\alpha)T} \right) \Vert \omega_0 \Vert_{\mathbb{H}^1_3   }^2 + \left( \frac{1}{\alpha} + e^{(1+2\alpha)T} \right) \Vert f \Vert_{L^2(0, T; \mathbb{L}^2_3   )   }^2.
\end{equation*}
\end{proposition}
\begin{proof} 
By results in \S 5.9 in \cite{constantinnote}, we have $\omega \in C([0, T]; \mathbb{H}^1_3)$. Since $\omega(t)$ is continuous in time, we can replace $\esssup$ with $\max$ in the second estimate of Proposition \ref{ImprovedRegularity}. 
\end{proof} 

\subsection{Obtaining Velocity From Vorticity} \label{SolvingForVelocity}

Having obtained the vorticity, we must now use it to recover the velocity. Normally, the velocity is recovered
by considering the identity 
\begin{equation}\label{Identity}
\nabla \times \left( \nabla \times u \right) - \nabla \left( \nabla \cdot u \right) = -\Delta u.
\end{equation}
which holds for a velocity in $\mathbb{R}^3$. Turning this around and assuming the flow is incompressible leads to the elliptic PDE
\begin{equation*}
\Delta u = -\nabla \times \omega \quad \mbox{ and } \quad \int_{\mathbb{T}} u \; dx = 0, 
\end{equation*}
which does indeed yield an incompressible flow. This flow then generates some vorticity which we would call $\omega_u$. However, it is actually quite tricky to verify that $\omega_u=\omega$, the vorticity which we started with. In two dimensions this can actually be verified by hand, but in three dimensions the argument becomes more difficult. Thus, we use an alternate argument which gives both an incompressible flow and the right vorticity. We start with the following proposition.

\begin{proposition}\label{AverageOmega}
Suppose $\omega$ is the solution of \eqref{Vorticity}. Then 
\begin{equation*}
\int_{\mathbb{T}} \omega(t) \; dx = 0, 
\end{equation*}
for all $0\leq t \leq T$.  
\end{proposition}
\begin{proof} 
First notice that since $\omega_0 = \nabla \times u_0$ we have
\begin{equation*}
    \int_\mathbb{T} \omega_0 \; dx = 0
\end{equation*}
by applying Stokes' theorem. Next, we take \eqref{Vorticity}, take the inner product with the vector $\mathbbm{1}=(1, 1, 1)$ and integrate over space to obtain
\begin{equation*}
    \int_{\mathbb{T}} \omega' \; dx + \alpha \int_\mathbb{T} \nabla \omega \cdot \nabla \mathbbm{1} \; dx = \int_\mathbb{T} \nabla \times \left( \beta \text{div} J \tilde{\sigma} - \beta \nabla p - (J\tilde{u} \cdot \nabla ) \tilde{u}  \right) \; dx + \int_\mathbb{T} g \; dx
\end{equation*}
and so
\begin{equation*}
    \int_\mathbb{T} \omega'(s) \; dx =0
\end{equation*}
for all $0 \leq s \leq T$, 
by again using Stokes' theorem and the fact that we assume $\int_\mathbb{T} g \; dx =0$. Next, by our estimates we have that $\omega \in H^1(0, T; \mathbb{L}^2_3 )$ and so by \S 5.9 in \cite{evans1997partial},  
\begin{equation}\label{21}
\omega(t)= \omega_0 + \int_0^t \omega'(s) \; ds
\end{equation}
and taking this equation and integrating it in space we obtain
\begin{equation*}
\int_\mathbb{T} \omega(t) \; dx = \int_\mathbb{T} \omega_0 \; dx + \int_\mathbb{T} \left( \int_0^t \omega'(s) \; ds   \right) \; dx = \int_0^t \left( \int_\mathbb{T} \omega'(s) \; dx \right) \, ds = 0,
\end{equation*}
completing the proof.  
\end{proof} 
With this in hand, we can now obtain our velocity. 
\begin{proposition}\label{VelocityFromVorticity}
Let $\omega$ be the solution of \eqref{Vorticity}. Then the elliptic problem
\begin{equation}\label{VelocityProblem}
\Delta F - \nabla (\nabla \cdot F ) = \omega \quad \mbox{ subject to } \int_{\mathbb{T}} F \, dx = 0
\end{equation}
has a unique solution for each $0\leq t \leq T$,  satisfying
\begin{equation}\label{Festimate}
\|F\|_{\mathbb{H}^3_3}^2 \leq C_E \|\omega \|_{\mathbb{H}^1_3}^2, 
\end{equation}
for some elliptic constant $C_E$. Furthermore, if we define
\begin{equation*}
u \coloneqq - \nabla \times F
\end{equation*} then this velocity $u$ belongs to $C([0, T]; \overbar{\mathbb{H}}^2_3)$ and it satisfies
\begin{equation*}
\nabla \cdot u = 0, \quad \nabla \times u = \omega, \quad \mbox{ and } \quad \int_\mathbb{T} u \; dx = 0,  
\end{equation*}
along with the estimate
\begin{equation}\label{Uestimate}
\|u(t)\|_{\overbar{\mathbb{H}}^2_3 }^2 \leq \tilde{C} \|\omega(t)\|_{\mathbb{H}^1_3}^2, 
\end{equation}
for some constant $\tilde{C}$ independent of $t$.
\end{proposition}
\begin{proof} 
It is well known that \eqref{VelocityProblem} has a unique solution on $\mathbb{T}$ if and only if the right hand side satisfies
\begin{equation*}
\int_\mathbb{T} \omega \; dx = 0,
\end{equation*}
which we established in Proposition \ref{AverageOmega} for each $t$. Therefore we can solve \eqref{VelocityProblem} for each $t$ to obtain a function $F=F(t)$. Furthermore, the solution satisfies \eqref{Festimate} for some constant $C_E$ independent of $t$. We can now consider $F$ as a periodic function in $\mathbb{R}^3$ and apply the usual vector identities. We define $u=u(t) \coloneqq -\nabla \times F(t)$ which is of course periodic and thus can be pulled back to give a velocity on the torus. We have
\begin{equation*}
\nabla \cdot u = -\nabla \cdot (\nabla \times F) = 0
\quad \mbox{ and } \quad \nabla \times u = - \nabla \times (\nabla \times F)=\Delta F - \nabla (\nabla \cdot F ) = \omega 
\end{equation*}
as desired. Pulling back the velocity to the torus, or simply using the periodicity, we obtain
\begin{equation*}
\int_{\mathbb{T}} u \; dx = - \int_{\mathbb{T}} \nabla \times F \; dx = 0.
\end{equation*}
We have that 
\begin{equation*}
\|u\|_{\mathbb{H}^2_3}^2 = \|\nabla \times F \|_{\mathbb{H}^2_3}^2 \leq C_1 \|F\|_{\mathbb{H}^3_3}^2, 
\end{equation*}
for some constant $C_1$ independent of $t$ (basically this constant just counts up the number of partial derivatives of $F$ appearing in the definition of $u$) and thus combining this with \eqref{Festimate} we obtain \eqref{Uestimate} for the constant $\tilde{C} = C_1C_E$ independent of $t$. Therefore, since $\omega \in C([0, T]; \mathbb{H}^1_3)$ and $u$ is incompressible, we obtain \begin{equation*}
u \in C([0, T]; \overbar{\mathbb{H}}^2_3)
\end{equation*}
as desired, completing the proof. 
\end{proof} 
We shall now obtain the time derivative $u'$ and investigate which space it lies in. 
\begin{proposition}
Let $\omega'$ be the weak time derivative of the solution of \eqref{Vorticity}. Then the elliptic problem
\begin{equation}\label{DerivativeVelocityProblem}
\Delta G - \nabla (\nabla \cdot G ) = \omega' \quad \mbox{ subject to } \int_{\mathbb{T}} G \; dx = 0, 
\end{equation}
has a unique solution for each $0\leq t \leq T$, satisfying
\begin{equation}\label{Festimate}
\|G\|_{\mathbb{H}^2_3}^2 \leq C_E \|\omega' \|_{\mathbb{L}^2_3}^2, 
\end{equation}
for some elliptic constant $C_E$. Furthermore if we define $a$ as follows: 
\begin{equation}
a := - \nabla \times G, 
\end{equation}
then it holds that 
\begin{equation*}
a = u'
\end{equation*}
which furthermore satisfies
\begin{equation}\label{PropertiesUPrime}
\nabla \cdot u' = 0, \quad \int_\mathbb{T} u' \; dx = 0 \quad \mbox{ and } \quad \nabla \times u' = \omega', 
\end{equation}
along with the estimate
\begin{equation}\label{UPrimeEstimate1}
\|u'(t)\|_{\overbar{\mathbb{H}}^1_3 }^2 \leq \tilde{C} \| \omega'(t) \|_{\mathbb{L}^2_3}^2
\end{equation}
where $\tilde{C}$ is the constant given in Proposition \ref{VelocityFromVorticity}, and thus  $u'\in L^2(0, T; \overbar{\mathbb{H} }^1_3)$.
\end{proposition}
\begin{proof} The properties \eqref{PropertiesUPrime} and \eqref{UPrimeEstimate1} and all follow exactly as in Proposition \ref{VelocityFromVorticity}. It is then easy to check directly that $a$ satisfies the definition of $u'$. 
\end{proof} 
 
 We can summarize what we have obtained so far as follows:

\begin{proposition} \label{uExistence}
    There exists an incompressible $u$ with $u\in C([0, T]; \overbar{\mathbb{H}}^2_3)$ and $u'\in L^2(0, T; \overbar{\mathbb{H} }^1_3)$ satisfying the estimates
    \begin{subeqnarray}
\max_{0\leq t \leq T} \|u(t) \|_{\overbar{\mathbb{H}}^2_2}^2 &\leq& 4C\left( 1+ e^{(1+2\alpha)T}   \right) \| u_0 \|_{\overbar{\mathbb{H}}^2_2}^2 + C\left( \frac{1}{\alpha} + e^{(1+2\alpha)T}  \right) \|f \|_{L^2(0, T; L^2(\mathbb{T})) }^2 
\slabel{IED4} \\ 
\|u'\|_{L^2(0, T; \mathbb{H}^1_2) }^2 &\leq& C(2+\alpha ) \left[ 4 \|u_0\|_{\overbar{\mathbb{H} }^2}^2 + \|f \|_{L^2(0, T; L^2(\mathbb{T})) }^2  \right ]\slabel{IED5}. 
\end{subeqnarray} 
\end{proposition}

\begin{proof}
We collect all of our previous estimates and use \eqref{EstimateForOmega} to write them in terms of $u_0$.
\end{proof}

\subsection{Solving for the Conformation Tensor ($\epsilon>0$)} \label{SolvingForConformationTensor}
We solve the equation  
\begin{subeqnarray}\label{sigmaeq}
\partial_t \sigma - \epsilon \Delta \sigma &=& \delta I - \left(J\tilde{u} \cdot \nabla \right) \tilde{\sigma}  + \left(\nabla J\tilde{u}\right) \tilde{\sigma} +\tilde{\sigma} \left(\nabla J\tilde{u}\right)^T  - \gamma \tilde{\sigma}  \\
\sigma(0) &=& \sigma_0 \in \mathbb{H}^2_9, 
\end{subeqnarray}
which can be thought of as a system of nine equations for the nine components of $\sigma$.
By defining the right hand side as 
\begin{equation*}
\tilde{f} = \delta I - \left(J\tilde{u} \cdot \nabla \right) \tilde{\sigma}  + \left(\nabla J\tilde{u}\right) \tilde{\sigma} +\tilde{\sigma} \left(\nabla J\tilde{u}\right)^T  - \gamma \tilde{\sigma}, \end{equation*}
we have that $\tilde{f}\in C([0, T]; \mathbb{H}^1_9)$ as long as $\tilde{u} \in C([0, T]; \overbar{\mathbb{H}}^2_3)$ and $\tilde{\sigma} \in C([0, T]; \mathbb{H}^2_9)$. We can thus obtain:

\begin{proposition}[Existence and Regularity for $\sigma$] \label{ImprovedRegularitySigma} The problem \eqref{sigmaeq} possesses a unique weak solution with $\sigma \in C([0, T]; \mathbb{H}^2_9)$ and $\sigma'\in L^2(0, T; \mathbb{H}^1_9)$ which satisfies the estimates
\begin{subeqnarray} \label{SigmaEstimates}
\max_{0\leq t \leq T} \|\sigma(t) \|_{\mathbb{H}^2_9}^2 &\leq& \left( 1+ e^{(1+2\epsilon)T} \right) \|\sigma_0 \|_{\mathbb{H}^2_9}^2 + \left( \frac{1}{\epsilon} + e^{(1+2\epsilon)T} \right) \|\tilde{f}\|_{L^2(0, T; \mathbb{H}^1_9) } \slabel{IEDSF1} \\ 
\|\sigma'\|_{L^2(0, T; \mathbb{H}^1_9) }^2 &\leq& (2+\epsilon ) \left[ \|\sigma_0\|_{\mathbb{H}^2_9}^2 + \| \tilde{f}\|_{L^2(0, T; \mathbb{H}^1_9) }^2  \right ] \slabel{IEDSF2} \\ 
\|\sigma\|_{L^2(0, T; \mathbb{H}^3_9) }^2 &\leq& C\left( 1+ Te^{(1+2\epsilon)T} \right)\|\sigma_0 \|_{\mathbb{H}^2_9}^2 + C\left( 1 + e^{(1+2\epsilon)T}   \right)\|\tilde{f}\|_{L^2(0, T; \mathbb{H}^1_9) }^2.  \slabel{IEDSF3}
\end{subeqnarray}
where $C$ is an elliptic constant depending only on $\epsilon$ and the geometry of the torus and is thus independent of $T$.
\end{proposition}

\begin{proof}
Since \eqref{sigmaeq} has the same form as \eqref{Vorticity} we obtain similar estimates for $\sigma$:
\begin{subeqnarray}
\esssup_{0\leq t \leq T} \|\sigma(t) \|_{\mathbb{H}^1_9}^2 &\leq& \left( 1+ e^{(1+2\epsilon)T}   \right) \|\sigma_0 \|_{\mathbb{H}^1_9}^2 + \left( \frac{1}{\epsilon} + e^{(1+2\epsilon)T} \right) \| \tilde{f}\|_{L^2(0, T; \mathbb{L}^2_9) }^2  \slabel{IEDS1} \\ 
\|\sigma'\|_{L^2(0, T; \mathbb{L}^2_9) }^2 &\leq& (2+\epsilon ) \left[ \|\sigma_0\|_{\mathbb{H}^1_9}^2 + \|\tilde{f}\|_{L^2(0, T; \mathbb{L}^2_9) }^2 \right ] \slabel{IEDS2} \\ 
\|\sigma\|_{L^2(0, T; \mathbb{H}^2_9) }^2 &\leq& C \left( 1+ Te^{(1+2\epsilon)T} \right) \|\sigma_0 \|_{\mathbb{H}^1_9}^2 + C\left( 1 + e^{(1+2\epsilon)T}   \right) \|\tilde{f}\|_{L^2(0, T; \mathbb{L}^2_9) }^2,  \slabel{IEDS3} 
\end{subeqnarray}
for some constant $C$ depending only on $\epsilon$ and the geometry of the torus. However, these are not good enough because we want the components of $\sigma$ to be in $H^2$, not $H^1$. Now the trick is to notice that if $\tilde{f}\in C([0, T]; \mathbb{H}^1_9)$ then $\partial_i \tilde{f}= \partial_{x_i} \tilde{f} \in C([0, T]; \mathbb{L}^2_9)$. Therefore the same estimates as above can be used to solve the equations
\begin{subeqnarray*}
\partial_t v_i - \epsilon \Delta v_i &=& \partial_i \tilde{f} \\
v_i(0) &=& \partial_i \sigma_0, 
\end{subeqnarray*}  
with the same sorts of estimates. But it is easy to see that $v_i = \partial_i \sigma$ and $v_i'=\partial_i \sigma'$, which allows us to conclude
\begin{equation*}
\sigma \in L^2(0, T; \mathbb{H}^3_9), \quad \sigma' \in L^2(0, T; \mathbb{H}^1_9)
\end{equation*}
and so by results in \S 5.9 in \cite{evans1997partial},  
we conclude that $\sigma \in C([0, T]; \mathbb{H}^2_9)$. Furthermore, the above estimates for $\sigma$ can be added to similar estimates for $v_i$ to obtain the estimates \eqref{SigmaEstimates}.
\end{proof}

\subsection{The Contraction Mapping and Short Time Existence}
With these estimates, establishing the existence of a solution using the contraction mapping principle becomes easy. As before define
\begin{equation*}
X=X(T)=C([0, T]; \overbar{\mathbb{H}}^2_3) \times C([0, T]; \mathbb{H}^2_9), \quad \mbox{ and } \quad Y=\overbar{\mathbb{H}}^2_3 \times \mathbb{H}^2_9, 
\end{equation*}
equipped with the norms
\begin{equation*}
\|(v, w)\|^2_X = \max_{0\leq t \leq T} \left( \|v(t) \|_{\overbar{\mathbb{H}}^2_3}^2  \|w(t)\|_{\mathbb{H}^2_9}^2 \right) \quad \mbox{ and } \quad 
\|(a, b)\|^2_Y= \|a \|_{\overbar{\overbar{\mathbb{H}}}^2_3}^2 + \|b \|_{\mathbb{H}^2_9}^2.
\end{equation*}
Fixing the initial data $(u_0, \sigma_0)\in Y$, and taking $(\tilde{u}, \tilde{\sigma})\in X$ we solve the system of equations \eqref{MMollifiedFormulation} to obtain solutions $u$ and $\sigma$ which by Theorems \ref{uExistence} and \ref{ImprovedRegularitySigma} satisfy $(u, \sigma)\in X$. We can thus consider this as a mapping
\begin{equation*}
\Phi: X \rightarrow X \quad \mbox{ defined by } \quad \Phi(\tilde{u}, \tilde{\sigma}) = \left( u, \sigma  \right).
\end{equation*}
We also define
\begin{equation*}
	B_R= \left \{(v, w)\in X : \Vert (v, w) \Vert^2_X \leq R^2   \right \}.  
\end{equation*}
Next, we notice that if $(\tilde{u}, \tilde{\sigma}) \in B_R$, then we can estimate
\begin{equation*} 
\|f\|_{L^2(\mathbb{T})}^2 \leq K_1 (K_g^2+ R^2+ R^4) 
\end{equation*}
by \eqref{EstimateForf} for some appropriate constant $K_1$ only depending on $\beta$ and the choice of $J$. Similarly, for $\tilde{f}$ we have
\begin{equation*}
\|\tilde{f}\|_{\mathbb{H}^1_4}^2 \leq K_2 \left( 1 + R^2 + R^4  \right),  
\end{equation*}
where $K_2$ only depends on $\delta, \gamma$ and $J$. With these, we can estimate
\begin{subeqnarray}\label{B1}
\|f\|_{L^2(0, T; H^2(\mathbb{T})) }^2 &\leq& TK_1(K_g^2 +R^2+R^4), \slabel{B1} \\ 
\|\tilde{f}\|_{L^2(0, T; \mathbb{H}^1_4) }^2 &\leq& TK_2(1+R^2+R^4). \slabel{B2} 
\end{subeqnarray} 
Now it is quite easy to prove the following proposition: 
\begin{proposition}\label{SelfMap}
There exist $R$ and a $T_{u, \sigma}>0$, depending only on the initial data, such that the map $\Phi$ maps the ball $B_R\subset X(T)$ to itself for all $0\leq T \leq T_{u, \sigma}$
\end{proposition}
\begin{proof} 
We first choose a $T'>0$ small enough so that 
\begin{equation*}
e^{(1+2\epsilon)T'}\leq 2 \quad  \text{and} \quad e^{(1+2\alpha)T'}\leq 2
\end{equation*}
which then holds for all $0\leq T \leq T'$. Now define
\begin{equation}\label{Rformula}
	R^2=12C\Vert (u_0, \sigma_0)  \Vert_Y^2 + 3\Vert (u_0, \sigma_0)  \Vert_Y^2 +4 
\end{equation}   
and now choose a $0<T_{u, \sigma}\leq T'$ sufficiently small so that 
\begin{subeqnarray*}
T_{u, \sigma}C\left(\frac{1}{\alpha}+2\right)K_1(K_g^2 +R^2+R^4) \leq 1 \quad \mbox{ and } \quad T_{u, \sigma}\left(\frac{1}{\epsilon} +2  \right)K_2\left( 1+ R^2 +R^4  \right) \leq 1, 	
\end{subeqnarray*}
which holds for all $0\leq T \leq T_{u, \sigma}$. Therefore using \eqref{IED4}, \eqref{IEDSF1}, \eqref{B1}, and \eqref{B2} we conclude that 
\begin{equation*}
\|(u, \sigma)\|_{X(T)}^2 \leq R^2 
\end{equation*}
for all $0\leq T \leq T_{u, \sigma}$. This completes the proof.  
\end{proof}

Next, we show that $\Phi$ is a contraction mapping for all sufficiently small $T$. 
\begin{proposition} \label{Contraction}
	There exists a $T_c>0$, depending only on the initial data such that $\Phi$ is a contraction on the ball $B_R$. 
\end{proposition}
\begin{proof} 
Recalling that the initial data $(u_0, \sigma_0)$ are fixed, we take $(\tilde{u}_1, \tilde{\sigma}_1), (\tilde{u}_2, \tilde{\sigma}_2) \in B_R$, and consider the corresponding solutions $(u_1, \sigma_1), (u_2, \sigma_2)\in B_R$. First consider the corresponding vorticities $\omega_1$ and $\omega_2$ and notice their difference $\bar{\omega}=\omega_2-\omega_1$ satisfies the equation
\begin{eqnarray}
\partial_t \bar{\omega} - \alpha \Delta \bar{\omega} = G, \quad \mbox{ with } \quad \bar{\omega}(0) = 0, 
\end{eqnarray}
where $G = f_2 - f_1$. Using the fact that we can write
\begin{eqnarray*}
\Omega(J\tilde{u}_2, \tilde{u}_2)-\Omega(J\tilde{u}_1, \tilde{u}_1) = \Omega(J\tilde{u}_2, \tilde{u}_2- \tilde{u}_1) + \Omega(J\tilde{u}_2-J\tilde{u}_1, \tilde{u}_1), 
\end{eqnarray*}
we can write $G$ as follows: 
\begin{eqnarray*}
G &=& \beta \nabla \times \text{div}(J\tilde{\sigma}_2-J\tilde{\sigma}_1) + \left(   J\tilde{u}_2 - J\tilde{u}_1 \right)\nabla \tilde{\omega}_1 + J\tilde{u}_2 \left( \nabla \tilde{\omega}_2 - \nabla \tilde{\omega}_1  \right) \\
&& +\Omega(J\tilde{u}_2, \tilde{u}_2- \tilde{u}_1) + \Omega(J\tilde{u}_2-J\tilde{u}_1, \tilde{u}_1).
\end{eqnarray*}
Notice that the forcing terms cancel. Similarly, let  $\bar{\sigma} = \sigma_2-\sigma_1$. Then, it satisfies the following equation: 
\begin{equation*}
\partial_t \bar{\sigma} - \epsilon \Delta \bar{\sigma} = \tilde{G}, \quad \mbox{ with } \quad \bar{\sigma}(0) = 0, 
\end{equation*}
where $\tilde{G} = \tilde{f}_2 - \tilde{f}_1$, which after some manipulations, can be written as follows: \begin{eqnarray*} 
\tilde{G} &=& \gamma(\tilde{\sigma}_2-\tilde{\sigma}_1) + \left( J\tilde{u}_2 \cdot \nabla \right) \left( \tilde{\sigma}_2-\tilde{\sigma}_1 \right) + \left[ J \left( \tilde{u}_2 - \tilde{u}_1 \right) \cdot \nabla  \right] \tilde{\sigma}_1 + \left[ \nabla \left( J \left( \tilde{u}_1 - \tilde{u}_2 \right)  \right) \right] \tilde{\sigma}_1 \\
&& +\left( \nabla J \tilde{u}_2 \right)\left( \tilde{\sigma}_1 - \tilde{\sigma}_2   \right) + \tilde{\sigma}_1 \left[  \nabla \left(  J \left(  \tilde{u}_1 - \tilde{u}_2 \right) \right)    \right]^T   + \left(  \tilde{\sigma}_1 - \tilde{\sigma}_2   \right) \left(  \nabla J\tilde{u}_2 \right)^T.
\end{eqnarray*}
Thus, we see that there exists a constant $C_R$ depending only on $R$ such that 
\begin{subeqnarray*}
\|G\|_{\mathbb{L}^2_3}^2 &\leq& C_R\| (\tilde{u}_2-\tilde{u}_1,  \tilde{\sigma}_2-\tilde{\sigma}_1)\|_X^2 \\ 
\|\tilde{G}\|_{\mathbb{H}^1_9(\mathbb{T})}^2 &\leq& C_R \| (\tilde{u}_2-\tilde{u}_1, \tilde{\sigma}_2-\tilde{\sigma}_1)\|_X^2.
\end{subeqnarray*}
Now choose some $0 < \theta<1$ and a $T_c>0$ such that $T_c\leq T_{u, \sigma}$ and such that 
\begin{equation*}
C T_c \left( \frac{1}{\alpha} + 2 \right)  C_R\leq \frac{\theta^2}{2} \quad \mbox{ and } \quad 
T_c \left(  \frac{1}{\epsilon} + 2   \right) C_R\leq \frac{\theta^2}{2}.
\end{equation*}
Using \eqref{IED4} and \eqref{IEDSF1} we conclude
\begin{equation*}
\max_{0\leq t \leq T_c} \|u_2(t)-u_1(t) \|_{\overbar{\mathbb{H}}^2_2} \leq \frac{\theta^2}{2} \|(\tilde{u}_2-\tilde{u}_1, \tilde{\sigma}_2-\tilde{\sigma}_1) \|_X^2  
\end{equation*}
and 
\begin{equation*}
\max_{0\leq t \leq T_c} \|\sigma_2(t)-\sigma_1(t) \|_{\mathbb{H}^2_4} \leq \frac{\theta^2}{2} \| (\tilde{u}_2-\tilde{u}_1, \tilde{\sigma}_2-\tilde{\sigma}_1) \|_X^2  
\end{equation*}
so putting these together we obtain
\begin{equation*}
\|(u_2-u_1, \sigma_2-\sigma_1)\|_X^2 \leq \theta^2 \| (\tilde{u}_2-\tilde{u}_1, \tilde{\sigma}_2-\tilde{\sigma}_1) \|_X^2   
\end{equation*}
for $X=X(T_c)$. Therefore, $\Phi$ is a contraction mapping. This completes the proof. 
\end{proof} 

\begin{proposition} \label{ShortTimeExistence}
	The system of equations \eqref{RegularizedVorticityForm} has a unique solution $(u, \sigma)$ which exists for $0\leq t \leq T_c$, where $T_c>0$ depends only on the initial data.
\end{proposition}
\begin{proof} Since $\Phi$ is a contraction mapping on $B_R \subset X(T_c)$, it has a unique fixed point $(u, \sigma)$ inside of  $B_R \subset X(T_c)$ which is the solution of our system. In fact, this is the only solution in all of $X(T_c)$, not just in $B_R$. Suppose there was some other solution, call it $(u_2, \sigma_2) \in X(T_c)$ with the same initial data $(u_0, \sigma_0)$. By continuity, there would have to be some time $T_2>0$ with $T_2 \leq T_c$ such that $(u_2, \sigma_2) \in B_R \subset X(T_2)$. But then, applying the contraction mapping for this shorter time would yield a unique solution. Then, both $(u, \sigma), (u_2, \sigma_2)\in B_R\subset X(T_2)$ would be solutions for $0\leq t \leq T_2$, meaning that $(u(t), \sigma(t))=(u_2(t), \sigma_2(t))$ for $0\leq t \leq T_2$. We can extend this argument to the entire interval $0\leq t \leq T_c$ to obtain the uniqueness of the solution on the established interval of existence. Furthermore, as we saw, our choice of $T_c$ ultimately depended on $R$ whose definition depends only on the choice of the initial data, and the elliptic constant $C$ which is independent of $T$, completing the proof. 
\end{proof} 

\begin{proposition} \label{extension}
   There exists a number $T^*>0$ (potentially with $T^*=\infty$) such that a unique solution to \eqref{RegularizedOldroydB} for $\epsilon>0$ exists for $0\leq t < T^*$ with $u(t)\in C([0, T]; \overbar{\mathbb{H}}^2_3)$, $u'(t) \in L^2(0, T; \overbar{\mathbb{H}}^1_3)$, $\sigma(t) \in C([0, T]; \mathbb{H}^2_9)$, and $\sigma'(t) \in L^2(0, T; \mathbb{H}^1_9)$ for all $0\leq T < T^*$.
\end{proposition}
\begin{proof} We have shown that a solution in the appropriate spaces exists for $0\leq t \leq T_c$. Now we can use $(u(T_c), \sigma(T_c))$ as new initial data and repeat the procedure, to extend our solution to a longer time. This extension is unique as we saw in Proposition \ref{ShortTimeExistence}, and using that proposition we can always extend the solution to a longer time interval. This completes the proof.  
\end{proof}  

We remark that there is a possibility that as the solution grows, the amount of time we can extend by becomes smaller, leading to a barrier time $T^*$ beyond which we cannot extend. In fact, we shall show that this can be ruled out later on by assuming that $\sigma_0$ is symmetric and positive semi-definite. 

\subsection{Short Term Stability and Positivity of $\sigma(t)$}
\label{EnergyEstimate}
In order to establish global existence we need to establish the energy estimate \eqref{k1k2}. In order to do this we need to establish the postive semi-definiteness of $\sigma(t)$, which we will show is true assuming $\sigma_0$ is symmetric and positive semi-definite. In \cite{constantinnote} this is proven by finding equations satisfied by the eigenvalues of $\sigma$, and showing by the maximum principle that if they start out non-negative, they remain (weakly) non-negative. While this works well in two dimensions, the algebra becomes messy in three dimensions. Thus we use an alternative proof strategy. Since we have better regularity, we can establish non-negativity point-wise instead of almost everywhere. The basic idea is to approximate $\sigma$
 by a smooth tensor, prove the positive semi-definiteness of this smooth tensor, and then show that this smooth tensor converges to $\sigma$. We begin with establishing the short term stability of $\sigma$.   
\begin{proposition} \label{ShortTermStability}
Suppose that $(u_1, \sigma_1)$ and $(u_2, \sigma_2)$ are two solutions to \eqref{RegularizedVorticityForm} with two different bounded initial states: 
\begin{equation*}
\|(u_1(0), \sigma_1(0)) \|_Y^2, \quad \|(u_2(0), \sigma_2(0))\|_Y^2 \leq \mathcal{M}.
\end{equation*}
Then there exists a time $T_\mathcal{M}>0$ depending only on $\mathcal{M}$ such that for $0\leq t \leq T_\mathcal{M}$
\begin{equation*}
    \max_{0\leq t \leq T_\mathcal{M}} \left(\Vert u_2(t) - u_1(t) \Vert_{\overbar{\mathbb{H}}^2_3}^2 + \Vert \sigma_2(t) - \sigma_1(t) \Vert_{\mathbb{H}^2_9}^2 \right) \leq C_\mathcal{M} \left(\Vert u_2(0) - u_1(0) \Vert_{\overbar{\mathbb{H}}^2_3}^2 + \Vert \sigma_2(0) - \sigma_1(0) \Vert_{\mathbb{H}^2_9} \right)
\end{equation*}
where the constant $C_\mathcal{M}$ depends only on $\mathcal{M}$.
\end{proposition}
\begin{proof} 
As we saw in the proofs of Propositions \ref{SelfMapOfBall}, \ref{Contraction}, and \ref{ShortTimeExistence} if we define 
\begin{equation*}
R^2=12C\mathcal{M}+3\mathcal{M}+4 
\end{equation*}
then there exists a time $T_\mathcal{M}$ depending only on $\mathcal{M}$ such that for any initial data smaller than $\mathcal{M}$ we will have a solution, which will furthermore be contained in $B_R \subset X(T_\mathcal{M})$. Compare with \eqref{Rformula}. Now, we define $\bar{u}=u_2-u_1$, $\bar{\omega}=\omega_2-\omega_1$, and $\bar{\sigma}=\sigma_2-\sigma_1$. We have that
\begin{subeqnarray}\label{usefulEstimates}
&& |J\bar{\sigma}| \leq K \|\bar{\sigma}\|_{\mathbb{L}^2_9}, \quad |J\bar{u}| \leq K \|\bar{u}\|_{\mathbb{L}^2_3}, \quad \|\bar{u}\|_{\mathbb{L}^2_3}^2 \leq C \|\bar{\omega}\|^2_{\mathbb{L}^2_3}, \\
&& \|\Omega (Ju_1, \bar{u})\|_{\mathbb{L}^2_3}^2 \leq 81 K \|u_1\|^2_{\mathbb{L}^2_3} \|\bar{u} \|_{\mathbb{H}^1_3}^2, \quad \|\Omega (J\bar{u}, u_2 ) \|_{\mathbb{L}^2_3}^2 \leq 81 K \|\bar{u}\|^2_{\mathbb{L}^2_3}\|u_2 \|_{\mathbb{H}^1_3}^2 
\end{subeqnarray}
where the constant $K$ only depends on the choice of $\eta$ in the definition of $J$, as we mentioned previously. Also notice the terms involving $u_1$ and $u_2$ can be bounded in terms of $R$ and hence $\mathcal{M}$. Here $C$ again is the same elliptic constant as before. 

By taking the equation for $\omega_2$ and the equation for $\omega_1$ and subtracting them, we obtain the equation for $\bar{\omega}$ which takes the form 
\begin{equation*}
\partial_t \bar{\omega} - \alpha \Delta \bar{\omega} = \beta \left(  \nabla \times \text{div} J \bar{\sigma}  \right) - (Ju_2\cdot \nabla)\bar{\omega} - (J\bar{u}\cdot \nabla) \omega_1 + \Omega(J\bar{u}, u_2) + \Omega(Ju_1, \bar{u}). 
\end{equation*}
We now take the inner product with $\bar{\omega}$ and integrate in space. Since we can estimate $\bar{u}$ in terms of $\bar{\omega}$, and the terms involving $u_1, \omega_1, \text{etc},$ can be estimated in terms of $\mathcal{M}$ we can easily obtain an estimate of the form: 
\begin{equation}\label{OmegaBarL2Derivative} \frac{1}{2}\frac{d}{dt} \| \bar{\omega} \|_{\mathbb{L}^2_3}^2 + \alpha \|\nabla \bar{\omega} \|_{\mathbb{L}^2_3}^2 
\leq C_1 \left( \|\bar{\omega} \|_{\mathbb{L}^2_3}^2 + \|\bar{\sigma} \|_{\mathbb{L}^2_9}^2 \right)
\end{equation}
where the constant $C_1$ depends only on $\mathcal{M}$. We can similarly obtain an equation for $\bar{\sigma}$ which gives
\begin{equation*} 
\partial_t \bar{\sigma} - \epsilon \Delta \bar{\sigma} = -(Ju_2 \cdot \nabla ) \bar{\sigma} +(J\bar{u} \cdot \nabla) \sigma_1 + (\nabla Ju_2)\bar{\sigma} + (\nabla J\bar{u})\sigma_1 + \bar{\sigma} ( \nabla Ju_2)^T + \sigma_1 (\nabla J\bar{u})^T - \gamma \bar{\sigma}
\end{equation*}
and upon taking the Frobenius inner product with $\bar{\sigma}$ and integrating in space we obtain
\begin{equation}\label{SigmaBarL2Derivative}
\frac{1}{2} \frac{d}{dt} \|\bar{\sigma} \|_{\mathbb{L}^2_9}^2 + \epsilon \|\nabla \bar{\sigma}\|_{\mathbb{L}^2_{27}}^2 \leq C_2 \left( \|\bar{\omega} \|_{\mathbb{L}^2_3}^2 +  \| \bar{\sigma}\|_{\mathbb{L}^2_9}^2 \right), 
\end{equation}
where the constant $C_2$ depends only on $\mathcal{M}$.

Adding \eqref{OmegaBarL2Derivative} and \eqref{SigmaBarL2Derivative}, multiplying by $2$, and defining $C_3=2C_1+2C_2$ we obtain the inequality
\begin{equation*}
\frac{d}{dt} \left( \|\bar{\omega}(t) \|_{\mathbb{L}^2_3}^2 + \|\bar{\sigma}(t) \|_{\mathbb{L}^2_9}^2 \right) \leq C_3 
\left( \|\bar{\omega}(t) \|_{\mathbb{L}^2_3}^2 + \| \bar{\sigma}(t) \|_{\mathbb{L}^2_9}^2 \right)
\end{equation*}
and so applying Gronwall's inequality 
\begin{equation*}
\|\bar{\omega}(t)\|_{\mathbb{L}^2_3}^2 + \| \bar{\sigma}(t)\|_{\mathbb{L}^2_9}^2 \leq e^{C_3 t} \left( \|\bar{\omega}(0)\|_{\mathbb{L}^2_3}^2 + \| \bar{\sigma}(0)\|_{\mathbb{L}^2_9}^2 \right),
\end{equation*}
which holds for $0\leq t \leq T_\mathcal{M}$. So we can estimate
\begin{equation*}
\|\bar{\omega}(t)\|_{\mathbb{L}^2_3}^2 + \| \bar{\sigma}(t) \|_{\mathbb{L}^2_9}^2 \leq e^{C_3 T_\mathcal{M}} \left( \| \bar{\omega}(0) \|_{\mathbb{L}^2_3}^2 + \|\bar{\sigma}(0) \|_{\mathbb{L}^2_9}^2 \right)\coloneqq C_4 \left( \| \bar{\omega}(0) \|_{\mathbb{L}^2_3}^2 + \| \bar{\sigma}(0) \|_{\mathbb{L}^2_9}^2 \right), 
\end{equation*}
for $0\leq t \leq T_M$.

Next, notice that the equation for $\bar{\omega}$ can be put in the form 
\begin{equation}\label{eqomegabar2}
\partial_t \bar{\omega} - \alpha \Delta \bar{\omega} + (Ju_2\cdot \nabla)\bar{\omega}= \beta \left(  \nabla \times \text{div} J \bar{\sigma}  \right)  - (J\bar{u}\cdot \nabla) \omega_1 + \Omega(J\bar{u}, u_2) + \Omega(Ju_1, \bar{u}) 
\end{equation}
where for $0\leq t \leq T_\mathcal{M}$ we have $|Ju_2|\leq KR$. Therefore, there exists a constant $C_{T_\mathcal{M}}$ such that
\begin{equation}\label{MaxOmegaH1}
\max_{0\leq t \leq T_\mathcal{M}} \|\bar{\omega}(t) \|_{\mathbb{H}^1_3}^2 \leq C_{T_\mathcal{M}} \left( \| \bar{\omega}(0) \|_{\mathbb{H}^1_3}^2 + \|f_\omega \|_{L^2(0, T_\mathcal{M}; \mathbb{L}^2_3)}^2 \right), 
\end{equation}
where $f_\omega$ is the right hand side of \eqref{eqomegabar2}. But by our previous propositions and estimates we have $L^2$ control over the right hand side and we can estimate
\begin{eqnarray*}
\|f_\omega \|_{L^2(0, T_\mathcal{M}; \mathbb{L}^2_3)}^2 &\leq& \int_0^{T_\mathcal{M}} C_5 \left( \|\bar{u}(t) \|_{\mathbb{H}^1_3}^2 + \|\bar{\sigma}(t) \|_{\mathbb{L}^2_9}^2 \right) dt  \leq \int_0^{T_\mathcal{M}} C_5 \left( C \|\bar{\omega}(t) \|_{\mathbb{L}^2_3}^2 + \|\bar{\sigma}(t) \|_{\mathbb{L}^2_9}^2 \right) dt \\
&\leq& \int_0^{T_\mathcal{M}} C_6 \left( \|\bar{\omega}(t) \|_{\mathbb{L}^2_3}^2 + \|\bar{\sigma}(t) \|_{\mathbb{L}^2_9}^2 \right) dt \leq T_\mathcal{M}  C_6 C_4 \left( \|\bar{\omega}(0) \|_{\mathbb{L}^2_3}^2 + \|\bar{\sigma}(0) \|_{\mathbb{L}^2_9}^2 \right), 
\end{eqnarray*}
where the constants depend only on $\mathcal{M}$ and the $K$ which comes from the definition of $J$. Therefore, plugging this estimate into \eqref{MaxOmegaH1} we see there is a constant $C_7$ depending only on $K$ and $\mathcal{M}$ such that
\begin{equation*}
    \max_{0\leq t \leq T_\mathcal{M}} \Vert \bar{\omega} (t) \Vert_{\mathbb{H}^1_3}^2 \leq C_7 \left(  \Vert \bar{\omega}(0) \Vert_{\mathbb{H}^1_3}^2 + \Vert \bar{\sigma}(0) \Vert_{\mathbb{L}^2_9}^2   \right).
\end{equation*}
Similarly, we can write the equation for $\bar{\sigma}$ as 
\begin{equation}\label{eqsigmabar2}
\partial_t \bar{\sigma} - \epsilon \Delta \bar{\sigma} +(Ju_2 \cdot \nabla ) \bar{\sigma} =  (J\bar{u} \cdot \nabla) \sigma_1 + (\nabla Ju_2)\bar{\sigma} + (\nabla J\bar{u})\sigma_1 + \bar{\sigma} ( \nabla Ju_2)^T + \sigma_1 (\nabla J\bar{u})^T - \gamma \bar{\sigma}, 
\end{equation}
where again we have $|Ju_2|\leq KR$ for $0\leq t \leq T_\mathcal{M}$. So again there is a constant $\tilde{C}_{T_\mathcal{M}}$ such that
\begin{equation*}
\max_{0\leq t \leq T_\mathcal{M}} \|\bar{\sigma} (t) \|_{\mathbb{H}^1_9}^2 \leq \tilde{C}_{T_\mathcal{M}} \left( \|\bar{\sigma}(0)\|_{\mathbb{H}^1_9}^2 + \| f_\sigma \|_{L^2(0, T_\mathcal{M}; \mathbb{L}^2_9)}^2 \right), 
\end{equation*}
where $f_\sigma$ is the right hand side of \eqref{eqsigmabar2}. Using the same arguments that we made for $\bar{\omega}$ we can find a constant such that
\begin{equation*}
\max_{0\leq t\leq T_\mathcal{M}} \|\bar{\sigma}(t) \|_{\mathbb{H}^1_9}^2 \leq C_8 \left( \| \bar{\omega}(0) \|_{\mathbb{L}^2_3}^2 + \| \bar{\sigma}(0) \|_{\mathbb{H}^1_9}^2 \right). 
\end{equation*}
Now, we want to obtain an $H^2$ bound for the components of $\bar{\sigma}$.
The idea is to take the equation \eqref{eqsigmabar2} and differentiate it with respect to $x^i$. This then gives the equation for $\partial_i \bar{\sigma}$ which can be put into the same form as \eqref{eqsigmabar2}
\begin{eqnarray*}
 \partial_t (\partial_i \bar{\sigma} ) - \epsilon \Delta (\partial_i \bar{\sigma} ) +(Ju_2 \cdot \nabla ) \partial_i \bar{\sigma} &=& -(\partial_i Ju_2 \cdot \nabla) \bar{\sigma} + \partial_i J \bar{u} \cdot \nabla \sigma_i + J\bar{u} \cdot \nabla \partial_i \sigma_1 + (\nabla \partial_i Ju_2) \bar{\sigma} \\&& \quad+ (\nabla Ju_2) \partial_i \bar{\sigma} 
+ (\nabla \partial_i J\bar{u})\sigma_1 + (\nabla J\bar{u}) \partial_i \sigma_1   + \partial_i \bar{\sigma} (\nabla J u_2 )^T \\ \quad \quad && \quad \quad + \bar{\sigma} (\nabla \partial_i Ju_2)^T  + \partial_i \sigma_1 (\nabla J\bar{u})^T + \sigma_1(\nabla \partial_i J\bar{u})^T - \gamma \partial_i\bar{\sigma} 
\end{eqnarray*}
and notice using our previous estimates we can estimate the norm of the right hand side and each term will have at least one factor of $\left(\Vert \bar{\omega}(0) \Vert_{\mathbb{H}^1_3}^2 + \Vert \bar{\sigma}(0)\Vert_{\mathbb{H}^1_9}^2 \right)$. Thus we obtain a bound of the form
\begin{equation*}
\max_{0\leq t \leq T_\mathcal{M}} \|\partial_i \bar{\sigma} \|_{\mathbb{H}^1_9}^2 \leq \tilde{C}_{T_\mathcal{M}}\left[ \|\partial_i \bar{\sigma}(0)\|_{\mathbb{H}^1_9}^2 +  C_9 \left( \|\bar{\omega}(0) \|_{\mathbb{H}^1_3}^2 + \| \bar{\sigma}(0)\|_{\mathbb{H}^1_9}^2 \right) \right].
\end{equation*}
We can repeat this for each $i$ to conclude the partial derivatives of the components of $\bar{\sigma}$ are in $H^1$, meaning that the components of $\bar{\sigma}$ are in $H^2$, and adding up all the estimates we obtain
\begin{equation*}
\|\bar{\sigma}\|_{\mathbb{H}^2_9}^2 \leq C_{10} \left( \|\bar{\omega}(0)\|_{\mathbb{H}^1_3}^2 + \|\bar{\sigma}(0) \|_{\mathbb{H}^2_9} ^2\right). 
\end{equation*}
Since we can estimate $\|\bar{u} \|_{\mathbb{H}^2_3}$ in terms of $\|\bar{\omega} \|_{\mathbb{H}^1_3}$ and $\|\bar{\omega}(0) \|_{\mathbb{H}^1_3}$ in terms of $\|\bar{u}(0) \|_{\mathbb{H}^2_3}$, we can put together all of our estimates to obtain
\begin{equation*}
\max_{0\leq t \leq T_\mathcal{M}} \left( \|\bar{u}(t) \|_{\overbar{\mathbb{H}}^2_3}^2 + \| \bar{\sigma}(t) \|_{\mathbb{H}^2_9}^2 \right) \leq C_\mathcal{M} \left(\|\bar{u}(0) \|_{\overbar{\mathbb{H}}^2_3}^2 + \| \bar{\sigma}(0) \|_{\mathbb{H}^2_9} \right), 
\end{equation*}
for all $0\leq t \leq T_\mathcal{M}$, for some constant $C_\mathcal{M}$ which can be expressed in terms of our previous constants, and thus only depends on $\mathcal{M}$, completing the proof. 
\end{proof} 
We will need one more proposition before we can prove positivity.

\begin{proposition} \label{Smoothness}
    If $\sigma_0$ is smooth, then $\sigma(t)$ remains smooth in the spatial variable for all $0\leq t < T^*$ where $T^*$ is the time given in Proposition \ref{extension}.
\end{proposition}
\begin{proof} We consider the solution to \eqref{RegularizedVorticityForm} for $0\leq t \leq T_c$ with the components of $u$ and $\sigma$ being in $H^2$. We have that $\Vert u \Vert_{\mathbb{L}^2_3}^2\leq R$ on this time interval and so, due to the regularization, we have complete point-wise control over $Ju$ and all its derivatives. Looking at the equation
\begin{equation}\label{equationinductive}
\partial_t \sigma  -\epsilon \Delta \sigma + \left(Ju \cdot \nabla \right) \sigma = \delta I  + \left(\nabla Ju\right) \sigma +\sigma \left(\nabla Ju\right)^T - \gamma \sigma , 
\end{equation}
subject to $\sigma(0)=\sigma_0$, we see that if 
\begin{equation}\label{EstimateSigmaHn}
\max_{0\leq t \leq T_c} \Vert \sigma(t) \Vert_{\mathbb{H}^n_9} \leq C_n, 
\end{equation}
then the right hand side of the equation is in $\mathbb{H}^{n}_9$. Therefore, letting $\lambda$ be a multi-index of length $n$, as long as $\sigma_0$ is smooth, we can apply $D^\lambda$ to \eqref{equationinductive}, to get an initial value problem for $D^{\lambda} \sigma$. We can then conclude that $D^\lambda \sigma$ has components in $H^1$, so $\sigma$ has components in $H^{n+1}$, and we can estimate 
\begin{equation*}
\max_{0\leq t \leq T_c} \Vert \sigma(t) \Vert_{\mathbb{H}^{n+1}_9} \leq C_{n+1}, 
\end{equation*}
as we did with our earlier estimates. Here $C_{n+1}$ depends on $C_n, R,$ and $T_c$. The base case where $n=2$ has already been established. Therefore, by induction we can prove the estimate \eqref{EstimateSigmaHn} for any $n$, which by the Sobolev embedding theorem means that $\sigma(t)$ is smooth in space for each $0\leq t \leq T_c$. We repeat this argument every time we extend our solution to conclude that $\sigma(t)$ is smooth for all $0\leq t < T^*$.
\end{proof} 

\begin{proposition} \label{PositivityOfSigma}
	If $\sigma_0$ is symmetric and positive semi-definite, then so is $\sigma(t)$ for all $0\leq t < T^*$ where $T^*$ is the time given in Proposition \ref{extension}.
\end{proposition} 
\begin{proof} The proof is based on the sketch given in Theorem 3.3 in \cite{chow2006hamilton}. There, the authors suggest how to make their argument precise. However, they assume that the components of their tensor are $C^2$ which is not the case for us. We will first prove the proposition assuming $\sigma=\sigma(x, t)$ is smooth in $x$ for each $t$ and then use this to establish the case when the components of $\sigma$ are only in $H^2$.  

First, the symmetry of $\sigma(t)$ easily follows from the equations assuming $\sigma_0$ is symmetric. There is no difficulty here. 

Next, assume that our $\sigma$ is smooth in $x$ for each $t$ and $\sigma(0)$ is positive semi-definite. Further assume that $u\in C([0, T]; \overbar{\mathbb{H}}^2_3)$ for each $T<T^*$. Therefore, for each $T$ we can estimate
\begin{equation*}
|\nabla Ju| \leq K \|u\|_{C([0, T]; \overbar{\mathbb{H}}^2_3)}
\end{equation*}
where $K>0$ is as before. Next we define the constant 
\begin{equation*}
S = 1+2K \Vert u \Vert_{C([0, T]; \overbar{\mathbb{H}}^2_3)}
\end{equation*}
and now we define the tensor
\begin{equation*}
A(t) = \sigma(t) + \varepsilon e^{St} I, 
\end{equation*}
where $\varepsilon>0$. Notice that since $\sigma(0)$ is positive semidefinite, $A(0)$ is strictly positive. Now using the equation for $\sigma$ we can write down the differential equation satisfied by $A(t)$ which is
\begin{eqnarray}\label{EquationA} 
\partial_t A &=& \epsilon \Delta A - (Ju\cdot \nabla)A + (\nabla Ju) A + A (\nabla Ju )^T - \varepsilon e^{St}(\nabla Ju) - \varepsilon e^{St}(\nabla Ju)^T \nonumber \\ 
&& \quad - \gamma A + \gamma \varepsilon e^{St}I + \delta I + S\varepsilon e^{St}I  
\end{eqnarray}
as can be easily checked. Notice that by our choice of $S$ we have that for any vector $W\neq 0$
\begin{equation}\label{estimateA}
    S\varepsilon e^{St}I_{ij}W^iW^j - \varepsilon e^{St}(\nabla Ju)_{ij}W^iW^j - \varepsilon e^{St}(\nabla Ju)^T_{ij}W^iW^j\geq \varepsilon e^{St}|W|^2>0
\end{equation}
where we employed the summation convention. Now, let $(x_1, t_1)$ be a point with the smallest time coordinate $t_1 \leq T$ at which $A$ has $0$ as an eigenvalue and let $V$ be the corresponding eigenvector. Notice that $t_1>0$ since $A(0)$ is strictly positive. Now contract $A_{ij}$ with the constant vector field equal to $V$ everywhere to obtain the quantity
\begin{equation*}
A_{ij}V^iV^j=A_{ij}(x, t)V^iV^j
\end{equation*}
which is positive for all $(x, t)$ with $0\leq t < t_1$. Hence, we must have that
\begin{equation}\label{contradiction}
\partial_t\left(A_{ij}V^iV^j\right) \leq 0 \quad \text{at} \quad (x_1, t_1)
\end{equation}
On the other hand, contracting \eqref{EquationA} with the constant vector field $V$ and freely moving indices up and down we obtain
\begin{eqnarray*}
\partial_t \left( A_{ij}V^i V^j\right) &=& \epsilon \Delta \left( A_{ij}V^i V^j\right) -(Ju \cdot \nabla ) \left( A_{ij}V^i V^j\right) + ( \nabla Ju)_i^{\; \; p}A^p_{\; \; j}V^iV^j + A_{ip}(\nabla Ju)_j^{\; \; p} V^i V^j \\ 
&& \quad -\varepsilon e^{St} (\nabla Ju)_{ij}V^i V^j  -\varepsilon e^{St} (\nabla Ju)_{ji}V^i V^j - \gamma A_{ij}V^i V^j +\gamma \varepsilon e^{St}|V|^2 \\
&& \quad + \delta|V|^2 + S\varepsilon e^{St}|V|^2, 
\end{eqnarray*}
and upon evaluating this at $(x_1, t_1)$ and using the fact that in that case 
\begin{equation*}
A_{ij}V^i=A_{ij}V^j=0
\end{equation*}
we get
\begin{eqnarray*}
\partial_t \left( A_{ij}V^i V^j\right) &=& \epsilon \Delta \left( A_{ij}V^i V^j\right) -(Ju \cdot \nabla ) \left( A_{ij}V^i V^j\right) - \varepsilon e^{St} (\nabla Ju)_{ij}V^i V^j \\ && \quad -\varepsilon e^{St} (\nabla Ju)_{ji}V^i V^j  +\gamma \varepsilon e^{St}|V|^2 + \delta|V|^2 + S\varepsilon e^{St}|V|^2.
\end{eqnarray*}
Since for $t_1$ we have that $A_{ij}V^iV^j$ achieves a spatial minimum at $x_1$, we can apply the second derivative test so that $\nabla (A_{ij}V^iV^j)=0$ and $\Delta (A_{ij}V^iV^j)\geq 0$ at $(x_1, t_1)$ to conclude
\begin{equation*}
\partial_t (A_{ij}V^iV^j) \geq \varepsilon e^{St}|V|^2 > 0 \quad \text{at} \quad (x_1, t_1).
\end{equation*}
However, this contradicts \eqref{contradiction} which means there is no point $(x, t)$ at which $0$ is an eigenvalue of $A$, and so $A$ is strictly positive for all $(x, t)$ with $0\leq t \leq T$. Since this holds for any $\varepsilon$, we conclude that $\sigma$ is positive semi-definite for all $(x, t)$ with $0\leq t \leq T$. Since $T<T^*$ can be chosen arbitrarily, we obtain the proposition in the case of smooth $\sigma$.

Next we want to extend this to our general case, where the components of $\sigma(t)$ are only in $H^2$ for each fixed $t$. Since the components are in $H^2$ and our dimensions is $n=3$ we conclude that the components are in $C^{0, \Gamma}$, for some $0<\Gamma<1$ and so it actually makes sense to talk about the positivity of $\sigma$ pointwise, as opposed to almost everywhere.

First we will prove the positivity of $\sigma$ on the interval $0\leq t \leq T_c$ of Theorem \ref{ShortTimeExistence}. Take $\sigma_0$ and approximate it in $\mathbb{H}^2_9$ by a smooth positive definite $\sigma_{0, \varepsilon}$ such that
\begin{equation*}
\|\sigma_0 - \sigma_{0, \epsilon} \|_{\mathbb{H}^2_9} \leq \varepsilon. 
\end{equation*}
We can do this since $\sigma_0$ is assumed to be positive semi-definite.
We can thus estimate
\begin{equation*}
\|(u_0, \sigma_{0, \varepsilon}) \|^2_Y \leq \| (u_0, \sigma_{0}) \|^2_Y + 2 \varepsilon \| (u_0, \sigma_{0}) \|_Y + \varepsilon^2, 
\end{equation*}
and so take the $\mathcal{M}$ in Proposition \ref{ShortTermStability} to be 
\begin{equation*}
\mathcal{M} = \|(u_0, \sigma_{0})\|^2_Y + 2 \varepsilon \|(u_0, \sigma_{0}) \|_Y + \varepsilon^2.
\end{equation*}
Denote the solution with initial data $(u_0, \sigma_{0, \varepsilon})$ by $(u_\varepsilon(t), \sigma_\varepsilon(t))$ which is guaranteed to exist on $0\leq t \leq T_\mathcal{M} < T_c$. Since $\mathcal{M}$ depends on $\varepsilon$ we write $T_\varepsilon$ instead of $T_\mathcal{M}$ to make this dependence explicit. Moreover, we see that as $\varepsilon \searrow 0$ then $T_\varepsilon \nearrow T_c$. By Proposition \ref{Smoothness} we have that $\sigma_\varepsilon (t)$ is smooth, and therefore positive on $0\leq t \leq T_\varepsilon$. 

Now take any $T<T_c$. Eventually for sufficiently small $\varepsilon<1$ we have $T<T_\varepsilon$. Since the initial velocity is the same, by Proposition \ref{ShortTermStability} we can estimate
\begin{equation*}
\max_{0\leq t \leq T} \Vert \sigma(t) - \sigma_{\varepsilon} (t) \Vert_{\mathbb{H}^2_9}^2 \leq C_{T_1} \left(   \Vert \sigma_0 - \sigma_{0,\varepsilon} \Vert_{\mathbb{H}^2_9}^2   \right)
\end{equation*}
where the constant $C_{T_1}$ corresponding to $\varepsilon=1$ works for all smaller $\varepsilon$. Therefore, taking the limit as $\varepsilon \rightarrow 0$, we have that the components of $\sigma_{\varepsilon}(t)$ converge to the components of $\sigma(t)$ in $H^2$ for each $t$. Therefore, they also converge in $C^{0, \Gamma}$. Hence, taking an arbitrary vector $W$ we obtain
\begin{equation*}
\lim_{\varepsilon \rightarrow 0} (\sigma_\varepsilon (t) )_{ij}W^i W^j = (\sigma (t))_{ij} W^i W^j \geq 0
\end{equation*}
for all $0\leq t \leq T < T_c$. Since we can choose $T$ arbitrarily as long as $T<T_c$, we conclude that in fact
\begin{equation*}
(\sigma(t))_{ij} W^i W^j \geq 0
\end{equation*}
for all $0\leq t \leq T_c$. Repeating this argument for each extension of our solution, we conclude that $\sigma(t)$ is positive semi-definite for all $0\leq t < T^*$. 
\end{proof}

\subsection{A priori estimates} Next, we will establish some apriori estimates for $\omega(t)$, $u(t)$, and $\sigma(t)$. 
\begin{proposition} \label{SigmaL2Apriori}
Given $u_0$ and a positive semi-definite $\sigma_0$, the solution $\sigma(t)$ of \eqref{RegularizedOldroydB} satisfies the estimate
\begin{equation*}
\|\sigma(t) \|_{\mathbb{L}_{9}}^2 \leq R_1(t), \mbox{ for } 0\leq t < T^*, 
\end{equation*}
where 
\begin{equation*}
R_1(t) = e^{2\gamma t + 2 \delta \sqrt{|\mathbb{T}|} t + 4KE_2 t + 2K E_1 t^2}\left(   \Vert \sigma_0 \Vert^2_{\mathbb{L}^2_{9}} +2\sqrt{|\mathbb{T}|}\delta t \right), 
\end{equation*}
where $E_1$ and $E_2$ are constants appearing on the right hand side of \eqref{k1k2} and $K$ is a constant appearing in \eqref{VBound}, which depends only on the definition of $J$.
\end{proposition}
\begin{proof} Recall the constant $K$ satisfies
\begin{equation*}
|Ju|+|\nabla Ju| \leq K \|u\|_{\mathbb{L}_9}.
\end{equation*}
Therefore, we take the equation for $\sigma$ and take its Frobenius innder product with $\sigma$, and integrate over the spatial variable. This yields the inequality
\begin{eqnarray*}
\frac{1}{2} \frac{d}{dt} \|\sigma(t)\|_{\mathbb{L}^2_{9}} + \epsilon \int_{\mathbb{T}  } |\nabla \sigma |^2 dx &\leq& \delta \int_{\mathbb{T}} |\sigma| dx + \gamma \int_{\mathbb{T}} |\sigma|^2 dx + 2 \int_{\mathbb{T}} |\nabla Ju| |\sigma|^2 dx  \\ 
&\leq& \delta \sqrt{|\mathbb{T}|}\Vert \sigma \Vert_{\mathbb{L}^2_{9} } + \left( \gamma + 2K(E_1 t + E_2 )  \right) \Vert \sigma \Vert_{\mathbb{L}^2_{9} }^2 \\
&\leq& \delta \sqrt{|\mathbb{T}|}\left(1 + \Vert \sigma \Vert_{\mathbb{L}^2_{9} }^2 \right) + \left( \gamma + 2K(E_1 t + E_2 )  \right) \Vert \sigma \Vert_{\mathbb{L}^2_{9} }^2 \\ &\leq& \delta \sqrt{|\mathbb{T}|} + \left( \delta  \sqrt{|\mathbb{T}|} + \gamma + 2K(E_1 t + E_2 )  \right) \Vert \sigma \Vert_{\mathbb{L}^2_{9} }^2, 
\end{eqnarray*}
where H\"{o}lder's inequality is used for the second inequality. We also notice that $\int_{\mathbb{T}} \left[(Ju\cdot \nabla )\sigma \right]\cdot \sigma dx = 0$ as can be checked by integrating by parts. Taking the inequality and multiplying by $2$ we obtain
\begin{equation*}
\frac{d}{dt} \|\sigma \|_{\mathbb{L}^2_{9} }^2 \leq 2\delta \sqrt{|\mathbb{T}|} + \left( 2\delta  \sqrt{|\mathbb{T}|} + 2\gamma + 4K(E_1 t + E_2 )  \right) \|\sigma \|_{\mathbb{L}^2_{9} }^2
\end{equation*}
and using Gronwall's inequality we obtain
\begin{equation*}
\|\sigma(t)\|_{\mathbb{L}_{9}}^2 \leq e^{2\gamma t + 2 \delta \sqrt{|\mathbb{T}|} t + 4KE_2 t + 2K E_1 t^2}\left(   \Vert \sigma_0 \Vert^{2}_{\mathbb{L}^2_{9}} +2\sqrt{|\mathbb{T}|}\delta t \right). 
\end{equation*}
This completes the proof. 
\end{proof} 

We can obtain a similar estimate for $\Vert \nabla \sigma \Vert^2_{\mathbb{L}^2_{27}}$.

\begin{proposition} \label{NablaSigmaL2Apriori}
Given $u_0$ and a positive semi-definite $\sigma_0$, for any fixed $T>0$ with $T<T^*$ the solution $\sigma(t)$ of \eqref{RegularizedOldroydB} with $\epsilon>0$ satisfies the estimate
\begin{equation*}
\|\nabla \sigma(t)\|_{\mathbb{L}^2_{27}}^2 \leq R_2(t), \quad 0 \leq t \leq T, 
\end{equation*}
where 
\begin{equation*}
R_2(t) = e^{ \frac{2M^2}{\epsilon}t} \left( \|\nabla \sigma_0 \|_{\mathbb{L}^2_{27}}^2 + \frac{2\delta^2}{\epsilon} |\mathbb{T}| t + \frac{2\gamma^2 + 8M^2}{\epsilon} \int_0^t  R_1(s)ds  \right), 
\end{equation*}
where the constant $M=K(E_1T + E_2)$ depends on $T$, and the constants $K, E_1,$ and $E_2$ are those appearing in Proposition \ref{SigmaL2Apriori}.
\end{proposition}
\begin{proof} 
First we fix any $T>0$ with $T<T^*$. Now we define the constant
\begin{equation*}
M = K(E_1T + E_2)
\end{equation*}
We take the Frobenius inner product of the equation for $\sigma$ with $\sigma'$ and integrate in space to obtain
\begin{subeqnarray}\label{split}
\|\sigma' \|_{\mathbb{L}^2_9}^2 + \frac{\epsilon}{2} \frac{d}{dt} \|\nabla \sigma \|_{\mathbb{L}^2_{27}}^2 &=& \int_{\mathbb{T}} \left(\left[  (-Ju\cdot \nabla )\sigma \right]  + \delta I - \gamma \sigma  +\left[ (\nabla Ju) \sigma  \right]  + \left[ \sigma (\nabla Ju)^T   \right]   \right) \colon \sigma' dx \\
&\leq& M \int_{\mathbb{T}} |\nabla \sigma| | \sigma'| dx + \gamma \int_{\mathbb{T}} |\sigma| |\sigma'| dx +2M\int_{\mathbb{T}} |\sigma| |\sigma'| dx +  \int_{\mathbb{T}} \delta |\sigma'| dx
\end{subeqnarray}
and now the trick is to use Cauchy's inequality with $\varepsilon$ to obtain a contribution of $\frac{1}{4} \|\sigma' \|_{\mathbb{L}^2_9}^2$ from each term to cancel out the $\| \sigma' 
\|_{\mathbb{L}^2_9}^2$ on the left hand side. In this case we obtain  
\begin{equation*}
\frac{\epsilon}{2} \frac{d}{dt} \|\nabla \sigma \|_{\mathbb{L}^2_{27}}^2 \leq M^2 \|\nabla \sigma \|_{\mathbb{L}_{27}^2}^2 + \gamma^2 \| \sigma \|_{\mathbb{L}^2_9}^2 + 4M^2 \| \sigma \|_{\mathbb{L}^2_9}^2 + \delta^2 |\mathbb{T}|
\end{equation*}
and upon multiplying by $2/\epsilon$ and using Proposition \ref{SigmaL2Apriori} we obtain
\begin{equation*}
\frac{d}{dt} \|\nabla \sigma \|_{\mathbb{L}^2_{27}}^2 \leq \frac{2M^2}{\epsilon} \| \nabla \sigma \|_{\mathbb{L}^2_{27}}^2 + \frac{2\delta^2}{\epsilon} |\mathbb{T}| + \frac{2\gamma^2 + 8M^2}{\epsilon}R_1(t) 
\end{equation*}
and upon using Gronwall's inequality we get 
\begin{equation*}
\|\nabla \sigma(t) \|_{\mathbb{L}^2_{27}}^2 \leq e^{ \frac{2M^2}{\epsilon}t} \left( \|\nabla \sigma_0 \|_{\mathbb{L}^2_{27}}^2 + \frac{2\delta^2}{\epsilon} |\mathbb{T}| t + \frac{2\gamma^2 + 8M^2}{\epsilon} \int_0^t  R_1(s)ds  \right), 
\end{equation*}
which holds for $0\leq t \leq T$, completing the proof. 
\end{proof}

Next we will obtain similar estimates for $\omega$. 

\begin{proposition} \label{OmegaL2Estimate}
Given $u_0$ and a positive semi-definite $\sigma_0$, the solution $\omega(t)$ of \eqref{RegularizedVorticityForm} satisfies the estimate
\begin{equation*}
\|\omega(t) \|_{\mathbb{L}^2_3}^2 \leq R_3(t), 
\end{equation*}
where 
\begin{equation}\label{r3} 
R_3(t) := e^{  (\beta + 4 + 162K E_2C_1C_E)t + 81K E_1 t^2            } \left(  \|\omega_0 \|_{\mathbb{L}^2_3} + 2 K_g^2 |\mathbb{T}| t + 9K^2 \beta |\mathbb{T}| \int_0^t R_1(s) \; ds \right). 
\end{equation}
Here $K$ is a constant only depending on the choice of $\eta$ in the definition of $J$ and $R_1(t)$ is the function appearing in Proposition \ref{SigmaL2Apriori}.
\end{proposition}
\begin{proof}  
By definition, the constant $K$ satisfies
\begin{equation*}
|\nabla \times \text{div} J \sigma| \leq 3K \|\sigma \|_{\mathbb{L}^2_3}.
\end{equation*}
Again this constant depends only on the choice of $\eta$ in the definition of $J$. Now, take the equation for $\omega$ and take the inner product with $\omega$ to obtain
\begin{subeqnarray*}
\frac{1}{2} \frac{d}{dt} \|\omega \|_{\mathbb{L}^2_3}^2 + \alpha \|\nabla \omega \|_{\mathbb{L}^2_9}^2 &=& \beta \int_{\mathbb{T}} \left( \nabla \times \text{div} J \sigma \right)\cdot \omega \;  dx - \int_{\mathbb{T}} \Omega(Ju, u) \cdot \omega \; dx + \int_{\mathbb{T}} g\cdot \omega dx \\ 
&\leq& \beta  \int_{\mathbb{T}} |\nabla \times \text{div} J \sigma| |\omega| \; dx + \int_{\mathbb{T}} |\Omega(Ju, u)||\omega| \; dx +  \int_{\mathbb{T}} |g| |\omega| dx \\
&\leq& \frac{\beta}{2}\int_{\mathbb{T}} |\nabla \times \text{div} J\sigma|^2 \; dx + \left(\frac{\beta}{2} + 2 \right) \int_{\mathbb{T}} |\omega|^2 \; dx + \int_{\mathbb{T}} |\Omega(Ju, u)|^2 \; dx + \int_{\mathbb{T}} |g|^2 \; dx \\
&\leq& \frac{9K^2\beta}{2} |\mathbb{T}| \Vert \sigma \Vert_{\mathbb{L}^2_9}^2 + K_g^2 |\mathbb{T}| + \frac{\beta + 4}{2} \Vert \omega \Vert_{\mathbb{L}^2_3}^2 + 81K \Vert u \Vert_{\mathbb{L}^2_3}^2 \Vert u \Vert_{\mathbb{H}^1_3}^2.
\end{subeqnarray*}
However, since we can estimate
\begin{equation*}
\|u\|_{\mathbb{H}^1_3}^2 \leq C_1 \|F\|_{\mathbb{H}^2_3}^2 \leq C_1 C_E \|\omega \|_{\mathbb{L}^2_3}^2 =\tilde{C} \|\omega\|_{\mathbb{L}^2_3}^2. 
\end{equation*}
we obtain the inequality
\begin{equation*}
\frac{d}{dt} \|\omega \|_{L^2}^2 \leq \left[ 9K^2\beta |\mathbb{T}| R_1(t) +2K_g^2 |\mathbb{T}|\right] + \left[ \beta + 4 + 162 K(E_1t+E_2) C_1 C_E   \right] \|\omega \|_{\mathbb{L}^2_3}^2, 
\end{equation*}
and applying Gronwall's inequality we obtain
\begin{equation*}
\|\omega(t)\|_{\mathbb{L}^2_3}^2 \leq e^{  (\beta + 4 + 162KE_2C_1C_E)t + 81KE_1 t^2            } \left( \|\omega_0\|_{\mathbb{L}^2_3} + 2K_g^2 |\mathbb{T}| t + 9K^2 \beta |\mathbb{T}| \int_0^t R_1(s) \; ds \right), 
\end{equation*}
which completes the proof. 
\end{proof} 

\begin{proposition} \label{NablaOmegaL2Estimate}
Given $u_0$ and a positive semi-definite $\sigma_0$, for any fixed $T > 0$, for which the solition exists in $[0,T]$, the solution $\omega(t)$ of \eqref{RegularizedVorticityForm} satisfies the estimate
\begin{equation*}
\|\nabla \omega(t)\|_{\mathbb{L}^2_9}^2 \leq R_4(t) \quad \forall t \in [0,T], 
\end{equation*}
where 
\begin{equation*}
R_4(t) = e^{\frac{2M^2}{\alpha}t}\left( \|\nabla \omega_0 \|_{\mathbb{L}^2_9}^2 + \frac{2t}{\alpha} \left[ 9\beta^2 K^2 |\mathbb{T}| R_1(T) + 81M C_1 C_E R_3(T) + K_g^2 |\mathbb{T}| \right]  \right), 
\end{equation*}
where $M$ is the constant defined in Proposition \ref{NablaSigmaL2Apriori} and thus only depends on $T$ and the initial data. 
\end{proposition}
\begin{proof} For fixed $T>0$, we take the equation for $\omega$ and take the inner product with $\omega'$. Then for $0\leq t \leq T$, this gives us
\begin{eqnarray*}
\|\omega'\|_{\mathbb{L}^2_3}^2 + \frac{\alpha}{2} \frac{d}{dt} \| \nabla \omega \|_{\mathbb{L}^2_9}^2 &=& - \int_{\mathbb{T}} \left( (Ju \cdot \nabla ) \omega  \right) \omega' + \beta \int_{\mathbb{T}} \left( \nabla \times \text{div} J \sigma \right)\cdot \omega' \;  dx - \int_{\mathbb{T}} \Omega(Ju, u) \cdot \omega' \; dx + \int_{\mathbb{T}} g\cdot \omega' dx \\
&\leq& \int_{\mathbb{T}} |Ju||\nabla \omega| |\omega'| \; dx + \beta  \int_{\mathbb{T}} |\nabla \times \text{div} J \sigma| |\omega'| \; dx + \int_{\mathbb{T}} |\Omega(Ju, u)||\omega'| \; dx +  \int_{\mathbb{T}} |g| |\omega'| dx \\
&\leq& M \int_{\mathbb{T}} |\nabla \omega| |\omega'| \; dx + \beta  \int_{\mathbb{T}} |\nabla \times \text{div} J \sigma| |\omega'| \; dx + \int_{\mathbb{T}} |\Omega(Ju, u)||\omega'| \; dx +  \int_{\mathbb{T}} |g| |\omega'| dx\\ 
&\leq& M^2 \int_{\mathbb{T}} |\nabla \omega |^2 \; dx + \beta^2 \int_{\mathbb{T}} |\nabla \times \text{div} J\sigma|^2 \; dx + \int_{\mathbb{T}} |\Omega(Ju, u)|^2 \; dx + \int_{\mathbb{T}} |g|^2 \; dx + \|\omega' \|_{\mathbb{L}^2_3}^2, 
\end{eqnarray*}
where we used Cauchy's inequality with $\varepsilon$ to extract exactly $\frac{1}{4} \|\omega' \|_{\mathbb{L}^2_3}^2$ from each term. Thus using the energy estimate we obtain 
\begin{eqnarray*}
\frac{\alpha}{2} \frac{d}{dt} \|\nabla \omega \|_{\mathbb{L}^2_9}^2 &\leq& M^2 \|\nabla \omega \|_{\mathbb{L}^2_9}^2 + \beta^2 \int_{\mathbb{T}} |\nabla \times \text{div} J\sigma|^2 \; dx + \int_{\mathbb{T}} |\Omega(Ju, u)|^2 \; dx + \int_{\mathbb{T}} |g|^2 \; dx \\
&\leq& M^2 \|\nabla \omega \|_{\mathbb{L}^2_9}^2 +
9\beta^2 K^2 |\mathbb{T}| \|\sigma \|_{\mathbb{L}^2_3}^2 + 81K \|u\|_{\mathbb{L}^2_3}^2 \|u\|_{\mathbb{H}^1_3}^2 + K_g^2 |\mathbb{T}| \\
&\leq& M^2 \|\nabla \omega \|_{\mathbb{L}^2_9}^2 +
9\beta^2 K^2 |\mathbb{T}| \|\sigma \|_{\mathbb{L}^2_3}^2 + 81M C_1 C_E \|\omega \|_{\mathbb{L}^2_3}^2 + K_g^2 |\mathbb{T}| \\
&\leq& 9\beta^2 K^2 |\mathbb{T}| R_1(t) + 81M C_1 C_E R_3(t) + K_g^2 |\mathbb{T}| +  M^2 \|\nabla \omega \|_{\mathbb{L}^2_9}^2 \\
&\leq& 9\beta^2 K^2 |\mathbb{T}| R_1(T) + 81M C_1 C_E R_3(T) + K_g^2 |\mathbb{T}| +  M^2 \|\nabla \omega \|_{\mathbb{L}^2_9}^2, 
\end{eqnarray*}
and so we obtain
\begin{equation*}
\frac{d}{dt} \|\nabla \omega \|_{L^2}^2 \leq \frac{2}{\alpha} \left(   9\beta^2 K^2 |\mathbb{T}| R_1(T) + 81M C_1 C_E R_3(T) + K_g^2 |\mathbb{T}| +  M^2 \|\nabla \omega \|_{\mathbb{L}^2_9}^2 \right)
\end{equation*}
which holds for all $0\leq t \leq T$. Now, applying Gronwall's inequality we finally get
\begin{equation*}
\|\nabla \omega(t) \|_{\mathbb{L}^2_9}^2 \leq  e^{\frac{2M^2}{\alpha}t}\left( \|\nabla \omega_0 \|_{\mathbb{L}^2_9}^2 + \frac{2t}{\alpha} \left[   9\beta^2 K^2 |\mathbb{T}| R_1(T) + 81M C_1 C_E R_3(T) + K_g^2 |\mathbb{T}| \right]  \right), 
\end{equation*}
completing the proof.
\end{proof} 

\subsection{Global Existence}

To establish global existence we begin with the following proposition.
\begin{proposition}\label{BlowUpCriterion}
If the time $T^*$ appearing in Proposition \ref{extension} is finite (that is, the solution cannot be extended past $T^*$), then necessarily
\begin{equation*}
\|(u, \sigma)\|_X^2 \rightarrow \infty
\end{equation*}
as $T\nearrow T^*$.
\end{proposition}
\begin{proof} 
Suppose $T^*<\infty$ is the finite time beyond which we can't extend and $\Vert (u, \sigma) \Vert_X^2 <B$ for all $T<T^*$. Take a sequence $T_k \nearrow T^*$. Then by Theorem \ref{ShortTimeExistence} there exists a $T>0$ depending only on $B$ such that the solution can be extended to $T_k+T$. For large enough $k$ we have $T_k+T>T^*$ contradicting the definition of $T^*$. 
\end{proof}

The point is that if we can show  $	\Vert (u, \sigma) \Vert_X^2$ remains finite for all finite $T$ then our solution is global, i.e. $T^*=\infty$. This, combined with the energy estimates allows us to obtain global existence of the solution.

\begin{proof}[\textbf{Proof of Theorem \ref{GlobalExistenceTheorem1}}] 
By Proposition \ref{BlowUpCriterion} we need to show that $\Vert u(t) \Vert_{\overbar{\mathbb{H}}^2_3 }$ and $\Vert \sigma(t) \Vert_{\mathbb{H}^2_9}$ are both finite for all finite $t$. Since we have 
\begin{equation*} 
\|u(t)\|_{\overbar{\mathbb{H}}^2_3 }^2 \leq \tilde{C} \|\omega(t) \|_{\mathbb{H}^1_3 }^2. 
\end{equation*}
It is now enough to show that the right hand side is finite for all finite time. However, since $\sigma_0$ is positive semi-definite, Propositions \ref{OmegaL2Estimate} and \ref{NablaOmegaL2Estimate} both hold, which means that the right hand side is in fact finite for all finite time. 

Similarly, since Propositions \ref{SigmaL2Apriori} and \ref{NablaSigmaL2Apriori} both hold we have that $\Vert  \sigma(t)  \Vert_{ \mathbb{H}^1_9  }^2 $ is finite for all finite $t$. Now, fixing any $T>0$ we look at
\begin{equation*}
\partial_t \sigma  -\epsilon \Delta \sigma + \left(Ju \cdot \nabla \right) \sigma = \delta I  + \left(\nabla Ju\right) \sigma +\sigma \left(\nabla Ju\right)^T - \gamma \sigma , 
\end{equation*}
and differentiate it with respect to $x^i$. This gives us an equation for the components $v_i=\partial_i \sigma$. The resulting equation can be put into the form 
\begin{equation*}
    \partial_t v_i - \epsilon \Delta v_i +(Ju \cdot \nabla) v_i = G
\end{equation*}
where 
\begin{equation*}
    \max_{0\leq t \leq T} \Vert G \Vert_{\mathbb{L}^2_9}^2
\end{equation*}
can be estimated directly in terms of our estimates for $\|u(t)\|_{\overbar{\mathbb{H}}^2_3 }^2$ and $\|\sigma(t)\|_{ \mathbb{H}^1_9}^2$. Thus we can obtain an explicit estimate for 
\begin{equation*}
    \max_{0\leq t \leq T} \Vert v_i \Vert_{\mathbb{H}^1_9}^2
\end{equation*}
and so we can conclude
\begin{equation*}
\max_{0\leq t \leq T}\|\sigma(t)\|_{\mathbb{H}^2_9}^2
\end{equation*}
is finite. Since the $T$ can be chosen arbitrarily, this completes the proof. 
\end{proof} 

We can now also prove Theorem \ref{GlobalExistenceTheorem1HigherRegularity}.

\begin{proof}[\textbf{Proof of Theorem \ref{GlobalExistenceTheorem1HigherRegularity}}]
The proof follows by induction. First, due to the regularization and the energy estimate, we have complete point-wise control over $Ju$ and all its derivatives. Looking at the equation
\begin{align} \label{InductiveEquationk} \begin{split}
\partial_t \sigma  -\epsilon \Delta \sigma + \left(Ju \cdot \nabla \right) \sigma = \delta I  + \left(\nabla Ju\right) \sigma +\sigma \left(\nabla Ju\right)^T - \gamma \sigma \\
 \sigma (0) =  \sigma_0,  \end{split}
\end{align}
we see that if we have 
\begin{equation} \label{inductiveCondition}
\max_{0\leq t \leq T} \Vert \sigma(t) \Vert_{\mathbb{H}^i_9} \leq C_i, 
\end{equation}
then the right hand side of the equation is in $\mathbb{H}^{i}_9$. Let $\lambda$ be a multi-index of length $i$. We can apply $D^\lambda$ to \eqref{InductiveEquationk} to obtain the initial value problem
\begin{equation*}
    \partial_t D^\lambda \sigma  -\epsilon \Delta D^\lambda \sigma + \left(Ju \cdot \nabla \right) D^\lambda \sigma = \mathcal{G}
\end{equation*}
where $\mathcal{G}$ is a sum of terms which are the products of smooth functions (coming from the derivatives of $Ju$) and the derivatives of $\sigma$ up to order $i$. Then, using \eqref{inductiveCondition} we conclude that we can bound $\Vert \mathcal{G} \Vert_{\mathbb{L}^2_9}$ in terms of $C_i$. Now, as long as $i<k$, then $D^\lambda \sigma_0 \in \mathbb{H}^1_9$ and so using of the form \eqref{IEDS1} we obtain that
\begin{equation*}
    \max_{0\leq t \leq T} \Vert D^\lambda \sigma \Vert_{\mathbb{H}^1_9} \leq C
\end{equation*}
for some constant $C$. Thus we can bound all of the partial derivatives of $\sigma$ to conclude that there is some constant $C_{i+1}$ such that
\begin{equation}
    \max_{0\leq t \leq T} \Vert \sigma(t) \Vert_{\mathbb{H}^{i+1}_9} \leq C_{i+1}.
\end{equation}
Moreover, using \eqref{IEDS2} we can conclude that $\sigma'\in L^2(0, T; \mathbb{H}^{i-1}_9)$.  

We can apply the exact same idea to the vorticity equation, and now the result follows by induction. Finally, by the Sobolev embedding theorem, if $k\geq 4$. then the components are in $C^{k-2}$ giving us classical solutions.
\end{proof}

\subsection{Global Stability of Solutions}
In this section, we investigate the stability of our solutions. 
\begin{proposition}[$L^2$ Stability] \label{L2Stability}
Let two sets of initial data $(u_{0, 1}, \sigma_{0,1})$ and $(u_{0, 2}, \sigma_{0,2})$ satisfying the assumptions of Theorem \ref{GlobalExistenceTheorem1} be given. We denote the corresponding solutions to \eqref{RegularizedVorticityForm} by $(\omega_1(t),u_1(t),\sigma_1(t))$ and $(\omega_2(t),u_2(t),\sigma_2(t))$, respectively. Then, for an arbitrary $T > 0$. The following estimate holds: 
\begin{eqnarray*}
\|\omega_2(t)-\omega_1(t)\|_{\mathbb{L}^2_3}^2 + \|\sigma_2(t)-\sigma_1(t)\|_{\mathbb{L}^2_9}^2 &\leq& R_6(t) \left( \|\omega_2(0)-\omega_1(0)\|_{\mathbb{L}^2_3}^2 + \|\sigma_2(0) - \sigma_1(0)\|_{\mathbb{L}^2_9}^2\right), 
\end{eqnarray*}
where $R_6(t) = e^{\int_0^t R_5(s)ds}$, with 
\begin{subeqnarray*}
R_5(t) = 2CK\sqrt{R_4(t)}+\beta+  \beta |\mathbb{T}| K^2   +2CK \sqrt{R_2(t)} + 4CK \sqrt{R_1(t)} + 4K(E_1 t + E_2) + 2 \gamma. 
\end{subeqnarray*}
\end{proposition}
\begin{proof} 
As before, taking the equation for $\omega_2$ and the equation for $\omega_1$ and subtracting them, we obtain the equation for $\bar{\omega} = \omega_2 - \omega_1$, which takes the form 
\begin{equation*}
\partial_t \bar{\omega} - \alpha \Delta \bar{\omega} = \beta \left(  \nabla \times \text{div} J \bar{\sigma}  \right) - (Ju_2\cdot \nabla)\bar{\omega} - (J\bar{u}\cdot \nabla) \omega_1-\Omega(J\bar{u}, u_2)-\Omega(Ju_1, \bar{u}). 
\end{equation*}
We now take the inner product with $\bar{\omega}$ and integrate to get
\begin{eqnarray}\label{OmegaBarL2Derivative} 
\frac{1}{2}\frac{d}{dt} \|\bar{\omega}\|_{\mathbb{L}^2_3}^2 + \alpha\|\nabla \bar{\omega} \|_{\mathbb{L}^2_9}^2 &=& \int_{\mathbb{T}} \left( \beta \left(  \nabla \times \text{div} J \bar{\sigma} \right) - (Ju_2\cdot \nabla)\bar{\omega} - (J\bar{u}\cdot \nabla) \omega_1  -\Omega(J\bar{u}, u_2)- \Omega(Ju, \bar{u}) \right) \bar{\omega} dx \nonumber \\
&\leq& \frac{\beta}{2} \int_{\mathbb{T}} |\nabla \times \text{div} J \bar{\sigma}|^2dx + \frac{\beta}{2}\int_{\mathbb{T}} |\bar{\omega}|^2dx + \int_{\mathbb{T}} |J\bar{u}||\nabla \omega_1 | |\bar{\omega}| dx \nonumber \\
&& + \int_{\mathbb{T}} K \Vert \bar{u} \Vert_{\mathbb{L}^2_3} |\nabla u_2 | |\bar{\omega}| \; dx  + \int_{\mathbb{T}} K \Vert u_1 \Vert_{\mathbb{L}^2_3} |\nabla \bar{u} | |\bar{\omega} | \; dx  \nonumber \\ 
&\leq& \frac{\beta}{2}|\mathbb{T}| K^2 \Vert \bar{\sigma} \Vert_{\mathbb{L}^2_9}^2 + \frac{\beta}{2} \Vert \bar{\omega} \Vert_{\mathbb{L}^2_3}^2 +CK\Vert \bar{\omega} \Vert_{\mathbb{L}^2_3} \int_{\mathbb{T}} |\nabla \omega_1 | |\bar{\omega}| dx \nonumber \\
&& + K \Vert \bar{u} \Vert_{\mathbb{L}^2_3} \Vert \nabla u_2 \Vert_{\mathbb{L}^2_9} \Vert \bar{\omega} \Vert_{\mathbb{L}^2_3} + K \Vert u_1 \Vert_{\mathbb{L}^2_3} \Vert \nabla \bar{u} \Vert_{\mathbb{L}^2_3} \Vert \bar{\omega} \Vert_{\mathbb{L}^2_3} \nonumber  \\
&\leq& \frac{\beta}{2} |\mathbb{T}| K^2 \Vert \bar{\sigma} \Vert_{\mathbb{L}^2_9}^2 + \frac{\beta}{2} \Vert \bar{\omega} \Vert_{\mathbb{L}^2_9}^2 +CK\Vert \nabla \omega_1 \Vert_{\mathbb{L}^2_9}\Vert \bar{\omega} \Vert_{\mathbb{L}^2_3}^2 \nonumber \\
&& + K \tilde{C} \Vert \omega_2 \Vert_{\mathbb{L}^2_3} \Vert \bar{\omega} \Vert_{\mathbb{L}^2_3}^2 + K \tilde{C} \Vert \omega_1 \Vert_{\mathbb{L}^2_3} \Vert \bar{\omega} \Vert_{\mathbb{L}^2_3}^2 \nonumber \\
&\leq& \left( CK\sqrt{R_4(t)}+ \frac{\beta}{2} + 2K\tilde{C} \sqrt{R_3(t)}     \right) \Vert \bar{\omega} \Vert_{\mathbb{L}^2_3}^2 + \frac{\beta}{2}|\mathbb{T}| K^2 \Vert \bar{\sigma} \Vert_{\mathbb{L}^2_9}^2. 
\end{eqnarray}
We can similarly obtain an equation for $\bar{\sigma}$ which gives
\begin{equation*} 
\partial_t \bar{\sigma} - \epsilon \Delta \bar{\sigma} = -(Ju_2 \cdot \nabla ) \bar{\sigma} +(J\bar{u} \cdot \nabla) \sigma_1 + (\nabla Ju_2)\bar{\sigma} + (\nabla J\bar{u})\sigma_1 + \bar{\sigma} ( \nabla Ju_2)^T + \sigma_1 (\nabla J\bar{u})^T - \gamma \bar{\sigma}
\end{equation*}
and upon multiplying by $\bar{\sigma}$ and integrating in space we obtain
\begin{eqnarray}\label{SigmaBarL2Derivative}
\frac{1}{2} \frac{d}{dt} \|\bar{\sigma}\|_{\mathbb{L}^2_9}^2 + \epsilon \|\nabla \bar{\sigma} \|_{\mathbb{L}^2_{27}}^2 &\leq& \int_{\mathbb{T}} \left( |J\bar{u}||\nabla \sigma_1||\bar{\sigma}| + 2 |\nabla Ju_2||\bar{\sigma}|^2 + 2 |\nabla J\bar{u}||\sigma_1||\bar{\sigma}| + \gamma |\bar{\sigma}|^2 \right)dx \nonumber \\
&\leq& CK\Vert \bar\omega \Vert_{\mathbb{L}^2_3}\int_{\mathbb{T}}|\nabla \sigma_1| |\bar{\sigma}| dx + 2K(E_1t + E_2) \int_{\mathbb{T}} |\bar{\sigma}|^2 dx \nonumber \\ 
&& + 2CK\Vert \bar{\omega}\Vert_{\mathbb{L}^2_3} \int_{\mathbb{T}} |\sigma_1||\bar{\sigma} |dx + \gamma \int_{\mathbb{T}} |\bar{\sigma}|^2 dx \nonumber \\
&\leq& CK\Vert \bar{\omega} \Vert_{\mathbb{L}^2_3} \Vert \nabla \sigma_1 \Vert_{\mathbb{L}^2_9} \Vert \bar{\sigma} \Vert_{\mathbb{L}^2_9} +2K(E_1 t + E_2) \Vert \bar{\sigma} \Vert_{\mathbb{L}^2_9}^2 \nonumber \\
&& + 2CK\Vert \bar{\omega} \Vert_{\mathbb{L}^2_3} \Vert \sigma_1 \Vert_{\mathbb{L}^2_9} \Vert \bar{\sigma} \Vert_{\mathbb{L}^2_9} + \gamma \Vert \bar{\sigma} \Vert_{\mathbb{L}^2_9}^2 \nonumber \\
&\leq& CK \sqrt{R_2(t)} \Vert \bar{\omega} \Vert_{\mathbb{L}^2_3}  \Vert \bar{\sigma} \Vert_{\mathbb{L}^2_9} +2K(E_1 t + E_2) \Vert \bar{\sigma} \Vert_{\mathbb{L}^2_9}^2 \nonumber \\
&& + 2CK \sqrt{R_1(t)}  \Vert \bar{\omega} \Vert_{\mathbb{L}^2_3}  \Vert \bar{\sigma} \Vert_{\mathbb{L}^2_9} + \gamma \Vert \bar{\sigma} \Vert_{\mathbb{L}^2_9}^2 \nonumber \\
&\leq& \frac{1}{2}\left(  CK \sqrt{R_2(t)} + 2CK \sqrt{R_1(t)}    \right)\Vert \bar{\omega} \Vert_{\mathbb{L}^2_3}^2 \nonumber \\ 
&&  + \frac{1}{2}\left(  CK \sqrt{R_2(t)} + 2CK \sqrt{R_1(t)} + 4K(E_1 t + E_2) + 2\gamma \right) \Vert \bar{\sigma} \Vert_{\mathbb{L}^2_9}^2.
\end{eqnarray}
Adding \eqref{OmegaBarL2Derivative} and \eqref{SigmaBarL2Derivative} and multiplying by $2$, we obtain the inequality
\begin{equation*}
\frac{d}{dt} \left( \|\bar{\omega}(t)\|_{\mathbb{L}^2_3}^2 + \|\bar{\sigma}(t) \|_{\mathbb{L}^2_9}^2 \right) \leq R_5(t) \left( \|\bar{\omega}(t)\|_{\mathbb{L}^2_3}^2 + \| \bar{\sigma}(t) \|_{\mathbb{L}^2_9}^2 \right), 
\end{equation*}
where 
\begin{equation*}
R_5(t) = 2CK\sqrt{R_4(t)}+\beta+ 2K\tilde{C} \sqrt{R_3(t)} + \beta |\mathbb{T}| K^2   +2CK \sqrt{R_2(t)} + 4CK \sqrt{R_1(t)} + 4K(E_1 t + E_2) + 2\gamma
\end{equation*}
and so applying Gronwall's inequality 
\begin{equation*}
\|\bar{\omega}(t)\|_{\mathbb{L}^2_3}^2 + \|\bar{\sigma}(t)\|_{\mathbb{L}^2_9}^2 \leq 
e^{\int_0^t R_5(s)ds} \left( \|\bar{\omega}_0\|_{\mathbb{L}^2_3}^2 + \|\bar{\sigma}_0 \|_{\mathbb{L}^2_9}^2\right), 
\end{equation*}
which holds for $0\leq t \leq T$, completing the proof. 
\end{proof}

We now obtain higher order regularity for $\bar{\omega}$.
\begin{proposition} \label{OmegaBarH1Prop}
Fix some $T>0$. Then there exists a constant $C_T$ depending only on $T$ such that
\begin{equation*}
\|\bar{\omega} (t) \|_{\mathbb{H}^1_3}^2 \leq C_T R_7(t) \left( \|\bar{\omega}_0 \|_{\mathbb{H}^1_3}^2 + \|\bar{\sigma}_0\|_{\mathbb{L}^2_9}^2\right), 
\end{equation*}
for all $0\leq t \leq T$,  where
\begin{equation*}
R_7(t) = 1 + t \left( 2\beta K |\mathbb{T}| +2CK R_4(T) \right)R_6(T).
\end{equation*}
\end{proposition}
\begin{proof} 
Notice that the equation for $\bar{\omega}$ can be put in the form 
\begin{equation}\label{eqomegabar}
\partial_t \bar{\omega} - \alpha \Delta \bar{\omega} + (Ju_2\cdot \nabla)\bar{\omega}= \beta \left(  \nabla \times \text{div} J \bar{\sigma}  \right)  - (J\bar{u}\cdot \nabla) \omega_1 -\Omega(J\bar{u}, u_2)-\Omega(Ju_1, \bar{u}) 
\end{equation}
where for any fixed $T$, we have control over $|Ju_2|$. Therefore, there exists a constant $C_T$ such that
\begin{equation*}
\|\bar{\omega} (t) \|_{\mathbb{H}^1_3}^2 \leq C_T \left( \|\bar{\omega}_0 \|_{\mathbb{H}^1_3}^2 + \|f \|_{L^2(0, t; \mathbb{L}^2_3)}^2 \right), 
\end{equation*}
where $f$ is the right hand side of \eqref{eqomegabar}. We have $L^2$ control over the right hand side and we can estimate: 
\begin{eqnarray*}
\|f\|_{L^2(0, t; \mathbb{L}^2_3)}^2 &\leq& \int_0^t \left( 2\beta \|\nabla \times \text{div} J \bar{\sigma}(s)\|_{\mathbb{L}^2_3}^2 + 2\| J\bar{u} \cdot \nabla \omega_1 (s) \|_{\mathbb{L}^2_3}^2 + \|\Omega(J\bar{u}, u_2)\|_{\mathbb{L}^2_3}^2 + \|\Omega(Ju_1, \bar{u}) \|_{\mathbb{L}^2_3}^2 \right)ds \\
&\leq& \int_0^t \left( 2\beta K |\mathbb{T}| \|\bar{\sigma}(s) \|_{\mathbb{L}^2_3}^2 + 2CK \|\bar{\omega} (s)\|_{\mathbb{L}^2_3}^2 \|\nabla \omega_1 (s)\|^2_{\mathbb{L}^2_9} + 81K \|\bar{u}\|_{\mathbb{L}^2_3}^2 \|u_2\|_{\mathbb{H}^1_3}^2 \right. \\
&& \left. + 81K \|u_1\|_{\mathbb{L}^2_3}^2\| \bar{u}\|_{\mathbb{H}^1_3}^2 \right) ds \\
&\leq& \int_0^t \left( 2\beta K |\mathbb{T}| \Vert \bar{\sigma}(s) \Vert_{\mathbb{L}^2_9}^2 + 2CK\Vert \bar{\omega} (s) \Vert_{\mathbb{L}^2_3}^2 \Vert \nabla \omega_1 (s) \Vert^2_{\mathbb{L}^2_9} + 81K \tilde{C}^2 \Vert \bar{\omega} (s) \Vert_{\mathbb{L}^2_3}^2 \Vert \omega_2 (s) \Vert_{\mathbb{L}^2_3}^2 \right. \\  
&& \left. + 81 K \tilde{C}^2  \Vert \omega_1 (s) \Vert_{\mathbb{L}^2_3}^2 \Vert \bar{\omega}(s) \Vert_{\mathbb{L}^2_3}^2 \right) ds \\
&\leq& t \left( 2\beta K |\mathbb{T}| +2CK R_4(T) + 162 K \tilde{C}^2 R_3(T) \right)R_6(T)\left( \Vert \bar{\omega}_0 \Vert_{\mathbb{L}^2_3}^2 + \Vert \bar{\sigma}_0 \Vert_{\mathbb{L}^2_9}^2\right),  
\end{eqnarray*}
and thus, we obtain
\begin{equation*}
\|\bar{\omega} (t)\|_{\mathbb{H}^1_3}^2 \leq C_T R_7(t) \left(    \Vert \bar{\omega}_0 \Vert_{\mathbb{H}^1_3}^2 + \Vert \bar{\sigma}_0 \Vert_{\mathbb{L}^2_3}^2     \right), 
\end{equation*}
where 
\begin{equation*} 
R_7(t) = 1 + t \left( 2\beta K |\mathbb{T}| +2CK R_4(T) + 162 K \tilde{C}^2 R_3(T) \right)R_6(T),  
\end{equation*}
which completes the proof. 
\end{proof} 
Similarly, we estimate $\bar{\sigma}$ as follows: 
\begin{proposition}
For fixed $T > 0$, there exists a constant $C_T$ depending only on $T$ such that 
\begin{equation*}
\|\bar{\sigma}\|_{\mathbb{H}^1_3}^2 \leq C_T R_8(t) \left( \|\bar{\omega}_0 \|_{\mathbb{H}^1_3}^2 + \| \bar{\sigma}_0 \|_{\mathbb{H}^1_9}^2 \right), 
\end{equation*}
where
\begin{equation*}
R_8(t) = 1+ 8CKtR_6(T)\left( 2R_1(T) + R_2(T) + 2 R_3(T) + \gamma \right).
\end{equation*}
\end{proposition}
\begin{proof} 
We can write the equation for $\bar{\sigma}$ as follows: 
\begin{equation}\label{eqsigmabar}
\partial_t \bar{\sigma} - \epsilon \Delta \bar{\sigma} +(Ju_2 \cdot \nabla ) \bar{\sigma} =  (J\bar{u} \cdot \nabla) \sigma_1 + (\nabla Ju_2)\bar{\sigma} + (\nabla J\bar{u})\sigma_1 + \bar{\sigma} ( \nabla Ju_2)^T + \sigma_1 (\nabla J\bar{u})^T - \gamma \bar{\sigma}, 
\end{equation}
where for any fixed $T$ we have control over $|Ju_2|$. Therefore, there exists a constant $C_T$ such that
\begin{equation*}
\|\bar{\sigma} (t)\|_{\mathbb{H}^1_9}^2 \leq C_T \left( \|\bar{\sigma}_0\|_{\mathbb{H}^1_9}^2 + \|f\|_{L^2(0, t; \mathbb{L}^2_9)}^2 \right), 
\end{equation*}
where $f$ is the right hand side of \eqref{eqsigmabar}. But again we have $L^2$ control over the right hand side and we can estimate: 
\begin{eqnarray*}
\|f\|_{L^2(0, t; \mathbb{L}^2_9)}^2 &\leq& \int_0^t 8\left( \|J\bar{u}\cdot \nabla \sigma_1 \|_{\mathbb{L}^2_9}^2 + 2\|(\nabla J u_2)\bar{\sigma}\|_{\mathbb{L}^2_9}^2 + 2\|(\nabla J\bar{u})\sigma_1\|_{\mathbb{L}^2_9}^2 + \gamma \| \bar{\sigma}\|_{\mathbb{L}^2_9}^2\right) ds \\
&\leq& 8 \int_0^t \left(CKR_2(T)R_6(T) + 2CKR_3(T)R_6(T)+2CKR_1(T)R_6(T) \right. \\
&& \left. + \gamma R_6(T)\right) \left( \| \bar{\omega}_0 \|_{\mathbb{H}^1_3}^2 + \|\bar{\sigma}_0 \|_{\mathbb{L}^2_9}^2 \right) ds \\ 
&\leq& 8CKtR_6(T)\left( 2R_1(T) + R_2(T) + 2 R_3(T) + \gamma \right) \left(\|\bar{\omega}_0 \|_{\mathbb{H}^1_3}^2 + \|\bar{\sigma}_0 \|_{\mathbb{L}^2_3}^2 \right), 
\end{eqnarray*}
and so
\begin{eqnarray*}
\|\bar{\sigma}\|_{\mathbb{H}^1_9}^2 &\leq& C_T \|\bar{\sigma}_0 \|_{\mathbb{H}^1_9}^2 + C_T 8CKtR_6(T)\left( 2R_1(T) + R_2(T) + 2 R_3(T) + \gamma \right) \left( \|\bar{\omega}_0 \|_{\mathbb{H}^1_3}^2 + \|\bar{\sigma}_0 \|_{\mathbb{L}^2_9}^2 \right) \\
&\leq& C_T R_8(t)  \left( \|\bar{\omega}_0 \|_{\mathbb{H}^1_3}^2 + \|\bar{\sigma}_0 \|_{\mathbb{H}^1_9}^2 \right), 
\end{eqnarray*}
where 
\begin{equation*}
R_8(t) = 1 + 8CKtR_6(T)\left( 2R_1(T) + R_2(T) + 2 R_3(T) + \gamma \right), 
\end{equation*}
completing the proof. 
\end{proof} 
Thus we have $H^1$ stability:
\begin{proposition}[$H^1$ Stability] Fix some $T>0$. We have the estimate:
\begin{equation*}
\|\bar{\omega}\|_{\mathbb{H}^1_3}^2 + \|\bar{\sigma}\|_{\mathbb{H}^1_9}^2 \leq C_T R_9 (t) \left( \|\bar{\omega}_0 \|_{\mathbb{H}^1_3}^2 + \|\bar{\sigma}_0\|_{\mathbb{H}^1_9}^2 \right), 
\end{equation*}
for all $0\leq t \leq T$ where 
\begin{equation*}
R_9(t) = R_7(t) + R_8(t)
\end{equation*}
\end{proposition}
\begin{proof} This follows immediately by adding together the estimates from the previous two propositions and noticing the obvious fact that $\|\bar{\sigma}_0\|_{\mathbb{L}^2_9}^2 \leq \|\bar{\sigma}_0 \|_{\mathbb{H}^1_9}^2$. This completes the proof. 
\end{proof} 
We now obtain a higher order estimate for $\bar{\sigma}$ in the following proposition and theorem: 
\begin{proposition}\label{prop27} 
For fixed $T > 0$, we have  
\begin{equation*}
\|\bar{\sigma}(t)\|_{\mathbb{H}^2_9}^2 \leq C_T R_{10}(t) \left( \|\bar{\omega}_0 \|_{\mathbb{H}^1_3}^2 + \| \bar{\sigma}_0 \|_{\mathbb{H}^2_9}^2 \right), 
\end{equation*}
where $R_{10}(t)$ is a continuous function on $0\leq t \leq T$, which depends on $T, K, E_1, E_2$.
\end{proposition}
\begin{proof} 
The idea is to take the equation \eqref{eqsigmabar} and differentiate it with respect to $x^i$. This then given the equation for $\partial_i \bar{\sigma}$ which can be put into the same form as \eqref{eqsigmabar}:
\begin{eqnarray*}
\partial_t (\partial_i \bar{\sigma} ) - \epsilon \Delta (\partial_i \bar{\sigma} ) +(Ju_2 \cdot \nabla ) \partial_i \bar{\sigma} &=& -(\partial_i Ju_2 \cdot \nabla) \bar{\sigma} + \partial_i J \bar{u} \cdot \nabla \sigma_i + J\bar{u} \cdot \nabla \partial_i \sigma_1 + (\nabla \partial_i Ju_2) \bar{\sigma}  \\ 
&& + (\nabla Ju_2) \partial_i \bar{\sigma} + (\nabla \partial_i J\bar{u})\sigma_1 + (\nabla J\bar{u}) \partial_i \sigma_1   + \partial_i \bar{\sigma} (\nabla J u_2 )^T \\ 
&& + \bar{\sigma} (\nabla \partial_i Ju_2)^T + \partial_i \sigma_1 (\nabla J\bar{u})^T + \sigma_1(\nabla \partial_i J\bar{u})^T - \gamma \partial_i\bar{\sigma},  
\end{eqnarray*}
and notice using our previous estimates we can estimate the norm of the right hand side and each term will have at least one factor of $\left(\|\bar{\omega}_0\|_{\mathbb{H}^1_3}^2 + \|\bar{\sigma}_0 \|_{\mathbb{H}^1_9} \right)$. Thus we obtain a bound of the form
\begin{equation*}
\|\partial_i \bar{\sigma}\|_{\mathbb{H}^1_9}^2 \leq C_T\left[\|\partial_i \bar{\sigma}_0 \|_{\mathbb{H}^1_9}^2 +  R_i(t)\left(\|\bar{\omega}_0 \|_{\mathbb{H}^1_3}^2 + \|\bar{\sigma}_0\|_{\mathbb{H}^1_9} \right) \right].
\end{equation*}
We can repeat this for each $i$ to conclude the partial derivatives of $\bar{\sigma}$ are in $H^1$, meaning that $\bar{\sigma}$ is in $H^2$. Combining all of the previous estimates, we see that there is a function $R_{10}(t)$ such that the proposition holds. This completes the proof. 
\end{proof} 
 Now we obtain $H^2$ stability for our solutions: 
\begin{theorem}[$H^2$ stability for $\epsilon >0$] \label{H2Stability}
Let $(u_{0, 1}, \sigma_{0,1})$ and $(u_{0, 2}, \sigma_{0,2})$ denote two sets of initial data which satisfy the assumptions of Theorem \eqref{GlobalExistenceTheorem1}. We denote the corresponding solutions to \eqref{RegularizedVorticityForm} by $(u_1(t),\sigma_1(t))$ and $(u_2(t),\sigma_2(t))$, respectively. For a fixed $T > 0$, there exists a continuous function $R(t)$ such that 
\begin{equation*}
\|u_2(t)-u_1(t)\|_{\overbar{\mathbb{H}}^2_3}^2 + \|\sigma_2(t) -  \sigma_1(t) \|_{\mathbb{H}^2_9}^2 \leq R(t) \left( \|u_{0, 2}- u_{0, 1} \|_{\overbar{\mathbb{H}}^2_3}^2 + \|\sigma_{0, 2} - \sigma_{0, 1} \|_{\mathbb{H}^2_9}^2 \right), 
\end{equation*}
for all $0\leq t \leq T$. The function $R(t)$ depends on the norms of the initial data, $K$, and $T$.   
\end{theorem}
\begin{proof} 
The proof follows easily from Propositions \ref{OmegaBarH1Prop} and \ref{prop27}, and the fact that $\|u\|_{\overbar{\mathbb{H}}^2_3}^2 \leq \tilde{C} \|\omega\|_{\mathbb{H}^1_3}^2$. 
\end{proof} 

\section{Global existence of solutions to the non-diffusive regularized Oldroyd-B model}\label{SectionNonDissipative}

In this section, we shall provide the proof of Theorem \ref{GlobalExistenceTheorem2}, the case that there is no $\epsilon \Delta \sigma$ term in the equation for the conformation tensor. The proof starts out the same, with the same spaces $X$ and $Y$ under consideration
We will take $\left( \tilde{u}, \tilde{\sigma} \right)\in X$ and then solve 
\begin{equation}\label{MMollifiedFormulation2}
\left \{ \begin{array}{l}
\partial_t \omega - \alpha \Delta \omega  = \beta \left( \nabla \times \text{div} J \tilde{\sigma}   \right)- (J\tilde{u} \cdot \nabla)\tilde{\omega} - \Omega(J\tilde{u}, \tilde{u}) + g	\\ 
\partial_t \sigma + \left(J\tilde{u} \cdot \nabla \right) \sigma  - \left(\nabla J\tilde{u}\right) \sigma -\sigma \left(\nabla J\tilde{u}\right)^T + \gamma \sigma = \delta I \\
\nabla \cdot u =0\, \quad \nabla \times u =\omega \, \quad \mbox{ and } \quad \int_{\mathbb{T}} u(x, t) \; dx = 0, 
\end{array}\right. 
\end{equation}
subject to the initial conditions: 
\begin{equation}
u(0) = u_0 \in \overbar{\mathbb{H}}^2_3  \quad \mbox{ with } \quad \int_{\mathbb{T}} u_0 \, dx =0 \quad \mbox{ and } \quad 
\sigma(0) =\sigma_0 \in \mathbb{H}^2_{9}.  
\end{equation}
to obtain a map $\Phi(\tilde{u}, \tilde{\sigma})=(u, \sigma)$ mapping $X$ to $X$.
The first part of the proof is the same as for Theorem \ref{GlobalExistenceTheorem1}, namely we solve for the vorticity and the corresponding velocity as in \S \ref{SolvingForVorticity} and \S \ref{SolvingForVelocity}.

The difference comes in solving for the conformation tensor, since the equation is no longer parabolic. For this purpose, an explicit formula for the conformation tensor will be used. Further, we observe that in this section it is a bit easier to think of all of our quantities as periodic in $\mathbb{R}^3$ we will do so 

\subsection{Solving for the Conformation Tensor ($\epsilon=0$)}

Given $\tilde{u}$, we wish to solve the equation for $\sigma$, 
\begin{equation}\label{sigma}
\partial_t \sigma + \left(J\tilde{u} \cdot \nabla \right) \sigma  - \left(\nabla J\tilde{u}\right) \sigma -\sigma \left(\nabla J\tilde{u}\right)^T + \gamma \sigma = \delta I, 
\end{equation}
subject to the initial condition that $\sigma(0)  = \sigma_0$. The key is to notice that the left hand side of \eqref{sigma} can be written as a material derivative. Given $\tilde{u}(x, t)\in C([0, T]; \overbar{\mathbb{H}}^2_3)$ we can think of it as a periodic vector field in $\mathbb{R}^3$. We then consider the vector field 
\begin{equation*}
v(x, t)=J\tilde{u}(x, t), 
\end{equation*}
which is then smooth in $x$ and continuous in $t$. Furthermore, since $\tilde{u}$ is periodic, so is $v$. For each $x\in \mathbb{R}^3$, we can solve the system of ODEs  
\begin{equation}\label{ODE} 
\left \{ \begin{array}{ll} 
\dot{x}(t)&= v(x(t), t) \\
x(0)&=x, 
\end{array}\right. 
\end{equation}
for $t \in [0, \tilde{T}]$ for some $0<\tilde{T}\leq T$. This $\tilde{T}$ can be chosen uniformly for all initial $x \in \mathbb{R}^3$ due to the periodicity of $v$.

Furthermore, if we have a bound of the form
\begin{equation}\label{BoundforU}
\|\tilde{u}\|_{C([0, T]; \overbar{\mathbb{H}}^2_3)} \leq r, 
\end{equation}
then $\tilde{T}$ can be chosen depending only on $r$, meaning that for different $\tilde{u}$ bounded by the same $r$, we can obtain a uniform $\tilde{T}$ (recall that the way \eqref{ODE} is solved is by using a contraction mapping, and hence why $\tilde{T}$ depends on $r$). We will denote this dependence by $\tilde{T}=\tilde{T}(r)$. Then we can obtain a solution on $[\tilde{T}, 2\tilde{T}]$ and so on, to obtain a solution on all of $[0, T]$.

This defines a flow continuous in both $x$ and $t$, which we denote by $\varphi(x, t)$. The flow has three components $\varphi=(\varphi^1, \varphi^2, \varphi^3)$, i.e., we indicate the $i$-th component by $\varphi^i$. Note that the point $x=(x^1, x^2, x^3)$ also has three components, for which $x^i(x)$ gives the $i$-th component of $x$. Furthermore, it can be shown that $\varphi(x, t)$ is actually smooth in $x$ for each $t$ and for each ${\varepsilon}>0$ satisfies 
\begin{equation}\label{A}
\max_{\substack{(x, t)\in \mathbb{R}^3 \times [0, T_{{\varepsilon}}  ] \\ |\lambda|\leq 3 }} \vert D^\lambda (\varphi^i(x, t) - x^i(x)) \vert < {\varepsilon},  
\end{equation} 
where $\lambda$ is a multi-index and  $T_{{\varepsilon}}>0$ depends only on $\varepsilon$ and the $r$ appearing on the right hand side of \eqref{BoundforU}. 
We now define the deformation tensor $F(x,t)$ by
\begin{equation*}
F_{ik}(x, t)=\frac{\partial \varphi^i}{\partial x^k}(x, t)
\end{equation*}
and denoting its transpose by $F^{\text{T}}$ it is well known \cite{leethesis2004,Linliuzhang2005} that the solution to \eqref{sigma} can be written as
\begin{equation}\label{SolSigma}
\sigma (x, t)=e^{-\gamma t} F(x, t)\sigma_0(x)F^{\text{T}}(x, t)+ \delta \int_{0}^{t} e^{-\gamma (t-s)}F(x, t)F^{-1}(x, s)F^{-\text{T}}(x, s)F^{\text{T}}(x, t)ds
\end{equation}
from which it is then easy to conclude $\sigma(x, t)\in C([0, T]; \mathbb{H}^2_9)$. This is obvious because since the flow is smooth, it is the regularity of $\sigma_0$ which determines the regularity of $\sigma$. Furthermore, from the formula it is obvious that if $\sigma_0$ is symmetric and positive semi-definite, then so is $\sigma$. Now we can establish the short-time existence.

\subsection{Short-time existence 
($\epsilon=0$)}

Now that we have obtained the solutions to the system \eqref{MMollifiedFormulation2}, we can consider the solution operator as a mapping $\Phi : X\rightarrow X$ such that
\begin{equation*}
\Phi(\tilde{u}, \tilde{\sigma}) = (u, \sigma). 
\end{equation*}
First, we will show that there is some ball of radius $R$ in $X$  for an appropriately small enough $T$.
It will be shown that the map $\Phi$ maps this ball to itself. To do so, define $U_0(t)=u_0$ and $\Sigma_0(t)=\sigma_0$ (these are just functions equal to the initial data for all $t$). Then we have $(U_0, \Sigma_0) \in X $ for all $T$.

Now, let 
\begin{equation}\label{Ball}
B_R(U_0, \Sigma_0) = \{(v, w)\in X : \Vert (v-U_0, w-\Sigma_0) \Vert_X^2 \leq R^2 \},  
\end{equation}
which is just the closed ball of radius $R$ around $(U_0, \Sigma_0)$. We shall show that there is some 
$R$ such that 
\begin{equation*}
\Phi: B_R(U_0, \Sigma_0) \rightarrow B_R(U_0, \Sigma_0), 
\end{equation*}
for some appropriate $T$. To do this we start with observing that  
\begin{equation*}
\Vert u(t) - U_0(t)  \Vert_{\overbar{\mathbb{H}}^2_3}^2 = \Vert u(t) - u_0  \Vert_{  \overbar{\mathbb{H}}^2_3}^2
\end{equation*}
which by \eqref{Uestimate} boils down to estimating $\Vert \omega(t)-\omega(0) \Vert_{   \mathbb{H}^1_3  }^2 $. 
This can be done quite easily as 
\begin{equation*}
\Vert \omega(t) - \omega(0) \Vert_{  \mathbb{H}^1_3     }^2 \leq 4\max_{0\leq \tau \leq T} \Vert \omega(\tau) \Vert_{  \mathbb{H}^1_3      }^2 \leq 4 \left(  \left(1+ e^{(1+2\alpha)T}   \right) \Vert \omega_0 \Vert_{H^{1}(\mathbb{T})}^2 + \left( \frac{1}{\alpha} + e^{(1+2\alpha)T}  \right) \Vert f \Vert_{L^2(0, T; L^2(\mathbb{T}))}^2     \right)
\end{equation*}
by Theorem \ref{KeyTheorem}. We can choose a sufficiently small $T'>0$ such that for all $T\leq T'$ we have that $e^{(1+2\alpha)T}\leq 2$ so the above reduces to 
\begin{eqnarray*}
\|\omega(t) - \omega(0) \|^2_{ \mathbb{H}^1_3  } &\leq& 12 \|\omega_0 \|_{H^{1}(\mathbb{T})}^2 + 4\left(2+\frac{1}{\alpha}\right) \|f\|_{L^2(0, T;  \mathbb{L}^2_3  )}^2 \\
&\leq& 48 \|(u_0, \sigma_0) \|_Y^2 + 4\left(2+\frac{1}{\alpha}\right) \|f\|_{L^2(0, T; \mathbb{L}^2_3  )}^2
\\  &\leq& 48 \left( \|(u_0, \sigma_0)\|_Y + 1 \right)^2 + 4\left(2+\frac{1}{\alpha}\right) \|f\|_{L^2(0, T; \mathbb{L}^2_3)}^2, 
\end{eqnarray*}
where we used \eqref{EstimateForOmega}, and so
\begin{equation*}
\max_{0\leq t \leq T} \|u(t)-U_0(t)\|_{ \overbar{\mathbb{H}}^2_3   }^2 \leq \tilde{C}\left( 48 \left( \| (u_0, \sigma_0) \|_Y + 1 \right)^2 + 4\left(2+\frac{1}{\alpha}\right) \|f\|_{L^2(0, T; \mathbb{L}^2_3  )  }^2  \right), 
\end{equation*}
by the estimate \eqref{Uestimate}. We now choose 
\begin{equation*}
R^2=100\tilde{C}\left( \Vert (u_0, \sigma_0) \Vert_Y+1\right)^2 
\end{equation*}
and notice, that for all $(\tilde{u}, \tilde{\sigma})\in B_R(U_0, \Sigma_0)$,  we have the estimate 
\begin{equation*}
\|(\tilde{u}, \tilde{\sigma}) \|_X \leq r \mbox{ with } r = \|(u_0, \sigma_0)\|_Y + R,  
\end{equation*} and also notice that by this definition, $r$ only depends on the norm of the initial data. Thus, in light of \eqref{EstimateForf}, for all $(\tilde{u}, \tilde{\sigma})\in B_R(U_0, \Sigma_0)$ we can choose a sufficiently small $T_u>0$ with $T_u\leq T'$ such that for all $0<T\leq T_u$ we have 
\begin{equation*}
4\left(2+\frac{1}{\alpha}\right) \| f \|_{L^2(0, T; \mathbb{L}^2_3    )}^2 < \left( \|(u_0, \sigma_0) \| + 1 \right)^2, 
\end{equation*}
and so
\begin{equation}\label{eq371}
\max_{0\leq t \leq T} \Vert u(t)-U_0(t) \Vert_{ \overbar{\mathbb{H}}^2_3   }^2 < 50C\left(\Vert (u_0, \sigma_0) \Vert+1\right)^2 + 50C\left( \Vert (u_0, \sigma_0) \Vert+1 \right)^2 =\frac{1}{2}R^2.
\end{equation}
On the other hand, using the following estimate 
\begin{eqnarray*}
\|\sigma(x, t) -  \Sigma_0(x, t) \|_{\mathbb{H}^2_9}^2 &=& \left \|
e^{-\gamma t} F(x, t)\sigma_0(x)F^{\text{T}}(x, t)-\sigma_0(x) \right. \\
&& \left. + \delta \int_{0}^{t} e^{-\gamma (t-s)}F(x, t)F^{-1}(x, s)F^{-\text{T}}(x, s)F^{\text{T}}(x, t)ds
\right\Vert_{\mathbb{H}^2_9}^2 \\
&\leq& 2 \left\Vert e^{-\gamma t} F(x, t)\sigma_0(x)F^{\text{T}}(x, t)-\sigma_0(x)  \right\Vert_{\mathbb{H}^2_9}^2 \\
&& + 2 \left\Vert \delta \int_{0}^{t} e^{-\gamma (t-s)}F(x, t)F^{-1}(x, s)F^{-\text{T}}(x, s)F^{\text{T}}(x, t)ds \right\Vert_{\mathbb{H}^2_9}^2,  
\end{eqnarray*} 
and using \eqref{A}, by shrinking $T$ if necessary, we can make the above expression as small as we want. Furthermore, this choice of $T$ only depends on $r$. Thus there exists a $T_\sigma>0$ such that for all $0<T\leq T_\sigma$, we have
\begin{equation}\label{eq375}
\max_{0\leq t \leq T} \Vert \sigma(t) - \Sigma_0(t) \Vert_{\mathbb{H}^2_9}^2<\frac{1}{2}R^2.
\end{equation}
We can use this to obtain the following proposition: 

\begin{proposition}\label{SelfMapOfBall}
There exists a $T_{u, \sigma}>0$, depending only on the initial data $(u_0, \sigma_0)$ such that the map $\Phi$ maps the ball \eqref{Ball} of radius $R^2 = 100\tilde{C}\left( \|(u_0, \sigma_0)\|_Y+1 \right)^2$ in $X$ to itself for all $T\leq T_{u, \sigma}$.
\end{proposition}
\begin{proof} 
Let $T_{u, \sigma}=\min \{T_u, T_\sigma \} > 0$. By \eqref{eq371} and \eqref{eq375} for all $T\leq T_{u, \sigma}$ we have
\begin{equation}
\|\left( u-U_0, \sigma-\Sigma_0  \right)\|_X^2 = \max_{0\leq t \leq T} \left(\|u-U_0 \|_{\overbar{\mathbb{H}}^2_3}^2 + \|\sigma - \Sigma_0 \|_{\mathbb{H}^2_9}^2 \right) < \frac{1}{2}R^2 + \frac{1}{2}R^2=R^2
\end{equation}
and this establishes the proposition. 
\end{proof} 
Next we show that there is some time $0<T_c\leq T_{u, \sigma}$ such that for all $T\leq T_c$ the map $\Phi$ is a contraction on our ball. Suppose we have $(\tilde{u}_1, \tilde{\sigma}_1), (\tilde{u}_2, \tilde{\sigma}_2) \in B_R(U_0, \Sigma_0)$. Applying the map $\Phi$ to both of them we obtain two solutions, $\Phi(\tilde{u}_1, \tilde{\sigma}_1)=(u_1, \sigma_1), \quad \Phi(\tilde{u}_2, \tilde{\sigma}_2)=(u_2, \sigma_2) $. We need to show that there exists a $0<\theta<1$ such that
\begin{equation*}
\|(u_2-u_1,\sigma_1-\sigma_2) \|_X^2 \leq \theta^2 \| (\tilde{u}_2-\tilde{u}_1, \tilde{\sigma}_2-\tilde{\sigma}_1) \|_X^2
\end{equation*} 
Let us first examine the term corresponding to the velocity. We have
\begin{equation*}
\|u_2(t) - u_1(t)\|_{ \overbar{\mathbb{H}}^2_3}^2 \leq \tilde{C} \|\omega_2(t)-\omega_1(t)\|^2_{\mathbb{H}^1_3}, 
\end{equation*}
where the $\omega_i$ are the corresponding vorticities. As before, we define $\bar{\omega}(t)\coloneqq \omega_2(t)-\omega_1(t)$, both satisfy the vorticity equation with the condition $\bar{\omega}(0)=0$ and $\bar{f}=f_2-f_1$. Thus, using the first estimate in Theorem \ref{ImprovedRegularity}, we obtain the bound: 
\begin{equation*}
\max_{0\leq t \leq T} \|\bar{\omega}(t) \|_{ \mathbb{H}^1_3}^2 \leq \left( \frac{1}{\alpha} + e^{(1+2\alpha)T}  \right) \|\bar{f} \|_{L^2(0, T; \mathbb{L}^2_3  )}^2 \leq \left( \frac{1}{\alpha} + 2  \right) \|\bar{f}\|_{L^2(0, T; \mathbb{L}^2_3)}^2. 
\end{equation*}
since $T\leq T_{u, \sigma} \leq T'$.
Since
\begin{equation*}
\bar{f}= \beta \left(  \nabla \times \text{div} J \bar{\sigma}  \right) - (Ju_2\cdot \nabla)\bar{\omega} - (J\bar{u}\cdot \nabla) \omega_1 + \Omega(J\bar{u}, u_2) + \Omega(Ju_1, \bar{u}). 
\end{equation*}
and so we can estimate
\begin{equation*}
\left \|\bar{f}(t) \right \|^2_{ \mathbb{L}^2_3} \leq C(r) \left( \| \tilde{u}_2 - \tilde{u}_1 \|_{\overbar{\mathbb{H}}^2_3}^2 + \| \tilde{\sigma}_2- \tilde{\sigma}_1 \|_{\mathbb{H}^2_9}^2  \right) \leq C(r) \|(\tilde{u}_2-\tilde{u}_1, \tilde{\sigma}_2-\tilde{\sigma}_1) \|_X^2,    
\end{equation*}
where the constant $C(r)$ only depends on $r$ and not $T$. Therefore we obtain  
\begin{equation*}
\left \|\bar{f}(t) \right \|_{L^2(0, T; \mathbb{L}^2_3)}^2 \leq T C(r) \|(\tilde{u}_2-\tilde{u}_1, \tilde{\sigma}_2-\tilde{\sigma_1}) \|_X^2    
\end{equation*}
and so
\begin{equation}\label{Contraction1}
\max_{0 \leq t \leq T} \|u_2(t) - u_1(t)\|_{ \overbar{\mathbb{H}}^2_3  }^2 
\leq \tilde{C}T\left[\left( 2 + \frac{1}{\alpha} \right) \cdot C(r)   \right] \|(\tilde{u}_2-\tilde{u}_1, \tilde{\sigma}_2-\tilde{\sigma}_1)\|_X^2, 
\end{equation}
for $T\leq T_{u, \sigma}$. Therefore, if we choose any $\theta$ with $0<\theta<1$, then we can choose a sufficiently small $0<T_\theta \leq T_{u, \sigma}$ such that for all $0\leq t\leq T_{\theta}$ we have
\begin{equation}\label{ContractionU}
\max_{0\leq t \leq T_\theta}\Vert u_2(t)-u_1(t)\Vert_{\overbar{\mathbb{H}}^2_3}^2 \leq \frac{\theta^2}{2} \Vert (\tilde{u}_2-\tilde{u}_1, \tilde{\sigma}_2-\tilde{\sigma}_1)\Vert_X^2. 
\end{equation}
Next, we need to obtain a similar estimate for $\sigma_2-\sigma_1$.

\subsubsection{Estimates for $\Vert \sigma_1 - \sigma_2 \Vert $}
In this section, let $C([0, T]; \overbar{\mathbb{H}}^2_3)=Z$.
Take the two vector fields $\tilde{u}_1$ and $\tilde{u}_2$ which are in our ball and so both satisfy 
\begin{equation*}
\|\tilde{u}_1 \|_Z, \|\tilde{u}_2 \|_Z \leq r = \|(u_0, \sigma_0) \|_Y + R.
\end{equation*}
Now consider the corresponding vector fields $v_1(x, t)=J\tilde{u}_1 (x, t)$ and $v_2(x, t)=J\tilde{u}_2(x, t)$ which are smooth in $x$ and continuous in $t$. By \eqref{VBound} they satisfy
\begin{equation}\label{VBound2}
\max_{\substack{0\leq t \leq T \\ x\in \mathbb{R}^3}} \sum_i \sum_{|\lambda|\leq 4} |D^{\lambda} v_m^i(x, t)| \leq K \| \tilde{u}_m \|_Z \leq K r,   
\end{equation}
for $m=1,2$ as well as 
\begin{equation}\label{E3}
\max_{\substack{0\leq t \leq T \\ x\in \mathbb{R}^3}} \sum_i \sum_{|\lambda|\leq 4}|D^{\lambda }\left(v_2^i(x, t) -v_1^i(x, t) \right)| \leq K \|\tilde{u}_2- \tilde{u}_1 \|_Z.
\end{equation}

\begin{remark}
Notice that $Kr$ is a Lipschitz constant for the functions $D^{\lambda }v_m^i(x, t)$ for $|\lambda|\leq 3$. This will be useful later.
\end{remark}

We then use these two vector fields to obtain two flows $\varphi_1(x, t)$ ans $\varphi_2(x, t)$, which we in turn use to obtain two corresponding stress tensors $\sigma_1(x, t)$ and $\sigma_2(x, t)$ using the formula \eqref{sigma}. By the same formula, since we are interested in the quantity
\begin{equation*}
\|\sigma_2(x, t) - \sigma_1(x, t) \|_{\mathbb{H}^2_9} 
\end{equation*}
we will need to study the spatial derivatives of the corresponding deformation tensors, labeled $F_1$ and $F_2$, up to order 2, which corresponds to the derivatives of the the flow up to order 3. In what follows, we will use raised and lowered indices as well as the Einstein summation convention to simplify the presentation.

We begin by introducing the following notation:
\begin{eqnarray}\label{Notation}
&& \frac{\partial \varphi^i}{\partial x^k}(x, t) = a^i_k(x, t), \quad \frac{\partial^2 \varphi^i}{\partial x^j \partial x^k}(x, t) = b^i_{jk}(x, t), \quad  \frac{\partial^3 \varphi^i}{\partial x^l \partial x^j \partial x^k}(x, t)=c^i_{ljk}(x, t) \\
&& \frac{\partial v^i}{\partial x^k}(x, t)=A^i_k(x, t), \quad \frac{\partial v^i}{\partial x^j \partial x^k}(x, t)=B^i_{jk}(x, t), \quad \frac{\partial v^i}{\partial x^l \partial x^j \partial x^k}(x, t)=C^i_{ljk}(x, t).  
\end{eqnarray}
Then by successively differentiating the flow map equation
\begin{equation}\label{Flow0}
\left \{ \begin{array}{l}
\partial_t \varphi^i(x, t) = v^i(\varphi(x, t), t) \\
\varphi^i(x, 0) = x^i(x), 
\end{array} \right. 
\end{equation}
with respect to the spatial variables, the spatial partial derivatives can be shown to satisfy the following equations
\begin{subeqnarray}
\partial_t a^i_j(x, t) &=& A^i_k(\varphi(x, t), t)a^k_j(x, t) \slabel{D1} \\
\partial_t b^i_{lj}(x, t) &=& B^i_{mk}(\varphi(x, t), t)a^m_l(x, t)a^k_j(x, t)+A^i_k(\varphi(x, t), t) b^k_{lj}(x, t) \slabel{D2} \\
\partial_t c^i_{nlj}(x,t) &=& C^i_{bmk}(\varphi(x, t), t) a^b_n(x, t)a^m_l(x, t)a^k_j(x, t)+B^i_{mk}(\varphi(x, t), t)b^m_{nl}(x, t)a^k_j(x, t) \nonumber \\
&& +B^i_{mk}(\varphi(x, t), t)a^m_l(x, t)b^k_{nj}(x, t)+B^i_{bk}a^b_n(x, t)b^k_{lj}(x, t)+A^i_k(\varphi(x, t), t)c^k_{nlj} \slabel{D3}, 
\end{subeqnarray}
subject to initial conditions that $a^i_j(x, 0) = \delta^i_j, b^i_{lj}(x, 0) = 0,$ and 
$c^k_{nlj}(x, 0) = 0$. 

We will denote objects for our two vector fields by $(a_1)^i_j, (A_1)^i_k$, etc. We have the following proposition.

\begin{proposition}\label{Prop9}
There exists an ${\varepsilon}>0$, a corresponding $T_{{\varepsilon}}>0$, and a constant $K_0$, all depending only on $r$ such that for all $0\leq t \leq T_{{\varepsilon}}$ we have the estimate  
\begin{equation}\label{S2}
|\varphi_2(x, t) - \varphi_1(x, t)|  
\leq K_0 t^{1/2} \|\tilde{u}_2 - \tilde{u}_1\|_Z.
\end{equation} 
\end{proposition}
\begin{proof} 
As shown earlier, we have the bound: 
\begin{equation}\label{E4}
\max_{\substack{(x, t)\in \mathbb{R}^3 \times [0, T_{{\varepsilon}}] \\ |\lambda|\leq 3}} |D^\lambda (\varphi^i(x, t) - x^i(x)) | < \varepsilon,  
\end{equation} 
where $\lambda$ is a multi-index, and  $T_{{\varepsilon}}>0$ depends only on $\varepsilon$ and the $r$. By differentiating $|\varphi_2(x, t)-\varphi_1(x, t)|^2$ in time, we obtain
\begin{eqnarray*}
\frac{d}{dt}|\varphi_2(x, t)-\varphi_1(x, t)|^2 &=& 2(\varphi_2(x, t) - \varphi_1(x, t)) (\partial_t\varphi_2(x, t)-\partial_t\varphi_1(x, t))\\ 
&\leq& 2|\varphi_2(x, t)-\varphi_1(x, t)| |\partial_t\varphi_2(x, t)-\partial_t\varphi_1(x, t))| \\
&=& 2|\varphi_2(x, t)-\varphi_1(x, t))| |v_2(\varphi_2(x, t), t)-v_1(\varphi_1(x, t), t)|. 
\end{eqnarray*}
Using \eqref{E3}, we also, obtain that 
\begin{eqnarray*} 
|v_2(\varphi_2(x, t), t)-v_1(\varphi_1(x, t), t)| &=& |v_2(\varphi_2(x, t), t)-v_2(\varphi_1(x, t), t) + v_2(\varphi_1(x, t), t) - v_1(\varphi_1(x, t), t)| \\
&\leq& |v_2(\varphi_2(x, t), t)-v_2(\varphi_1(x, t), t)|+|v_2(\varphi_1(x, t), t)-v_1(\varphi_1(x, t), t)| \\
&\leq& Kr|\varphi_2(x, t)-\varphi_1(x, t)|+K\|\tilde{u}_2-\tilde{u}_1\|_Z. 
\end{eqnarray*}
Further, we have that by \eqref{E4}, 
\begin{equation*}
|\varphi_2(x, t)-\varphi_1(x, t)|=|\varphi_2(x, t)-x+x-\varphi_1(x, t)|\leq |\varphi_2(x, t)-x|+|x-\varphi_1(x, t)|\leq 3{\varepsilon} + 3{\varepsilon}. 
\end{equation*}
These result in
\begin{subeqnarray}
\frac{d}{dt}|\varphi_2(x, t)-\varphi_1(x, t)|^2 &\leq& 2Kr|\varphi_2(x, t)-\varphi_1(x, t)|^2+ 2K|\varphi_2(x, t)-\varphi_1(x, t)|\Vert \tilde{u}_2-\tilde{u}_1 \Vert_Z \slabel{E2} \\ 
\frac{d}{dt}|\varphi_2(x, t)-\varphi_1(x, t)|^2 &\leq& 2Kr|\varphi_2(x, t)-\varphi_1(x, t)|^2 + 12K{\varepsilon} \Vert \tilde{u}_2-\tilde{u}_1\Vert_Z \slabel{Key},  
\end{subeqnarray}
for $0\leq t \leq T_{{\varepsilon}}$. Gronwall's inequality and the fact that $\vert \varphi_2(x, 0)-\varphi_1(x, 0)\vert=0$, give 
\begin{equation*}
|\varphi_2(x, t)-\varphi_1(x, t)|^2\leq e^{2Krt} \left( 12K{\varepsilon} t \Vert \tilde{u}_2 - \tilde{u}_1 \Vert_Z \right)\leq e^{2KrT_{ {\varepsilon}   }} \left( 12K {\varepsilon} T_{{\varepsilon} } \Vert \tilde{u}_2 - \tilde{u}_1 \Vert_Z \right), 
\end{equation*}
for all $0\leq t \leq T_{ {\varepsilon}}$. By choosing a sufficiently small $ {\varepsilon}$, so that the following inequalities hold true  
\begin{equation}\label{Choices}
e^{2KrT_{{\varepsilon}}}<2, \quad 12K {\varepsilon} \sqrt{T_{{\varepsilon}}} < \frac{1}{2}, \quad 2K\sqrt{T_{{\varepsilon}}} < \frac{1}{2}, \quad \mbox{ and } \quad T_{{\varepsilon}} < 1, 
\end{equation}
we have that
\begin{equation*}
|\varphi_2(x, t)-\varphi_1(x, t)|^2 < \sqrt{T_{ {\varepsilon}}  }\Vert \tilde{u}_2-\tilde{u}_1 \Vert_Z.
\end{equation*}
In particular, since $T_{ {\varepsilon} }<1$, we also have
\begin{equation}\label{E1}
|\varphi_2(x, t)-\varphi_1(x, t)| < \Vert \tilde{u}_2-\tilde{u}_1 \Vert_Z^{1/2}.
\end{equation}
Now, the problem here is that the exponent $1/2$ here is not good enough. The trick is to take \eqref{E1} and plug it into \eqref{E2} to obtain
\begin{equation*}
\frac{d}{dt}|\varphi_2(x, t)-\varphi_1(x, t)|^2 \leq 2Kr|\varphi_2(x, t)-\varphi_1(x, t)|^2+ 2K\Vert \tilde{u}_2-\tilde{u}_1 \Vert_Z^{3/2}
\end{equation*}
and then apply Gronwall's inequality again to obtain
\begin{equation*} 
|\varphi_2(x, t)-\varphi_1(x, t)|^2 < e^{2KrT_{{\varepsilon}}}(2KT_{{\varepsilon}} \Vert \tilde{u}_2-\tilde{u}_1 \Vert_Z^{3/2})<\sqrt{T_{ {\varepsilon}} } \Vert \tilde{u}_2-\tilde{u}_1 \Vert_Z^{3/2} < \Vert \tilde{u}_2-\tilde{u}_1 \Vert_Z^{3/2}
\end{equation*}
where the inequalities follow from \eqref{Choices} and thus from our choice of ${\varepsilon}$.
Iterating this, we have that 
\begin{equation*} 
|\varphi_2(x, t)-\varphi_1(x, t)|^2 < \|\tilde{u}_2-\tilde{u}_1 \|_Z^{a_n}, 
\end{equation*}
where the sequence $a_n$ satisfies 
\begin{equation*}
a_0=1, \quad a_{n}=1+\frac{a_{n-1}}{2}.
\end{equation*}
It can be easily checked that $\lim_{n\rightarrow \infty} a_n=2$ and so we conclude
\begin{equation}\label{E5}
|\varphi_2(x, t)-\varphi_1(x, t)|^2 \leq \|\tilde{u}_2-\tilde{u}_1\|_Z^{2}.
\end{equation}
One final application of Gronwall's inequality using \eqref{E5} in \eqref{E2}, and using \eqref{Choices} then yields
\begin{eqnarray*}
|\varphi_2(x, t)-\varphi_1(x, t)|^2 \leq e^{2Krt} \left( 12K t \|\tilde{u}_2 - \tilde{u}_1 \|_Z^2 \right) \leq e^{2KrT_{ {\varepsilon}} } \left( 12K t \|\tilde{u}_2 - \tilde{u}_1\|_Z^2 \right) \leq 24K t \|\tilde{u}_2 - \tilde{u}_1 \|_Z^2, 
\end{eqnarray*}
for all $0\leq t\leq T_{{\varepsilon}}$. 
Taking square roots and setting $K_0=\sqrt{24K}$ yields \eqref{S2} and completes the proposition.
\end{proof} 

\begin{proposition}
There exists an $\varepsilon > 0$ such that 
\begin{subeqnarray*}
|(F_2)_{ij}(x, t)-(F_1)_{ij}(x, t)| &\leq& M t^{1/2} \|\tilde{u}_2 - \tilde{u}_1 \|_Z \\ 
|(F^{-1}_2)_{ij}(x, t)-(F^{-1}_1)_{ij}(x, t)| &\leq& Mt^{1/2} \|\tilde{u}_2 - \tilde{u}_1 \|_Z, 
\end{subeqnarray*}
for $0\leq t \leq T_\varepsilon$.
\end{proposition}
\begin{proof} 
Raising an index, we have that $F_{ij}=a^i_j$. To streamline notation, in what follows we will let $\varphi(x, t)=\varphi$ and suppress the explicit dependence on $x$ and $t$  where convenient. Since  
\begin{equation*}
|a_2 - a_1|^2 = \sum_{i,j} |(a_2)^i_j - (a_1)^i_j|^2, 
\end{equation*}
we have that 
\begin{eqnarray*}
\frac{d}{dt}|a_2-a_1|^2 &=& 2(a_2-a_1) \cdot (\partial_t a_2 - \partial_t a_1) \leq 2| a_2-a_1||\partial_t a_2 - \partial_t a_1| \\
&=& 2|a_2 - a_1| \left( \sum_{i,j} |\partial_t(a_2)^i_j - \partial_t(a_1)^i_j|^2   \right)^{1/2}. 
\end{eqnarray*}
By \eqref{D1} we now have
\begin{equation*}
|\partial_t(a_2)^i_j - \partial_t(a_1)^i_j | \leq \sum_k | (A_2(\varphi_2))^i_k(a_2)^k_j-(A_1(\varphi_1))^i_k(a_1)^k_j|. 
\end{equation*}
Now the trick is to write
\begin{eqnarray*}
(A_2(\varphi_2))^i_k(a_2)^k_j-(A_1(\varphi_1))^i_k(a_1)^k_j &=&
\left((A_2(\varphi_2))^i_k-(A_2(\varphi_1))^i_k\right)(a_2)^k_j + (A_2(\varphi_1))^i_k \left( (a_2)^k_j- (a_1)^k_j \right) \\
&& +\left( (A_2(\varphi_1))^i_k - (A_1(\varphi_1))^i_k  \right)(a_1)^k_j, 
\end{eqnarray*}
and so estimating each of these terms using Proposition \ref{Prop9}, \eqref{A}, and \eqref{E3}, we obtain
\begin{eqnarray*}
|\partial_t (a_2)_i^j - \partial_t (a_1)_i^j | &\leq& \sum_k \left( Kr\varepsilon \| \tilde{u}_2 - \tilde{u}_1 \|_Z + \varepsilon |(a_2)^k_j - (a_1)^k_j| + \varepsilon \| \tilde{u}_2 - \tilde{u}_1 \|_Z  \right) \\
&\leq& 3(Kr+1)\varepsilon \|\tilde{u}_2 - \tilde{u}_1 \|_Z + 3\varepsilon |a_2-a_1| \leq K_1 \|\tilde{u}_2 - \tilde{u}_1 \|_Z + 3\varepsilon |a_2-a_1|,   
\end{eqnarray*}
where for $\varepsilon < 1$ we can let $K_1=3(Kr+1)$. Thus we can estimate
\begin{equation}\label{Key2}
\frac{d}{dt}\vert a_2(x, t) - a_1(x, t) \vert^2 \leq 2 |a_2(x, t)-a_1(x, t)| \left(  3K_1 \Vert \tilde{u}_2 - \tilde{u}_1 \Vert_Z + 9\varepsilon \vert a_2(x, t)-a_1(x, t) \vert \right). 
\end{equation}
Now the idea is that \eqref{Key2} is of the same form as \eqref{Key} and $a_2(x, 0)=a_1(x, 0)$, and thus the argument of Proposition \ref{Prop9} (after possibly shrinking $\varepsilon$ and $T_\varepsilon$) carries through. Thus there exists a constant $K_2$ such that
\begin{equation}\label{S3}
|a_2(x, t) - a_1(x, t)|\leq K_2 t^{1/2} \|\tilde{u}_2 - \tilde{u}_1\|_Z
\end{equation}
We can apply similar arguments to $b_1, b_2$ and $c_1, c_2$ to obtain for some $K_3 > 0$, 
\begin{equation}\label{S4}
|b_2(x, t)-b_1(x, t)| \leq K_3 t^{1/2} \|\tilde{u}_2 - \tilde{u}_1\|_Z, 
\end{equation}
and for some $K_4 > 0$,
\begin{equation}\label{S5}
|c_2(x, t) - c_1(x, t)| \leq K_4 t^{1/2} \|\tilde{u}_2 - \tilde{u}_1 \|_Z, 
\end{equation}
for $0\leq t\leq T_\varepsilon$ after possibly shrinking $T_\varepsilon$. Putting together all of these estimates, we see there is a constant $K_5$, an $\varepsilon>0$ and a corresponding time $T_\varepsilon>0$ both depending only on $r$ such that
\begin{equation}\label{key}
|D^{\lambda}\left[ (F_2)_{ij}(x, t) - (F_1)_{ij}(x, t) \right]|\leq K_5t^{1/2} \|\tilde{u}_2-\tilde{u}_1\|_Z, 
\end{equation} 
for all $\lambda \leq 2$. Since we can write $F^{-1}$ in terms of the components of $F$, our previous estimates and \eqref{Key}, a tedious calculation yields the estimate: 
\begin{equation}\label{key3}
|D^{\lambda}\left[ (F^{-1}_2)_{ij}(x, t) - (F^{-1}_1)_{ij}(x, t)  \right]|\leq K_6t^{1/2} \|\tilde{u}_2-\tilde{u}_1\|_Z. 
\end{equation} 
Letting $M$ be the maximum of $K_5$ and $K_6$ yields the proposition. 
\end{proof} 

\begin{proposition}
There exists a constant $N$ such that 
\begin{equation*}
\|\sigma_2(x, t)-\sigma_1(x, t)\|_{\mathbb{H}^2_9} \leq Nt^{1/2} \| \tilde{u}_2-\tilde{u}_1\|_Z, 
\end{equation*}
for all $0\leq t \leq T_\varepsilon$.
\end{proposition}
\begin{proof} 
By the formula \eqref{SolSigma}, we have 
\begin{eqnarray*}
&& \sigma_2 (x, t) - \sigma_1(x, t) = e^{-\gamma t} F_2(x, t)\sigma_0(x)F_2^{\text{T}}(x, t)+ \delta \int_{0}^{t} e^{-\gamma (t-s)}F_2(x, t)F_2^{-1}(x, s)F_2^{-\text{T}}(x, s)F_2^{\text{T}}(x, t)ds  \\
&& \qquad -e^{-\gamma t} F_1(x, t)\sigma_0(x)F_1^{\text{T}}(x, t)+ \delta \int_{0}^{t} e^{-\gamma (t-s)}F_1(x, t)F_1^{-1}(x, s)F_1^{-\text{T}}(x, s)F_1^{\text{T}}(x, t)ds \\
&&= e^{-\gamma t} \left( F_2(x, t)\sigma_0(x)F_2^{\text{T}}(x, t) - F_1(x, t)\sigma_0(x)F_1^{\text{T}}(x, t)  \right) \\
&& \qquad + \delta \int_{0}^{t}  e^{-\gamma (t-s)}\left(F_2(x, t)F_2^{-1}(x, s)F_2^{-\text{T}}(x, s)F_2^{\text{T}}(x, t)- F_1(x, t)F_1^{-1}(x, s)F_1^{-\text{T}}(x, s)F_1^{\text{T}}(x, t) \right) ds.
\end{eqnarray*}
Let us suppress $x$ and $t$ to make the idea clearer. The trick is to express 
\begin{eqnarray*}
F_2\sigma_0 F_2^{\text{T}} - F_1\sigma_0 F_1^{\text{T}} &=& F_2\sigma_0 F_2^{\text{T}} -F_1\sigma_0 F_2^{\text{T}} + F_1\sigma_0 F_2^{\text{T}}  -F_1\sigma_0 F_1^{\text{T}} \\
&=&(F_2-F_1)\sigma_0F_2^{\text{T}} + F_1\sigma_0(F_2^{\text{T}}-F_1^{\text{T}}). 
\end{eqnarray*}
From \eqref{A} and Proposition \ref{Prop9} and $\Vert \sigma_0 \Vert_{\mathbb{H}^2_4} \leq r$, we have control over the second derivatives of all the terms and so it is then easy to see that there is a constant $K_7$ depending only on $r$ such that for all $0\leq t \leq T_\varepsilon$ we have the estimate
\begin{equation*}
\|F_2(x, t)\sigma_0(x) F_2^{\text{T}}(x, t)-F_1(x, t)\sigma_0(x) F_1^{\text{T}}(x, t) \|_{\mathbb{H}^2_4} \leq K_7 t^{1/2} \|\tilde{u}_2 - \tilde{u}_1\|_Z. 
\end{equation*}
We can also write the integral term as a difference of terms in the same way (this is a bit more tedious here because there are a lot more terms) and obtain a similar estimate. Putting these estimates together then yields the proposition for some appropriate constant $N$. 
\end{proof}  

Thus we can choose a sufficiently small $0<T'_{\theta}$ such that for $T\leq T_{\theta}'$, 
\begin{equation}\label{ContractionSigma}
\max_{0\leq t \leq T} \Vert \sigma_2(t)-\sigma_1(t) \Vert_{\mathbb{H}^2_9}^2 \leq \frac{\theta^2}{2}  \Vert (\tilde{u}_2-\tilde{u}_1, \tilde{\sigma}_2-\tilde{\sigma}_1)\Vert_X^2.
\end{equation}
With these estimates, we can finally prove the following.  
\begin{proposition}[Short Time Existence for  $\epsilon=0$]\label{ShortTimeExistence2}
For each initial data $(u_0, \sigma_0)\in Y$, there exists a unique local weak solution $(u, \sigma) \in X=C([0, T_c]; \overbar{\mathbb{H}}^2_4)\times C([0, T_c]; \mathbb{H}^2_9)$ with $u'(t)\in L^2(0, T_c; \overbar{\mathbb{H}}^1_2)$ and $\sigma' \in C([0, T_c]; \mathbb{H}^1_9)$ to the system \eqref{RegularizedVorticityForm} with $\epsilon=0$ for $0\leq t \leq T_c$ where $T_c>0$ depends only on the norm of the initial data.
\end{proposition}
\begin{proof} 
Let $T_c=\min \lbrace T_{u, \sigma}, T_\theta, T_\theta' \rbrace>0$. Then by Proposition \ref{SelfMapOfBall}, the map $\Phi$ maps the closed ball $B_R(U_0, \Sigma_0)$ to itself. Furthermore, by \eqref{ContractionU} and \eqref{ContractionSigma} we have
\begin{eqnarray*}
\|(u_1-u_2, \sigma_1-\sigma_2)\|_{X(T_c)}^2 &=& \max_{0\leq t \leq T_c} \left( \|u_1 - u_2 \|_{\mathbb{H}^2_3}^2 + \|\sigma_1(t) - \sigma_2(t) \|_{\mathbb{H}^2_9}^2 \right) \\
&\leq& \theta^2 \|(\tilde{u}_1 - \tilde{u}_2, \tilde{\sigma}_1-\tilde{\sigma}_2),  \|_{X(T_c)}^2. 
\end{eqnarray*} 
and so $\Phi$ is a contraction mapping on the ball. Therefore it has a unique fixed point satisfying
\begin{equation*}
\Phi(u,\sigma)=(u,\sigma), 
\end{equation*}
which is therefore the unique solution of the system . By \S \ref{VelocityFromVorticity}, $u'(t)\in L^2(0, T_c; \overbar{\mathbb{H}}^1_2)$. As for $\sigma'$, we can write
\begin{equation*}
\partial_t \sigma  = \delta I - \left(J{u} \cdot \nabla \right) \sigma  + \left(\nabla J{u}\right) \sigma +\sigma \left(\nabla J{u}\right)^T - \gamma \sigma,
\end{equation*}
from which we can conclude that $\sigma' \in C([0, T_c]; \mathbb{H}^1_9)$. The proof of the uniqueness of the solution is the same as in Proposition \ref{ShortTimeExistence}. By the way we defined $R$ and $r$, these only depend on $\Vert (u_0, \sigma_0) \Vert_Y$. Since $T_c$ only depends on $r$, it only depends on the norm of the initial data. This completes the proof. 
\end{proof} 

We can now prove global existence for the non-diffusive model, again using the blow up criterion and the energy estimates.

\begin{proof}[\textbf{Proof of Theorem \ref{GlobalExistenceTheorem2}}] 
Since we have control over the $L^2$ norm of $u$, the definition of $J$ and the above theorem implies that
\begin{equation}\label{VBound3}
\max_{x\in \mathbb{R}^2} \sum_i \sum_{|\lambda|\leq 4}|D^{\lambda }Ju^i(x, t)| \leq \bar{K}(t),
\end{equation}
where $\bar{K}(t) = K(E_1+E_2t)$. 

We could obtain an $L^2$ bound for $\sigma$ as before, by taking the equation for $\sigma$ and taking the Frobenius product with $\sigma$, integrating over $\mathbb{T}$, and using Gronwall's inequality. However, this will not help us getting an $H^2$ estimate for our components. Instead, using the regularization and energy estimate, we can actually obtain pointwise bounds for the conformation tensor and its derivatives up to order two, which of course gives us $H^2$ bounds for the components. Recall our previous notation \eqref{Notation}, where $F_{ij}=a^i_j$. As before, we let
\begin{equation*}
|a|^2=\sum_{i,j} |a^i_j|^2.
\end{equation*} 
Differentiating with respect to time and using \eqref{D1} and \eqref{VBound3} to estimate $|A^i_k|\leq \bar{K}(t)$, we obtain
\begin{subeqnarray}\label{a}
\frac{d}{dt} |a|^2 &=& 2a\cdot \partial_ta \leq 2|a||\partial_ta|=2|a|\left( \sum_{i,j} |A^i_ka^k_j|^2 \right)^{1/2} \leq 2|a|\left( \sum_{i,j} |A^i_k|^2 |a^k_j|^2 \right)^{1/2} \\
&\leq& 2|a| \left( 27K(E_1+E_2t) |a|  \right)=54K(E_1+E_2t)|a|^2. 
\end{subeqnarray} 
Also, since $a^i_j(0) = \delta^i_j$ we have $|a(0)|^2= 3 $,  we can apply Gronwall's inequality to obtain
\begin{equation}\label{abound}
|a(x, t)|^2\leq n(t), 
\end{equation}
where $n(t) = 3e^{54\int_{0}^{t}\bar{K}(s)ds}$. Note that $n(t)$ is finite with finite integral for all finite $t$. We perform a similar procedure for $\partial_{j} F_{ik}=b^i_{jk}$. Defining
\begin{equation*}
|b|^2=\sum_{i,j,k}|b^i_{jk}|^2, 
\end{equation*} 
and using \eqref{VBound3} and \eqref{abound}, we obtain that  
\begin{subeqnarray}\label{b}
\frac{d}{dt} |b|^2 &\leq& 2|b||\partial_t b| = 2|b| \left( \sum_{i,l,j} |\partial_tb^i_{lj} |^2 \right)^{1/2} = 2|b| \left(  \sum_{i,l,j} |B^i_{mk}a^m_l a^k_j+A^i_k b^k_{lj}|^2 \right)^{1/2} \\
&\leq& 2|b| \left( \left( \sum_{i,l,j} |B^i_{mk}|^2|a^m_l|^2|a^k_j|^2  \right)^{1/2} + \left( \sum_{i,l,j} |A^i_k|^2|b^k_lj|^2 \right)^{1/2}  \right)\\
&\leq& 2|b| \left( 243\bar{K}(t)n(t) + 243\bar{K}(t)|b|\right) = 486\bar{K}(t)n(t)|b| + 486\bar{K}(t)|b|^2 \\
&\leq& 486\bar{K}(t)n(t)\left( 1+|b|^2 \right) + 486\bar{K}(t)|b|^2 = 486\bar{K}(t)n(t)+486\bar{K}(t)\left( n(t) +1  \right) |b|^2. 
\end{subeqnarray}
Noting that $|b(x,0)|^2=0$, by applying Gronwall's inequality, we obtain that 
\begin{equation}\label{bbound}
|b(x, t)|^2 \leq m(t), 
\end{equation} 
where 
\begin{equation*}
m(t) = e^{\int_{0}^{t} 486\bar{K}(s)(n(s)+1) \, ds}\left( \int_{0}^{t} 486\bar{K}(s)n(s) \, ds  \right), 
\end{equation*}
which is again finite and with finite integral for all finite $t$. Finally we define
\begin{equation*} 
|c|^2=\sum_{i,j,k,l} |c^i_{jkl}|^2. 
\end{equation*}
Following similar steps above, we can obtain
\begin{equation}\label{c}
\frac{d}{dt}|c|^2 \leq \psi(t) + \phi(t)|c|^2, 
\end{equation}
where 
\begin{subeqnarray*} 
\psi(t) &=& \left(4374\bar{K}(t)n(t)^3+6561\bar{K}(t)m(t)n(t)  \right) \\
\phi(t) &=& \left( 4374 \bar{K}(t)n(t)^3+6561\bar{K}(t)m(t)n(t)+486\bar{K}(t) \right). 
\end{subeqnarray*}
Upon noticing $|c(x, 0)|^2 = 0$ and using Gronwall's inequality, we obtaint that 
\begin{equation} \label{cbound}
|c(x, t)|^2 \leq h(t) \quad \mbox{ with } \quad h(t) = e^{\int_{0}^{t}\phi(s)ds}\left( \int_{0}^{t} \psi(s)ds \right), 
\end{equation}
which is again finite with finite integral for all finite $t$. We now observe that the incompressibility condition
\begin{equation*}
\nabla \cdot Ju(x, t)=0
\end{equation*}
can be expressed as 
\begin{equation*}
\det(F) = 1. 	
\end{equation*}
Consequently, we can express $F^{-1}$ in terms of the components of $F$. Now, the functions $n(t), m(t),$ and $h(t)$ are all smooth, monotonically increasing functions of $t$ which are finite for all finite $0\leq t<\infty$. 

With \eqref{abound}, \eqref{bbound}, and \eqref{cbound}, we have bounds for the derivatives of $F$ up to order $2$. Applying this to \eqref{SolSigma}, we can find a smooth function $p(t)$ depending only on $\bar{K}(t)$ and $\|\sigma_0 \|_{\mathbb{H}^2_4}$, such that $p(t)<\infty$ for all $0\leq t<\infty$ and
\begin{equation}\label{p}
\|\sigma(t)\|_{\mathbb{H}^2_9}^2 \leq p(t).
\end{equation}
for all $0 \leq t < \infty$. Therefore, the norm of the conformation tensor remains finite for all finite $t$.

Next, we examine $\|\omega(t) \|_{  \mathbb{H}^1_3 }$. We can repeat the analysis of Proposition \ref{OmegaL2Estimate} and Proposition \ref{NablaOmegaL2Estimate} with $p(t)$ replacing $R_1(t)$ to conclude that there is some function, denoted by $q(t)$, such that it is continuous for $0\leq t < \infty$ and 
\begin{equation*}
\|\omega(t)\|_{\mathbb{H}^1_3 }^2 \leq q(t), 
\end{equation*}
and so using
\begin{equation*} 
\|u(t)\|_{\overbar{\mathbb{H}}^2_3 }^2 \leq \tilde{C} \|\omega(t)\|_{\mathbb{H}^1_3 }^2, 
\end{equation*}
we obtain that 
\begin{equation*}
\|u(t)\|_{\mathbb{H}^2_3 }^2 \leq \tilde{C}q(t), 
\end{equation*}
and so 
\begin{equation}\label{Pt}
\|\sigma(t)\|_{\mathbb{H}^2_9}^2 + \| u(t) \|_{  \mathbb{H}^2_3 }^2 \leq p(t) + \tilde{C}q(t) = P(t), 
\end{equation}
is finite for all finite $t$. Using Proposition \ref{BlowUpCriterion}, which is still applicable in this case, we complete the proof. 
\end{proof} 

As before, we can obtain better regularity.

\begin{proof}[\textbf{Proof of Theorem \ref{GlobalExistenceTheorem2HigherRegularity2}}]
By the formula \eqref{SolSigma}, we see that if \eqref{InitialConditionsHk} holds, then  $\sigma \in C([0, T]; \mathbb{H}^{k}_{9})$, and $\sigma'\in L^2(0, T; \mathbb{H}^{k-1}_9)$. As mentioned, we can treat the vorticity equation as we did the $\sigma$ equation in Theorem \ref{GlobalExistenceTheorem1HigherRegularity} to obtain $u\in C([0, T]; \overbar{\mathbb{H}}^k_3)$,  $u'\in L^2(0, T; \overbar{\mathbb{H}}^{k-1}_3)$, and again for $k\geq 4$ the solutions are classical by the Sobolev embedding theorem.
\end{proof}

\subsection{Global Stability of Solutions for $\epsilon=0$}

Proving the global stability of the non-diffusive solutions in the $H^2$ norm is difficult. However, we are able to obtain stability in the $L^2$ norm. 

\begin{theorem}[$L^2$ stability for $\epsilon =0$] \label{L2StabilityNonDiffusive}
Let $(u_{0, 1}, \sigma_{0,1})$ and $(u_{0, 2}, \sigma_{0,2})$ denote two sets of initial data which satisfy the assumptions of Theorem \ref{GlobalExistenceTheorem2}. We denote the corresponding solutions to \eqref{RegularizedVorticityForm} by $(u_1(t),\sigma_1(t))$ and $(u_2(t),\sigma_2(t))$, respectively. For a fixed $T > 0$, there exists a continuous function $R(t)$ such that 
\begin{equation*}
\|u_2(t)-u_1(t)\|_{\overbar{\mathbb{L}}^2_3}^2 + \|\sigma_2(t) -  \sigma_1(t) \|_{\mathbb{L}^2_9}^2 \leq R(t) \left( \|u_{0, 2}- u_{0, 1} \|_{\overbar{\mathbb{L}}^2_3}^2 + \|\sigma_{0, 2} - \sigma_{0, 1} \|_{\mathbb{L}^2_9}^2 \right).
\end{equation*}
\end{theorem}

\begin{proof}
The proof is the same as that of Proposition \ref{L2Stability}. As can be easily seen, it holds in the case $\epsilon=0$.
\end{proof}

\section{$L^2-$convergence of solution of diffusive model to that of non-diffusive model}\label{SEC:L2ConvergenceSection}

Now we show that starting with the same initial data, the diffusive solutions converge to the non-diffusive solution, at least in the $L^2$ sense. Fix some initial data $(u_0, \sigma_0)$ satisfying the assumptions of Theorem \ref{GlobalExistenceTheorem1} or \ref{GlobalExistenceTheorem2}. As we saw, we can obtain solutions for any $\epsilon \geq 0$, with components in $H^2$. Denote the solutions by $(u_\epsilon, \sigma_\epsilon)$ for $\epsilon>0$ and $(u, \sigma)$ for $\epsilon=0$. It is natural to expect that 
\begin{equation}
    (u_\epsilon, \sigma_\epsilon) \rightarrow (u, \sigma) \quad \text{as} \quad \epsilon \rightarrow 0 
\end{equation}
where the components converge in $H^2$. We expect this to be true, though it is difficult to prove. For now, we establish convergence in $L^2$ which is already an interesting result.

\begin{theorem}
Fix some initial data $(u_0, \sigma_0) \in Y$ with $\sigma_0$ symmetric and positive semi-definite. For each $\epsilon>0$ let $(u_\epsilon, \sigma_\epsilon)$ denote the solution of Theorem \ref{GlobalExistenceTheorem1}. Let $(u, \sigma)$ denote the solution of Theorem \ref{GlobalExistenceTheorem2} for $\epsilon=0$.  Then for any $0\leq t < \infty$ we have
\begin{equation*}
\lim_{\epsilon \rightarrow 0} \left(\|u_\epsilon (t) - u (t)\|_{\mathbb{L}^2_3}^2 + \|\sigma_\epsilon (t) - \sigma (t)\|_{\mathbb{L}^2_9}^2 \right) = 0
\end{equation*}
and moreover for any finite $T>0$ we have that the convergence is uniform on $0\leq t \leq T$. 
\end{theorem}
\begin{proof}
As usual, we define $\bar{u} = u_\epsilon - u$ and $\bar{\sigma}=\sigma_\epsilon-\sigma$ and then subtract the corresponding equations. Then for the velocity we obtain
\begin{equation*}
 \partial_t \bar{u} - \alpha \Delta \bar{u} = (-Ju\cdot \nabla ) \bar{u} - (J\bar{u} \cdot \nabla)u_\epsilon + \beta \text{div} J \bar{\sigma} - \beta \nabla \bar{p}  
\end{equation*}
and upon multiplying by $\bar{u}$ and integrating over $\mathbb{T}$ and making some basic estimates we obtain
\begin{equation*}
\frac{1}{2} \|\bar{u} \|_{\mathbb{L}^2_3}^2  \leq \frac{\beta}{2} K^2 |\mathbb{T}| \|\bar{\sigma} \|_{\mathbb{L}^2_3}^2 + \frac{\beta}{2} \|\bar{u} \|_{\mathbb{L}^2_3}^2 + K \|\nabla u_\epsilon \|_{\mathbb{L}^2_9} \|\bar{u}\|_{\mathbb{L}^2_3}^2.
\end{equation*}
Similarly we obtain the equation for $\bar{\sigma}$
\begin{equation*} 
\partial_t \bar{\sigma} - \epsilon \Delta \bar{\sigma} = -(Ju_\epsilon \cdot \nabla ) \bar{\sigma} +(J\bar{u} \cdot \nabla) \sigma + (\nabla Ju_\epsilon)\bar{\sigma} + (\nabla J\bar{u})\sigma + \bar{\sigma} ( \nabla Ju_\epsilon)^T + \sigma (\nabla J\bar{u})^T - \gamma \bar{\sigma} + \epsilon \Delta \sigma
\end{equation*}
and upon taking the Frobenius product with $\bar{\sigma}$, integrating over $\mathbb{T}$ and making some basic estimates we obtain
\begin{subeqnarray*}
\frac{1}{2} \frac{d}{dt}\Vert \bar{\sigma} \|_{\mathbb{L}^2_9}^2 &\leq&  K \|\nabla \sigma \|_{\mathbb{L}^2_9} \left(  \| \bar{u} \|_{\mathbb{L}^2_3} \| \bar{\sigma} \|_{\mathbb{L}^2_9}   \right) + 2K(E_1 t + E_2 ) \| \bar{\sigma} \|_{\mathbb{L}^2_3}^2
\\ && \quad + 2K \|\sigma \|_{\mathbb{L}^2_9}  \left( \| \bar{u} \|_{\mathbb{L}^2_3}  \| \bar{\sigma} \|_{\mathbb{L}^2_9} \right) + \gamma \| \bar{\sigma} \|_{\mathbb{L}^2_9}^2 + \epsilon \| \Delta \sigma \|_{\mathbb{L}^2_9} \| \bar{\sigma} \|_{\mathbb{L}^2_9} \\ &\leq& K \| \nabla \sigma \|_{\mathbb{L}^2_9} \left(  \frac{1}{2}\| \bar{u} \|_{\mathbb{L}^2_3}^2 + \frac{1}{2} \| \bar{\sigma} \|_{\mathbb{L}^2_9}^2    \right) + 2K(E_1 t + E_2) \| \bar{\sigma} \|_{\mathbb{L}^2_3}^2
    \\ && \quad + 2K \| \sigma \|_{\mathbb{L}^2_9}  \left(  \frac{1}{2}\| \bar{u} \|_{\mathbb{L}^2_3}^2 + \frac{1}{2} \| \bar{\sigma} \|_{\mathbb{L}^2_9}^2    \right) + \gamma \| \bar{\sigma} \|_{\mathbb{L}^2_9}^2 + \epsilon \left( \frac{1}{2}\| \Delta \sigma \|_{\mathbb{L}^2_9}^2 + \frac{1}{2} \| \bar{\sigma} \|_{\mathbb{L}^2_9}^2 \right).
\end{subeqnarray*}

Now fix an arbitrary $T>0$. Then for all $0\leq t \leq T$ the terms $\| \nabla \sigma \|_{\mathbb{L}^2_{27}}^2$ and $\| \sigma \|_{\mathbb{L}^2_9}^2$ and $\| \Delta \sigma \|_{\mathbb{L}^2_9}^2$ can all be estimated in terms of $P(T)$ where $P(t)$ is given in \eqref{Pt} and in the end only depends on the initial data.  Also estimating $\| \nabla u_\epsilon \|_{\mathbb{L}^2_3} \leq 1 + \| \nabla u_\epsilon \|_{\mathbb{L}^2_3}^2 $, putting together the two estimates and multiplying by $2$ we can get an estimate of the form
\begin{equation*}
\frac{d}{dt} \left( \| \bar{u} \|_{\mathbb{L}^2_3}^2 + \| \bar{\sigma} \|_{\mathbb{L}^2_9}^2  \right) \leq \left(  C_1  +  \| \nabla u_\epsilon \|_{\mathbb{L}^2_3}^2     \right) \left( \| \bar{u} \|_{\mathbb{L}^2_3}^2 + \| \bar{\sigma} \|_{\mathbb{L}^2_9}^2    \right)  + \epsilon P(T)
\end{equation*}
for $0\leq t \leq T$ where the constant $C_1>0$ ultimately only depends on $T$ and the initial data. Applying Gronwall's inequality, the fact that $\bar{u}(0)=0$ and $\bar{\sigma}(0)=0$, and the energy estimate we obtain
\begin{equation*}
\|\bar{u}(t) \|_{\mathbb{L}^2_3}^2 + \| \bar{\sigma}(t) \|_{\mathbb{L}^2_9}^2 \leq \epsilon t P(T) e^{C_1 t + \int_0^t \| \nabla u_\epsilon (s) \|_{\mathbb{L}^2_3}^2 ds} \leq \epsilon T P(T)
 e^{C_1 T + E_1 + E_2 T}
\end{equation*}
for $0\leq t \leq T$. Taking the limit $\epsilon \rightarrow 0$ we obtain the uniform convergence we desire. \end{proof}

We mention that this is precisely what is observed in numerical simulations. In the numerical setting, due to the finitiness of the discrete mesh, all norms can be shown to be equivalent. Thus the analytical $L^2$ convergence implies numerical convergence. We hope in a future paper to establish the full $H^2$ convergence of the solutions of our model in the limit as the diffusivity of the stress goes to $0$.

\section{Conclusion}\label{con}
In this paper, we have established the global existence and regularity of the solutions to the system \eqref{RegularizedVorticityForm} in both the diffusive and non-diffusive cases. We also established the stability of the solutions and the convergence of the diffusive solutions to the non-diffusive one, in some appropriate sense. We remark that a similar result was established in the discrete case for $\epsilon = 0$ in \cite{lee2011global}. There, the main idea was to use sufficiently small time step size and large mesh size, to obtain control over $\nabla u$ in the equation for the conformation tensor. However, the resulting solutions do not converge as the mesh size is shrunk down to zero. Thus the regularizer in our modified model is the source of the pointwise control over the $\nabla Ju$ terms, and thus allows us to obtain our results. This strongly suggests that some kind of regularizer is needed, at least in the non-diffuisve case.

In a future paper, we will investigate the well-posedness of our model in the discrete setting and carry out numerical experiments. Furthermore, we will attempt to establish the $H^2$ convergence of the diffusive solutions to the non-diffusive one. Finally, we will see if somehow the regularizer can be removed to obtain global solutions to the original Oldroyd-B model.

\appendix

\section{Derivation of Vorticity Formulations}\label{VorticityDerivation}
When taking the curl of the velocity equation, the following identities  
are useful
\begin{align} \label{VectorIdentities}
    \begin{split} 
    &\omega= \nabla \times u \\
    &(u\cdot \nabla ) u =\omega \times u + \nabla \left( \frac{1}{2} |u|^2   \right) \\
    &\nabla \times (\omega \times u ) = (u \cdot \nabla )\omega -(\omega \cdot \nabla)u + \omega (\nabla \cdot u ) - u (\nabla \cdot \omega) \\
    &\nabla \cdot (\nabla \times u ) =\nabla \cdot \omega = 0 \\
    &\nabla \times (\nabla f ) = 0
    \end{split}
\end{align}
where $f$ is any function.

Using these identities along with the assumption of incompressibility
\begin{equation}
    \nabla \cdot u =0 
\end{equation}
we obtain
\begin{equation}
    \nabla \times \left(  (u \cdot \nabla) u \right) = (u\cdot \nabla) \omega - (\omega \cdot \nabla ) u 
\end{equation}
and so, taking the curl of both sides of the velocity equation we have
\begin{equation}
  \nabla \times \left(  \partial_t u - \frac{\nu}{R_e} \Delta u    + u \cdot \nabla u \right)   = \nabla \times \left( \frac{1}{R_e}\text{div} \sigma -\frac{1}{R_e}\nabla p+ \frac{1}{R_e}f \right)
\end{equation}
which gives the equation for the vorticity vector $\omega$
\begin{equation}
    \partial_t \omega - \frac{\nu}{R_e} \Delta \omega + (u \cdot \nabla ) \omega - (\omega \cdot \nabla) u = \frac{1}{R_e} \nabla \times \text{div} \sigma + g
\end{equation}
where 
\begin{equation}
g=\nabla \times \left(  \frac{1}{R_e} f \right).
\end{equation}

More genrally we have the identity, 
\begin{equation}
(v \cdot \nabla ) u = \frac{1}{2} \left( \nabla (v \cdot u) - \nabla \times (v \times u) - u \times (\nabla \times v) - v \times (\nabla \times u) - u (\nabla \cdot v) + v (\nabla \cdot u) \right ).
\end{equation}
Asuming $u$ and $v$ are incompressible, this leads to (with $w_u = \nabla \times u$ and $w_v = \nabla \times v$), 
\begin{align} \begin{split}
\nabla \times  ((v \cdot \nabla) u) &= 
\frac{1}{2} \nabla \times \left ( - \nabla \times (v \times u) - u \times (\nabla \times v) - v \times w \right ) \\ 
&= \frac{1}{2} \nabla \times \left ( - \nabla \times (v \times u) - u \times w_v - v \times w_u \right ) \\
&= \frac{1}{2} \left ( - \nabla \times \nabla \times (v \times u) + \nabla \times (w_v \times u) + \nabla \times (w_u \times v) \right ) \\
&= \frac{1}{2} \left ( - \nabla \times \nabla \times (v \times u) + (u \cdot \nabla ) w_v - (w_v \cdot \nabla) u + (v \cdot \nabla) w_u - (w_u \cdot \nabla)v \right ). \end{split}
\end{align}
Another way of expressing this term is
\begin{equation} \label{A9}
    \nabla \times ( (v \cdot \nabla) u)= (v \cdot \nabla )w_u + \Omega(v, u)
\end{equation}
where
\begin{small}
\begin{equation} \label{A10ExpressionForOmega}
    \Omega(v, u) =\left( \begin{array}{c}
(\partial_2 v^1)(\partial_1 u^3)+(\partial_2 v^2)(\partial_2 u^3)+(\partial_2 v^3 )(\partial_3 u^3) -(\partial_3 v^1)(\partial_1 u^2)-(\partial_3 v^2)(\partial_2 u^2)-(\partial_3 v^3)(\partial_3 u^2)\\
(\partial_3 v^1)(\partial_1 u^1)+(\partial_3 v^2)(\partial_2 u^1)+(\partial_3 v^3)(\partial_3 u^1)-(\partial_1 v^1)( \partial_1 u^3)-(\partial_1 v^2)(\partial_2 u^3)-(\partial_1 v^3)(\partial_2 u^3)
\\(\partial_1 v^1)(\partial_1 u^2)+(\partial_1 v^2)(\partial_2 u^2)+(\partial_1 v^3)(\partial_3 u^2)-(\partial_2 v^1)(\partial_1 u^1)-(\partial_2 v^2)(\partial_2 u^1)-(\partial_2 v^3)(\partial_3 u^1)
\end{array} \right).
\end{equation}
\end{small}
We point out that this formula works whether or not the vector fields are incompressible. It is obtained by writing out everything in components, taking the curl, and collecting some terms. While not very pleasing to the eye, the functional $\Omega(v, u)$ is linear in each argument
\begin{align}
    \begin{split}
       & \Omega(c_1 v + c_2 \tilde{v}, u)=c_1 \Omega(v, u) + c_2 \Omega(\tilde{v}, u) \\
        &\Omega(v, c_1 u + c_2 \tilde{u})=c_1 \Omega(v, u) + c_2 \Omega(v, \tilde{u})
    \end{split}
\end{align}
which is a property that will be useful later on. 

Consequently, if $u = v$, then
\begin{eqnarray*}
\nabla \times  ((u \cdot \nabla) u) = (u \cdot \nabla ) w - (w \cdot \nabla) u. 
\end{eqnarray*}
as before. We will let $v=Ju$ in equation \eqref{A9} to obtain the regularized vorticity form of our equations.

We point out the vorticity equation simplifies in the case of a two dimensional flow. In that case, we consider the two dimensional flow as a a three dimensional flow where one of the componets is $0$ 
\begin{equation}
    u= \left(
\begin{array}{c}
u_1\\
u_2\\
0
\end{array}
\right)
\end{equation}
the force is $f=(f_1, f_2, 0)$ and the conformation tensor satisfies $\sigma_{ij}=0$ when $i \; \text{or} \; j =3 $
in which case the vorticity only has a single component
\begin{equation}
  \nabla \times  u= \left(
\begin{array}{c}
0\\
0\\
\partial_1 u_2 - \partial_2 u_1
\end{array}
\right)= \left(
\begin{array}{c}
0\\
0\\
w
\end{array}
\right)
\end{equation}
and so, since $(\omega \cdot \nabla) u =0$ in that case, we get that the equation for the vorticity vector reduces to the scalar equation for the single component $w$
\begin{equation}
    \partial_t w - \frac{\nu}{R_e} \Delta w + (u \cdot \nabla ) w = \frac{1}{R_e} \left(\nabla \times \text{div} \sigma \right)_3 + g_3
\end{equation}
which could be analyzed independently.

\bibliographystyle{model1-num-names}
\bibliography{ref_new}

\end{document}